\documentclass[12pt,english]{article}
\usepackage[T1]{fontenc}
\usepackage[latin9]{inputenc}
\usepackage{geometry}
\usepackage{babel}
\usepackage{textcomp}
\usepackage{url}
\usepackage{amsmath}
\usepackage{amssymb}
\usepackage{setspace}
\usepackage[authoryear]{natbib}
\usepackage[unicode=true,
 bookmarks=false,
 breaklinks=false,pdfborder={0 0 1},backref=section,colorlinks=false]
 {hyperref}
\hypersetup{
 colorlinks,citecolor=blue,urlcolor=blue}

\makeatletter

\providecommand{\tabularnewline}{\\}

\usepackage{babel}
\usepackage{babel}
\usepackage{babel}
\usepackage{amsthm}

\usepackage{tikz}
\usepackage[T1]{fontenc}
\usetikzlibrary{decorations.pathreplacing, calc,fit,shapes, positioning,arrows, patterns}

\providecommand{\tabularnewline}{\\}
\newtheorem{rem}{Remark}\newtheorem{thm}{Theorem}\newtheorem{prop}{Proposition}\newtheorem{cor}{Corollary}\newtheorem{example}{Example}\newtheorem{lem}{Lemma}\newtheorem{defn}{Definition}\usepackage{bbm}

\def\p{\vskip4truept \noindent}

\DeclareMathOperator{\E}{\mathbb{E}}
\DeclareMathOperator{\R}{\mathbb{R}}
\DeclareMathOperator{\N}{\mathbb{N}}
\DeclareMathOperator{\1}{\mathbbm{1}}

\DeclareMathOperator{\CU}{\mathcal{U}}
\DeclareMathOperator{\CW}{\mathcal{W}}
\DeclareMathOperator{\CV}{\mathcal{V}}
\DeclareMathOperator{\CG}{\mathcal{G}}
\DeclareMathOperator{\CQ}{\mathcal{Q}}

\DeclareMathOperator{\val}{val}
\DeclareMathOperator{\marg}{marg}
\def\conv{ {\rm conv} }

\newcommand{\dist}{\mathbbm{d}}
\newcommand{\Prob}{\mathbbm{P}}

\usepackage{babel}

\usepackage{babel}

\makeatother

\begin{document}

\title{Value-Based Distance Between Information Structures \thanks{We are grateful to Satoru Takahashi and Siyang Xiong for comments.
Our work has benefited from the AI Interdisciplinary Institute ANITI.
ANITI is funded by the French "Investing for the Future - PIA3"
program under the Grant agreement n°ANR-19-PI3A-0004. J. Renault
and F. Gensbittel gratefully acknowledge funding from the French National
Research Agency (ANR) under the Investments for the Future (Investissements
d'Avenir) program, grant ANR-17-EURE-0010. J. Renault gratefully acknowledges
funding from ANR MaSDOL. M. P\k{e}ski gratefully acknowledges financial
support from the Insight Grant of the Social Sciences and Humanities
Research Council of Canada and the hospitality of HEC Paris, where
parts of this research were completed.}}
\author{Fabien Gensbittel \thanks{Fabien Gensbittel: Toulouse School of Economics, University Toulouse
Capitole, \protect\url{fabien.gensbittel@tse-fr.eu}}, Marcin P\k{e}ski\thanks{M. P\k{e}ski: Department of Economics, University of Toronto, \protect\url{marcin.peski@utoronto.ca}},
Jérôme Renault\thanks{Jérôme Renault: Toulouse School of Economics, University Toulouse
Capitole, \protect\url{jerome.renault@tse-fr.eu}}}

\maketitle
\begin{abstract}
We define the distance between two information structures as the largest
possible difference in value across all zero-sum games. We provide
a tractable characterization of distance and use it to discuss the
relation between the value of information in games versus single-agent
problems, the value of additional information, informational substitutes,
complements, or joint information. The convergence to a countable
information structure under value-based distance is equivalent to
the weak convergence of belief hierarchies, implying, among other
things, that for zero-sum games, approximate knowledge is equivalent
to common knowledge. At the same time, the space of information structures
under the value-based distance is large: there exists a sequence of
information structures where players acquire increasingly more information,
and $\varepsilon>0$ such that any two elements of the sequence have
distance of at least $\varepsilon$. This result answers by the negative
the second (and last unsolved) of the three problems posed by J.F.
Mertens in his paper ``Repeated Games'', ICM 1986. 
\end{abstract}

\newpage{}

\section{Introduction}

The role of information is of fundamental importance for the economic
theory. It is well known that even small differences in information
may lead to significant differences in the behavior \citep{rubinstein_electronic_1989}.
A recent literature on strategic (dis)-continuities has studied these
differences intensively and in full generality. A typical approach
is to consider all possible information structures, modeled as elements
of an appropriately defined universal space of information structures,
and study the differences in the strategic behavior across all games.

A similar methodology has not been applied to study the relationship
between information and the agent's bottom line, their payoffs. There
are perhaps few reasons for this. First, following \citet{dekel_topologies_2006},
\citet{weinstein2007structure} and others, the literature has focused
on interim rationalizability as the solution concept. Compared with
the equilibrium, this choice has several advantages: it is easier
to analyze, it is more robust from a decision-theoretic perspective,
it can be factorized through the Mertens-Zamir hierarchies of beliefs
(\citet{dekel_topologies_2006}, \citet{ElyPeski06}), and it does
not suffer from existence problems (unlike the equilibrium - see \citet{simon_games_2003}).
However, the value of information is typically measured in the ex-ante
sense, where solution concepts like Bayesian Nash equilibrium are
more appropriate. Also, the multiplicity of solutions necessitates
that the literature take a set-based approach. This, of course, makes
a quantitative comparison of the value of information difficult. Last
but not least, the freedom to choose games without any restrictions
makes the equilibrium payoff comparison between information structures
trivial (see Section \ref{sec:Payoff distance NZS} for a detailed
discussion of this point).

Despite the challenges, we find the questions concerning the strategic
value of information to be important and fascinating. How can we measure
the value of information on the universal type space? How much can
a player gain (or lose) from additional information? Which information
structures are similar in the sense that they always lead to the same
payoffs? In order to address these questions, and given the last point
in the previous paragraph, we must restrict the analysis to a class
of games. We propose to focus on zero-sum games. We do so for both
substantive and pragmatic reasons. First, the question of the value
of information is of special importance when players' interests are
opposing. With zero-sum games, the information has natural comparative
statics: a player is better off when her information improves and/or
the opponent's information worsens (\citet{peski_comparison_2008}).
Such comparative statics are intuitive, and although they hold in
single-agent decision problems (\citet{blackwell1953}), they do not
hold for general non-zero-sum games, where better information may
worsen a player's strategic position, and where players may have incentive
to engage in pre-game communication to manipulate information available
to others. Second, many of the constructions in the strategic discontinuities
literature rely on special classes of games, like coordination games,
or betting games (\citet{rubinstein_electronic_1989}, \citet{morris_typical_2002},
\citet{ely_critical_2011}, \citet{chen_e-mail_2013} among others).
This begs the question whether some of the surprising phenomena, like
the difference between approximate knowledge and common knowledge,
apply in other classes of games. Our restriction allows for the clarification
of this issue for zero-sum games.

On the other hand, the restriction avoids all the problems mentioned
above. Finite zero-sum games always have an equilibrium on common
prior information structures (\citet{mertens_sorin_zamir_2015}) that
depends only on the distribution over hierarchies of beliefs. The
equilibrium has decent decision-theoretic foundations (\citet{brandt2019justifying}),
and, even if it is not unique, the ex-ante payoff always is unique
and equal to the value of the zero-sum game. Finally, as we demonstrate
through numerous results and examples, 
 the restriction uncovers a rich internal structure of the
universal type space.

We define the distance between two common prior information structures
as the largest possible difference in value across all zero-sum payoff
functions that are bounded by a constant. This has a straightforward
interpretation as a tight upper bound on the gain or loss moving from
one information structure to another. Our first result provides a
characterization of the distance in terms of total variation distance
between sets of information structures. This distance can be computed
as a solution to a convex optimization problem.

The characterization is tractable in applications. In particular,
we use it to describe the conditions under which the distance between
information structures is maximized in single-agent problems (which
are a subclass of zero-sum games). We provide bounds to measure the
impact of the marginal distribution over the state. We also use it
in a series of results on the comparison of the value of information.
A tight upper bound on the value of an additional piece of information
is defined as the distance between two type spaces, in one of which
one or two players have access to new information. We give conditions
when the value of new information is maximized in single-agent problems.
We describe the situations in which the value of one piece of information
decreases when another piece of information becomes available, i.e.
the opposing players' pieces of information are substitutes. Similarly,
we show that, under some conditions, the value of one piece of information
increases when the other player receives an additional piece of information,
i.e. the opposing players' pieces of information are complements.\footnote{\citet{hellwig2009knowing} studies the complementarity and substitutability
of information on acquisition decisions in a beauty contest game.} Finally, we show that the new information matters only if it is valuable
to at least one of the players individually. The joint information
contained in the correlation between players' signals is not valuable
in the zero-sum games.

The second main result shows that the space of information structures
is large under value-based distance: there exists an infinite sequence
of information structures $u^{n}$ and $\varepsilon>0$ such that
the value-based distance between each pair of structures is at least
$\varepsilon$. In particular, it is not possible to approximate the
set of information structures with finitely many well-chosen information
structures. In the proof, we construct a Markov chain wherein the
first element of the chain is correlated with the state of the world.
We construct an information structure $u^{n}$ so that one player
observes the first $n$ odd elements of the sequence and the other
player observes the first $n$ even elements. Our construction implies
that in information structure $u^{n+1}$, each player gets an extra
signal. Thus, having more and more information may lead... nowhere.
This is unlike the single-player case, where more signals corresponds
to a martingale and the values converge uniformly over bounded decision
problems.

The Markov construction implies that all the information structures
$n^{\prime}\geq n$ have the same $n$-th order belief hierarchies
(\citet{mertzam:85}). As a consequence, our distance is not robust
with respect to the product convergence of belief hierarchies. This
observation may sound familiar to a reader of the strategic (dis)continuities
literature. However, we emphasize that the proof of our result is
entirely novel. Among other reasons, many earlier constructions heavily rely on
 coordination games (\citet{rubinstein_electronic_1989}, \citet{morris_typical_2002},
\citet{ely_critical_2011}, \citet{chen_e-mail_2013} among others).
Such constructions cannot be done with zero-sum games. 

More importantly, there are significant differences between strategic
topologies and the topology induced by value-based distance. For instance,
the type spaces from the famous email game example of \citet{rubinstein_electronic_1989},
or any approximate knowledge spaces, converge to the common knowledge
of the state for value-based distance. More generally, we show that
any sequence of countable information structures converges to a countable
structure under value-based distance if and only if the associated
hierarchies of beliefs converge in the product topology. The impact
of higher-order beliefs becomes significant only for uncountable information
structures.

An important contribution is that our result leads to the solution
of the last open problem posed in \citet{mertens_repeated_1986}\footnote{\label{fn:Mertens Problems}Problem 1 asks about the convergence of
the value, and it has been proved false in \citet{ziliotto_zero-sum_2016}.
Problem 3 asks about the equivalence between the existence of the
uniform value and uniform convergence of the value functions, and
it has been proved to be false by \citet{monderer_asymptotic_1993}
and \citet{lehrer_discounting_1994}.}. Specifically, his Problem 2 asks about the equicontinuity of the
family of value functions over information structures across all (uniformly
bounded) zero-sum games. The positive answer would have implied the
equicontinuity of the discounted and average values in repeated games,
which would have consequences for convergence in the limit theorems\footnote{\label{foot_zerosum} Equicontinuity of value functions is used to
obtain limit theorems in several works such as \citet{mertens1971value},
\citet{forges1982infinitely}, 
\citet{rosenberg2001operator}, \citet{rosenberg2000zero}, \citet{rosenberg2000maxmin},
\citet{rosenberg2004stochastic}, \citet{renault_value_2006}, \citet{gensbittel_value_2012},
\citet{venel2014commutative}, and \citet{renault_venel_2017}.}. Our results, however, indicate that the answer to the problem is
negative.

Our paper adds to the literature on the topologies of information
structures. \citet{dekel_topologies_2006} (see also \citet{morris_typical_2002})
introduce \emph{uniform-strategic topologies}, where two types are
close if, for any (not necessarily zero-sum) game, the sets of (almost)
rationalizable outcomes are (almost) equal.\footnote{\citet{dekel_topologies_2006} focuses mostly on a weaker notion of
\emph{strategic topology} that differs from the uniform strategic
topology in the same way that pointwise convergence differs from uniform
convergence. \citet{chen_uniform_2010} and \citet{chen_uniform_2016}
provide a characterization of strategic and uniform-strategic topologies
in terms of convergences of belief hierarchies.} There are two key differences between that and our approach. First,
the uniform-strategic topology applies to all (including non-zero-sum)
games. Our restriction allows us to show that some of the surprising
phenomena studied in this literature, like the difference between
approximate knowledge and common knowledge, are not relevant for zero-sum
games. Second, we work with \emph{ex-ante} information structures
and the equilibrium solution concept, whereas uniform-strategic topology
is designed to work on the \emph{interim} level, with rationalizability.
The ex-ante equilibrium approach is more appropriate for value comparison
and other related questions. For instance, in the information design
context, the quality of the information structure is typically evaluated
\emph{before} players receive any information.

Finally, this paper contributes to a recent but rapidly growing field
of information design (\citet{kamenica_bayesian_2011}, \citet{ely_beeps_2015},
\citet{bergemann_bayes_2015}, to name a few). In that literature,
an agent designs or acquires an information that is later used in
either a single-agent decision problem or a strategic situation. In
principle, the design of information may be divorced from the game
itself. For example, a bank may acquire software to process and analyze
large amounts of financial information before knowing what stock it
will trade, or, a spy master allocates resources to different tasks
or regions before she understands the nature of future conflicts.
Value-based distance is a tight upper bound on the willingness to
pay for a change in information structure. Our results provide insight
into the structure of the information designer's choice space, including
its diameter and internal complexity.

\section{Model}

\label{sec:Model}

A (countable) \emph{information structure} is an element $u\in\Delta\left(K\times\N\times\N\right)$
of the space of probabilities over tuples $\left(k,c,d\right)$, where
$K$ is a fixed finite set with $|K|\geq2$, and $\N$ is the set
of nonnegative integers\footnote{All the results from Sections \ref{sec:Computing} and \ref{sec:Applications} can be extended to uncountable information structures, the proof is relegated to Appendix \ref{sec:Uncountable-information-structur}).}. The interpretation is that $k$ is a state of nature, and $c$ and
$d$ are the signals of player 1 (maximizer) and player 2 (minimizer),
respectively. In other words, an information structure is a 2-player
common prior Harsanyi type space over $K$ with at most countably
many types. The set of information structures is denoted by $\CU=\CU\left(\infty\right)$,
and for $L=1,2,...$, $\CU\left(L\right)$, denotes the subset of
information structures where each player receives a signal smaller
than or equal to $L-1$ with probability $1$. If $C$ and $D$ are
nonempty countable sets, we always interpret elements $u\in\Delta(K\times C\times D)$
as information structures, using fixed enumerations of $C$ and $D$.
In particular, if $C$ and $D$ are finite with cardinality of at
most $L$, we view $u\in\Delta(K\times C\times D)$ as an information
structure in $\CU\left(L\right)$. For each $u,v\in\CU$, let us define
the total variation norm as $\|u-v\|=\sum_{k,c,d}|u(k,c,d)-v(k,c,d)|$.

A\emph{ payoff function} is a mapping $g:K\times I\times J\to[-1,1]$,
where $I,J$ are finite nonempty action sets. The set of payoff functions
with action sets of cardinality $\leq L$ is denoted by $\CG(L)$,
and let $\CG=\bigcup_{L\geq1}\CG(L)$ be the set of all payoff functions.

Information structure $u$ and payoff function $g$ together define
a zero-sum Bayesian game $\Gamma(u,g)$ played as follows: first,
$(k,c,d)$ is selected according to $u$, player 1 learns $c$, and
player 2 learns $d$. Next, simultaneously, player 1 chooses $i\in I$
and player 2 chooses $j\in J$, and finally the payoff of player 1
is $g(k,i,j)$. The zero-sum game $\Gamma(u,g)$ has a value (the
unique equilibrium, or minmax, payoff of player 1), which we denote
by $\val(u,g)$.

We define \emph{the value-based distance} between two information
structures as the largest possible difference in value across all
payoff functions: 
\begin{equation}
\dist(u,v)=\sup_{g\in\CG}|\val(u,g)-\val(v,g)|.\label{eq:ZS distance}
\end{equation}
This has a straightforward interpretation as the tight upper bound
on the gain or loss from moving from one information structure to
another. Since all payoffs are in $[-1,1]$, it is easy to see that
$\dist(u,v)\leq\|u-v\|\leq2$.\footnote{\label{fn:The-inequality-is}The inequality is a property of zero-sum
games. For every game $g\in\CG$, let $\sigma$ be an optimal strategy
of player $1$ in $\Gamma(u,g)$ and let $\tau$ be an optimal strategy
of player $2$ in $\Gamma(v,g)$. Using the saddle-point property
of the value, the difference $\val(u,g)-\val(v,g)$ is no larger than
the differences of payoffs in $\Gamma(u,g)$ and $\Gamma(v,g)$ when
the players play $(\sigma,\tau)$ in both games. This difference is
clearly no larger than $\|u-v\|$.}

The distance (\ref{eq:ZS distance}) satisfies two axioms of a metric:
the symmetry and the triangular inequality. However, it is possible
that $\dist(u,v)=0$ for $u\neq v$. For instance, if we start from
an information structure $u$ and relabel the player signals, we obtain
an information structure $u'$ that is formally different from $u$
but ``equivalent'' to $u$. Say that $u$ and $v$ are equivalent,
and write $u\sim v$ if, for all game structures $g$ in $\CG$, $\val(u,g)=\val(v,g)$.
We let $\CU^{*}=\CU/\sim$ be the set of equivalence classes. Thus,
$\dist$ is a pseudo-metric on $\CU$ and a metric on $\CU^{*}$.

For each information structure $u\in\Delta\left(K\times C\times D\right)$,
there is a unique belief-preserving mapping that maps signals $c$
and $d$ into induced Mertens-Zamir hierarchies of beliefs $\tilde{c}\in\Theta_{1}$
and $\tilde{d}\in\Theta_{2}$, where $\Theta_{i}$ is the universal
space of player $i$'s belief hierarchies over $K$ (see \citet{mertens_sorin_zamir_2015}).
The mapping induces a consistent probability distribution $\tilde{u}\in\Delta(K\times\Theta_{1}\times\Theta_{2})$
over the state and hierarchies of beliefs. Let $\Pi_{0}=\left\{ \tilde{u}:u\in\CU\right\} $
be the space of all such distributions. The closure of $\Pi_{0}$
(in the weak topology, that is, the topology induced by the product
convergence of belief hierarchies) is denoted as $\Pi$, where $\Pi$
is the space of consistent probability distributions induced by generalized
(measurable, possibly uncountable) information structures. The space
$\Pi$ is compact under weak topology; $\Pi_{0}$ is dense in $\Pi$
(see Corollary III.2.3 and Theorem III.3.1 in \citet{mertens_sorin_zamir_2015}).
Note that, for a payoff function $g$ and $u\in\Pi$, one can similarly
define the value $\val(u,g)$ of the associated Bayesian game (see
Proposition III.4.2 in \citet{mertens_sorin_zamir_2015}).

\section{Characterization of the distance}

\label{sec:Computing}

We start with the notion of garbling, used in \citet{blackwell1953}
to compare statistical experiments. A garbling is a mapping $q:\N\rightarrow\Delta(\N)$.
The set of all garblings is denoted by $\CQ=\CQ\left(\infty\right)$,
and for each $L=1,2,...$, $\CQ(L)$ denotes the subset of garblings
$q:\N\to\Delta(\{0,...,L-1\})$. Given a garbling $q$ and an information
structure $u$, we define the information structures $q.u$ and $u.q$
so that, for each $k,c,d$, 
\[
q.u(k,c,d)=\sum_{c^{\prime}}u(k,c{}^{\prime},d)q(c{}|c^{\prime})\text{ and }u.q(k,c,d)=\sum_{d^{\prime}}u(k,c,d{}^{\prime})q(d|d').
\]
We will interpret garblings in two different ways. First, a garbling
is seen as an information loss in the Blackwell's comparison of experiments
sense: suppose that $(k,c',d)$ is selected according to $u$, $c$
is selected according to probability $q(c')$, and player 1 learns
$c$ (and player 2 learns $d$). The new information structure is
exactly equal to $q.u$, where the signal received by player 1 is
deteriorated from garbling $q$. Similarly, $u.q$ corresponds to
the dual situation where player 2's signal is deteriorated. Further,
garbling $q$ can also be seen as a behavioral strategy of a player
in a Bayesian game $\Gamma(u,g)$: if the signal received is $c$,
play the mixed action $q(c)$ (the sets of actions of $g$ being identified
with subsets of $\N$). The relation between the two interpretations
plays an important role in the proof of Theorem \ref{thm1} below.
\begin{thm} \label{thm1} For each $L=1,2,...,\infty$, and each
$u,v\in\CU\left(L\right)$, 
\begin{eqnarray}
\sup_{g\in\CG}\left(\val(v,g)-\val(u,g)\right) & = & \min_{q_{1},q_{2}\in\CQ\left(L\right)}\|q_{1}.u-v.q_{2}\|.\label{eq:d characterization}
\end{eqnarray}
\[
\mathnormal{Hence},\hspace{0.5cm}\dist(u,v)=\max\left\{ \min_{q_{1},q_{2}\in\CQ\left(L\right)}\|q_{1}.u-v.q_{2}\|,\min_{q_{1},q_{2}\in\CQ\left(L\right)}\|u.q_{1}-q_{2}.v\|\right\} .
\]
If $L<\infty$, the supremum in (\ref{eq:d characterization}) is
attained by some $g\in\CG\left(L\right)$. \end{thm}
\begin{center}
\begin{figure}
\begin{centering}
\begin{tikzpicture} [scale = 6, dot/.style={circle,inner sep=2pt,fill},   extended line/.style={shorten >=-10,shorten <=-10}, blend mode=multiply] 
	
	\coordinate (u) at (0.6,0.75);
	\coordinate (v) at (1.5, 0.2);
	\fill[domain=0.1:1.1, variable=\x, red!30, opacity = 0.5] (0.1, 0.05) -- plot ({\x}, {0.75-0.7*4*(\x-0.6)*(\x-0.6)}) -- (1.1,0.05) -- cycle;
	\fill[domain=1:2, variable=\x, blue!30, opacity = 0.5] (1, 0.8) -- plot ({\x}, {0.2+0.6*4*(\x-1.5)*(\x-1.5)}) -- (2,0.8) -- cycle;

	\coordinate (x1) at (0.928, 0.449);
	\coordinate (x2) at (1.117, 0.552);

	\draw (u) node [dot] {};
	\draw (u) node [above] {$u$};
	\draw (v) node [dot] {};
	\draw (v) node [below] {$v$};

	\draw (0.6, 0.3) node [scale = 1.2] {$\CQ.u$};
	\draw (1.5, 0.6) node [scale = 1.2] {$v.\CQ$};
	
	\draw [>=triangle 45, <->] (x1)  -- (x2) node[above, midway, sloped, scale =0.7] {$\min\left\Vert \right\Vert >0$};

\end{tikzpicture}
\par\end{centering}
\caption{\label{fig:illustartion of theorem 1}}
\end{figure}
\par\end{center}

The first part of Theorem 1 finds a tight upper bound on the difference
of value between all zero-sum games played on information structures
$v$ and $u$. It is equal to a total variation distance between two
sets of garblings of the original information structures: the set
$\CQ.u=\left\{ q.u:q\in\CQ\right\} $ of structures obtained from $u$ by
deteriorating information of player 1 and the set $v.\CQ=\left\{ v.q:q\in\CQ\right\} $
obtained from $v$ by deteriorating information of player 2. The two
sets are illustrated on Figure \ref{fig:illustartion of theorem 1}.
(The direction up is better for player 1, and worse for player 2.) 

The result simplifies the problem of computing the value-based distance.
First, it reduces the dimensionality of the optimization domain from
payoff functions and strategy profiles (to compute the value) to a
pair of garblings. More importantly, the solution to the original
problem (\ref{eq:ZS distance}) is typically a saddle point as it
involves finding optimal strategies in a zero-sum game. On the other
hand, the function $\left\Vert q_{1}.u-v.q_{2}\right\Vert $ is convex
in garblings $(q_{1},q_{2})$, and, if $L<\infty$, the domains of
the optimization problem $\left\{ q.u:q\in\CQ\left(L\right)\right\} ,\left\{ u.q:q\in\CQ\left(L\right)\right\} $
are convex and compact. Thus, for finite structures, the right-hand
side of (\ref{eq:d characterization}) is a convex, compact, and finitely
dimensional optimization problem.

Theorem \ref{thm1} is closely related to the comparison of information
structures. Say \emph{player 1 prefers $u$ to $v$} in every game;
i.e., write $u\succeq v$, if for all $g\in\CG$, $\val(u,g)-\val(v,g)\geq0$.
The definition extends Blackwell's comparison of experiments to zero-sum
games. If the two sets from Figure \ref{fig:illustartion of theorem 1}
have non-empty intersection, the distance between them is equal to
0, and player 1 prefers $u$ to $v$. Conversely, if the two sets
do not have an intersection, then there are games for which the difference
in value on $v$ and $u$ is positive (and arbitrarily close to the
total variation distance) and player 1 does not prefer $u$ to $v$.
Hence Theorem \ref{thm1} implies the following \begin{cor} \label{cor1}
$u\succeq v$ $\Longleftrightarrow$ there exists $q_{1}$, $q_{2}$
in $\CQ$ s.t. $q_{1}.u=v.q_{2}$. \end{cor}

The above result extends \citep{peski_comparison_2008} to countable
information structures (\citep{peski_comparison_2008} was stated
only for finite structures). 

The intuition of the proof of Theorem \ref{thm1} is as follows. For
simplicity, suppose that the information structure is finite. We first
show that the maximum of the difference in values over all games $g$,
i.e., the right-hand side of (\ref{eq:d characterization}), is not
smaller than the left-hand side of (\ref{eq:d characterization}).
The starting point is to identify each garbling with a mixed strategy
in Bayesian game $\Gamma(u,g)$ induced from information structure
$u$. Using this identification, the expected payoff in this game
can be written as $\langle g,q_{1}.u.q_{2}\rangle$, where $\langle g,u\rangle=\sum_{k,c,d}g(k,c,d)u(k,c,d)$
and $q_{i}$ are garbling/strategies. Among others, each player can
use strategy $\text{Id}$ which plays the received signal. Because
such a strategy is always available to both players, using the minmax
property of equilibrium, the difference in values $\val(v,g)-\val(u,g)$
is no less than the difference between player 2's optimal payoff against
the $\text{Id}$ strategy of player 1 in $v$ (i.e. $\inf_{q_{2}}\langle g,v.q_{2}\rangle$)
and player 1's optimal payoff against the $\text{Id}$ strategy of
player 2 in $u$ (i.e. $\sup_{q_{1}}\langle g,q_{1}.u\rangle$). Since
this holds for any game $g$, it follows that the value-based distance
is bounded below by $\sup_{g}\inf_{q_{1},q_{2}}\langle g,v.q_{2}-q_{1}.u\rangle$.
The latter is equal to the left-hand side of (2). 

Next, we show that, for each game $g$, the difference in values is
not higher than the left-hand side of (\ref{eq:d characterization}).
Using the monotony of the value with respect to information, we have
that 
\[
\val(v,g)-\val(u,g)\leq\val(v.q_{2},g)-\val(q_{1}.u,g)\leq\|v.q_{2}-q_{1}.u\|
\]
for arbitrary garblings $q_{1}$ and $q_{2}$. The last inequality
was stated above (see footnote \ref{fn:The-inequality-is}). Observing
that $\|v.q_{2}-q_{1}.u\|=\sup_{g}\langle g,v.q_{2}-q_{1}.u\rangle$,
we deduce that the value-based distance is also bounded above by $\inf_{q_{1},q_{2}}\sup_{g}\langle g,v.q_{2}-q_{1}.u\rangle$.
Theorem 1 then follows from the Sion's Minimax Theorem. We leave the
complete proof to the Appendix.

\section{Applications}

\label{sec:Applications}

The characterization from Theorem \ref{thm1} is quite tractable.
This section contains a few straightforward applications. 
Appendix \ref{sec:Examples-and-Counterexamples} contains numerous examples to illustrate the computations and the subsequent results.

\subsection{The impact of the marginal over $K$}

Among the many ways in which two information structures can differ,
the most obvious one is that they may have different distributions
over states $k$. In order to capture the impact of such differences,
the next result provides tight bounds on the distance between two
type spaces with a given distribution overs the states: \begin{prop}
\label{prop:diameter} For each $p,q\in\Delta K$, each $u,v\in\CU$
such that $\marg_{K}u=p,\marg_{K}v=q$, we have 
\begin{align}
\sum_{k}\left|p_{k}-q_{k}\right|\leq\dist\left(u,v\right) & \leq2\left(1-\max_{p^{\prime},q{}^{\prime}\in\Delta K}\sum_{k}\min\left(p_{k}q{}_{k}^{\prime},p{}_{k}^{\prime}q_{k}\right)\right).\label{eq:diameter bounds}
\end{align}
If $p=q$, the upper bound is equal to $2\left(1-\max_{k}p_{k}\right)$.
\end{prop}

The bounds are tight. The lower bound in (\ref{eq:diameter bounds})
is reached when the two information structures do not provide any
information to any of the players. The upper bound is reached with
information structures where one player knows the state perfectly
and the other player does not know anything.

When $p=q$, Proposition \ref{prop:diameter} describes the diameter
of the space of information structures with the same distribution
$p$ of states. The result is useful for, among others, information
design questions, where such space is exactly the choice set when
Nature fixes the distribution of states, and the designer of information
chooses how much information to acquire. In such a case, the diameter
has an interpretation of the (tight) upper bound on the potential
gain/loss from moving between information structures.

\subsection{Single-agent problems}

A natural question is what games maximize value-based distance $\dist$.
The next result characterizes the situations, when the maximum in
(\ref{eq:ZS distance}) is attained by a special class of zero-sum
games: single-agent problems.

Formally, a payoff function $g\in\CG(L)$ is a \emph{single-agent
(player 1) problem} if player 2's action set is a singleton, $J=\left\{ *\right\} $.
Let $\CG_{1}\subset\CG$ be the set of player 1 problems. Then, for
each $g\in\CG_{1}$ and each information structure $u$, $\val\left(g,u\right)$
is the maximal expected payoff of player 1 in problem $g$. Let 
\begin{equation}
\dist_{1}\left(u,v\right):=\sup_{g\in\CG_{1}}\left|\val\left(u,g\right)-\val\left(v,g\right)\right|\leq\dist\left(u,v\right).\label{eq:single agent distance}
\end{equation}

For any structure $u\in\Delta\left(K\times C\times D\right)$, we
say that the players' information is \emph{conditionally independent},
if, under $u$, signals $c$ and $d$ are conditionally independent
given $k$.

\begin{prop}\label{prop: SA=00003D00003D00003D00003D00003D00003D00003D00003D00003D00003DGames}Suppose
that $u,v\in\Delta\left(K\times C\times D\right)$ are two information
structures with conditionally independent information such that $\marg_{K\times D}u=\marg_{K\times D}v$.
Then, $\dist\left(u,v\right)=\dist_{1}\left(u,v\right).$ \end{prop}

Proposition \ref{prop: SA=00003D00003D00003D00003D00003D00003D00003D00003D00003D00003DGames}
says that, if two information structures differ only by one player's
additional piece of information, and the players' information are
conditionally independent in both cases, then the maximum value-based
distance (\ref{eq:ZS distance}) is attained in a single-agent decision
problem. Such problems form a relatively small subclass of games and
they are easier to identify. 
In Appendix \ref{sec:Value-of-experiments}, we apply the proposition to compute the exact distance between information structures induced by multiple Blackwell experiments.

The proof of Proposition \ref{prop: SA=00003D00003D00003D00003D00003D00003D00003D00003D00003D00003DGames}
relies on the characterization from Theorem \ref{thm1} and shows
that the minimum in the optimization problem is attained by the same
pair of garblings as in the single-agent version of the problem.

\subsection{Value of additional information: games vs. single agent }

Consider two information structures $u\in\Delta\left(K\times(C\times C{}^{\prime})\times D\right)$
and $v=\marg_{K\times C\times D}u$. When moving from $v$ to $u$,
player 1 gains an additional signal $c^{\prime}$. Because $u$ represents
more information, $u$ is (weakly) more valuable, and the value of
the additional information is defined as $\dist\left(u,v\right),$
which is equal to the tight upper bound on the gain from the additional
signal. A corollary to Proposition \ref{prop: SA=00003D00003D00003D00003D00003D00003D00003D00003D00003D00003DGames}
shows that if the signals of the two players are independent conditional
on the state, the gain from the new information is the largest in
single-agent problems. \begin{cor}\label{prop: Games vs Single agent}
Suppose that information in $u$ (and therefore in $v$) is conditionally
independent. Then, $\dist\left(u,v\right)=\dist_{1}\left(u,v\right).$
\end{cor}

\subsection{Informational substitutes}

Next, we ask two questions about the impact of a piece of information
on the value of another piece of information. In both cases, we use
some conditional independence assumptions that are weaker than in
Proposition \ref{prop: SA=00003D00003D00003D00003D00003D00003D00003D00003D00003D00003DGames}.
Suppose that 
\[
\begin{split} & u\in\Delta\left(K\times(C\times C_{1}\times C_{2})\times D\right)\text{ and }v=\marg_{K\times(C\times C_{1})\times D}u,\\
 & u^{\prime}=\marg_{K\times(C\times C_{2})\times D}u,\text{ and }v^{\prime}=\marg_{K\times C\times D}u.
\end{split}
\]
When moving from $v^{\prime}$ to $u^{\prime}$ or $v$ to $u$, player
1 gains an additional signal $c_{2}$. The difference is that, in
the latter case, player 1 has more information that comes from signal
$c_{1}$. The next result shows the impact of an additional signal
on the value of information. \begin{prop} \label{prop. Info substitutes}Suppose
that, under $u$, $c_{1}$ is conditionally independent from $\left(c,c_{2},d\right)$
given $k$. Then, $\dist\left(u^{\prime},v^{\prime}\right)\geq\dist\left(u,v\right).$
\end{prop} Given the assumptions, the marginal value of signal $c_{2}$
decreases when signal $c_{1}$ is also present. In other words, the
two pieces of information are substitutes.

\subsection{Informational complements}

Another question is about the impact of an additional piece of information
for the other player on the value of information. Suppose that 
\begin{align*}
u & \in\Delta\left(K\times(C\times C_{1})\times(D\times D_{1})\right)\text{ and }v=\marg_{K\times C\times(D\times D_{1})}u,\\
u^{\prime} & =\marg_{K\times(C\times C_{1})\times D}u\text{ and }v^{\prime}=\marg_{K\times C\times D}u.
\end{align*}
When moving from $v^{\prime}$ to $u^{\prime}$ or $v$ to $u$, in
both cases, player 1 gains additional signal $c_{1}$. However, in
the latter case, player 2 obtains an additional piece of information
from signal $d_{1}$. The next result shows the impact of the opponent's
signal on the value of information. \begin{prop} \label{prop. Info complements}
Suppose that $\left(c,c_{1}\right)$ and $d$ are conditionally independent
given $k$. Then, $\dist\left(u^{\prime},v^{\prime}\right)\leq\dist\left(u,v\right).$
\end{prop} Given the assumptions, signal $c_{1}$ becomes more valuable
when the opponent also has access to additional information. Hence,
the two pieces of information are complements.

\subsection{Value of joint information}

\label{subsec:Value-of-joint}

Finally, we consider a situation where two players simultaneously
receive additional information. Consider a distribution $\mu\in\Delta\left(X\times Y\times Z\right)$
over countable spaces. We say that random variables $x$ and $y$
are $\varepsilon$-conditionally independent given $z$ if 
\[
\sum_{z}\mu\left(z\right)\sum_{x,y}\left|\mu\left(x,y|z\right)-\mu\left(x|z\right)\mu\left(y|z\right)\right|\leq\varepsilon.
\]

Let $u\in\Delta\left(K\times(C\times C_{1})\times(D\times D_{1})\right)$
and $v=\text{marg}_{K\times C\times D}u$. When moving from $v$ to
$u$, both players receive a piece of additional information.

\begin{prop} \label{prop: joint info}Suppose that $d_{1}$ is $\varepsilon$-conditionally
independent from $\left(k,c\right)$ given $d$, and $c_{1}$ is $\varepsilon$-conditionally
independent from $\left(k,d\right)$ given $c$. Then, $\dist\left(u,v\right)\leq\varepsilon.$
\end{prop}

The proposition considers the potential scenario where the additional
signal for each player does not provide the respective player with
any significant information about the state of the world or the original
information of the other player. While such signals would be useless
in a single-agent decision problem, they may be useful in a strategic
setting, as valuable information may be contained in their joint distribution.\footnote{How useful it is depends on the solution concept. The joint information
is important for Bayesian Nash Equilibrium and Independent Interim
Rationalizability - see the leading example of \citet{ElyPeski06}.
The joint information is not important \emph{by assumption }for the
Bayes Correlated Equilibrium of \citet{bergemann_bayes_2015} or Interim
Correlated Rationalizability of \citet{dekel_interim_2007}. } Nevertheless, Proposition \ref{prop: joint info} says that the information
that is jointly shared by the two players is not valuable in zero-sum
games.

Despite its simplicity, Proposition \ref{prop: joint info} has powerful
consequences. Below, we use it to show that information structures
with approximate knowledge of the state also have approximate common
knowledge of the state. More generally, we use it in the proof of
Theorem \ref{THM4WEAKTOPOLOGY}.

\section{Large space of information structures}

\label{sec:Compactness}

\subsection{$(\CU^{*},\dist)$ is not totally bounded}

In this section, we assume without loss of generality that $K=\left\{ 0,1\right\} $.

\begin{thm} \label{THM3LARGESPACE}There exists $\varepsilon>0$
and a sequence $(u^{l})$ of information structures such that $\dist(u^{l},u^{p})>\varepsilon$
if $l\neq p$.\end{thm}

The theorem says that the space of information structures is large:
it cannot be partitioned into finitely many subsets such that all
structures in a subset are arbitrarily close to each other.

The proof, with the exception of one step that we describe below,
is constructive. For fixed large $N$, we construct a probability
$\mu$ over infinite sequences $k,c_{1},d_{1},c_{2},d_{2},...$ that
starts with a state $k$ followed by alternating signals for each
player. The sequence $c_{1},d_{1},c_{2},d_{2},...$ follows a Markov
chain on $\{1,...,N\}$, and state $k$ only depends on $c_{1}$.
In structure $u^{l}$, player 1 observes signals $(c_{1},c_{2},....,c_{l})$,
and player 2 observes $(d_{1},d_{2},....,d_{l})$. Thus, the sequence
of structures $u^{l}$ can be understood as fragments of a larger
information structure, where progressively more information is revealed
to each player. The theorem shows that the larger structure is not
the limit of its fragments in the value-based distance. In particular,
there is no analog of the martingale convergence theorem for the value-based
distance for such sequences.

This has to be contrasted with two other settings, where the limits
of information structures are well defined. First, in the single-player
case, any sequence of information structures in which the player is
receiving progressively more signals converges for distance $\dist_{1}$.
Second, the Markov property means that (a) the state is independent
from all players' information conditionally on $c_{1}$, and (b) each
new piece of information is independent from the previous pieces of
information, conditional on the most recent information of the other
player. This ensures that the $l$-th level hierarchy of beliefs of
any type in structure $u^{l}$ is preserved by all consistent types
in structures $u^{p}$ for $p\geq l$. Therefore, Theorem \ref{THM3LARGESPACE}
exhibits a sequence of type spaces in which belief hierarchies converge
in the product topology. In particular, it shows that the knowledge
of the $l$-th level hierarchy of beliefs for any arbitrarily high
$l$ is not sufficient to play $\varepsilon$-optimally in all finite
zero-sum games.

\subsection{Last open problem of Mertens\label{subsec:Last-open-problem}}

Recall that, for each information structure $u$, $\tilde{u}$ denotes
the associated consistent probability distribution over belief hierarchies.
Because each finite-level hierarchy of beliefs becomes constant as
we move along the sequence $u^{l}$, it must be that sequence $\tilde{u}^{l}$
converges weakly in $\Pi$ to the limit $\tilde{u}^{l}\rightarrow\tilde{\mu}.$
The limit is the consistent probability obtained from the prior distribution
$\mu$. Theorem \ref{THM3LARGESPACE} shows that 
\[
\limsup_{l}\sup_{g\in\CG}\left|\val\left(\mu,g\right)-\val\left(u^{l},g\right)\right|\geq\varepsilon.
\]
In particular, the family of all functions $(u\mapsto\val(u,g))_{g\in\CG}$
is not equicontinuous on $\Pi$ equipped with the weak topology. This
answers negatively the second of the three problems posed by \citet{mertens_repeated_1986}
in his Repeated Games survey from ICM: ``This equicontinuity or Lispchitz
property character is crucial in many papers...'' (see also footnote
\ref{fn:Mertens Problems}).

The importance of the Mertens question comes from the role that it
plays in the limit theorems of repeated games. The existence of a
limit value has attracted a lot of attention since the first results
by \citet{aumann_repeated_1995} and \citet{mertens1971value} for
repeated games and by \citet{bewley1976asymptotic} for stochastic
games. Once the equicontinuity of an appropriate family of value functions
is established, the existence of the limit value is typically obtained
by showing that there is at most one accumulation point of the family
$(v_{\delta})$, for example, by showing that any accumulation point
satisfies a system of variational inequalities admitting at most one
solution (see e.g. the survey \citet{laraki2015advances} and footnote
\ref{foot_zerosum} for related works).

\subsection{Comments on the proof}

Fix $\alpha<\frac{1}{25}$. We show that we can find a sufficiently
high, even-valued $N$ and a set $S\subseteq\left\{ 1,...,N\right\} ^{2}$
with certain mixing properties: 
\begin{align*}
\left|\left\{ i:\left(i,j\right)\in S\right\} \right| & \simeq\frac{N}{2},\text{ for each }j,\\
\left|\left\{ i:\left(i,j\right),\left(i,j^{\prime}\right)\in S\right\} \right| & \simeq\frac{N}{4},\text{ for each }j,j{}^{\prime},\\
\left|\left\{ i:\left(i,j\right),\left(i,j^{\prime}\right),\left(l,i\right)\in S\right\} \right| & \simeq\frac{N}{8},\text{ for each }j,j{}^{\prime},l,etc.
\end{align*}
The ``$\simeq$'' means that the left-hand side is within $\alpha$-related
distance to the right-hand side. Altogether, there are 8 properties
of this sort (see Appendix \ref{subsec:Existence-of-Markov chain})
that essentially mean that various sections of $S$ are ``uncorrelated''
with one another.

We are unable to directly construct $S$ with the required properties.
Instead, we show the existence of set $S$ using the probabilistic
method of P. Erd\H{o}s (for a general overview of the method, see
\citet{alon_probabilistic_2008}). Suppose that sets $S\left(i\right)$,
for $i=1,...,N$, are chosen independently and uniformly from all
$\frac{N}{2}$-element subsets of $\left\{ 1,...,N\right\} $. We
show that, if $N\geq10^{8}$, then set $S=\left\{ \left(i,j\right):j\in S\left(i\right)\right\} $
satisfies the required properties with positive probability, proving
that a set satisfying these properties exists. Our proof is not particularly
careful about optimal $N$ (or about the largest $\varepsilon$ allowing
for the conclusions of Theorem \ref{THM3LARGESPACE}).

Given $S$, we construct probability distribution $\mu$. First, state
$k$ is chosen with equal probability, and $c_{1}$ is chosen so that
$\frac{c_{1}}{N+1}$ is the conditional probability of $k=1$. Next,
inductively, for each $l\geq1$, we choose 
\begin{itemize}
\item $d_{l}$ uniformly from set $S\left(c_{l}\right)=\left\{ j:\left(c_{l},j\right)\in S\right\} $
and conditionally independently from $k,...,d_{l-1}$ given $c_{l}$,
and 
\item $c_{l+1}$ uniformly from set $S\left(d_{l}\right)$ and conditionally
independently from $k,...,c_{l}$ given $d_{l}$. 
\end{itemize}
As a result, $c_{1}$, $d_{1}$, $c_{2}$, $d_{2}$,... follows a
Markov chain.

To provide a lower bound for the distance between different information
structures, we construct a sequence of games. In game $g^{p}$, player
1 is supposed to reveal the first $p$ pieces of her information;
player 2 reveals the first $p-1$ pieces. The payoffs are such that
it is a dominant strategy for player 1 to precisely reveal her first-order
belief about the state, which amounts to truthfully reporting $c_{1}$.
Furthermore, we verify whether the sequence of reports $\left(\hat{c}_{1},\hat{d}_{1},...,\hat{c}_{p-1},\hat{d}_{p-1},\hat{c}_{p}\right)$
belongs to the support of the distribution of the Markov chain. If
it does, then player 1 receives payoff $\varepsilon\sim\frac{1}{10\left(N+1\right)^{2}}$.
If it does not, we identify the first report in the sequence that
deviates from the support. The responsible player is punished with
payoff $-5\varepsilon$ (and the opponent receives $5\varepsilon$).

The payoffs and the mixing properties of matrix $S$ ensure that players
have incentives to report their information truthfully. We check it
formally, and we show that, if $l>p$, then $\dist\left(u^{l},u^{p}\right)\geq\val\left(u^{l},g^{p+1}\right)-\val\left(u^{p},g^{p+1}\right)\geq2\varepsilon.$

Our argument implies that the conclusion of Theorem \ref{THM3LARGESPACE}
is true for $\varepsilon=2.10^{-17}$. However, our argument is not
optimized for the largest possible value of $\varepsilon$, and we
strongly suspect that the threshold $\varepsilon$ is much larger.

\section{Value-based topology}

\label{sec:Value-based-topology}

\subsection{Relation to the weak topology}

Previous sections discussed the quantitative aspect of value-based
distance. Now, we analyze its qualitative aspect: topological information.

\begin{thm} \label{THM4WEAKTOPOLOGY} Let $u$ be in $\CU^{*}$.
A sequence $(u_{n})$ in $\CU^{*}$ converges to $u$ for value-based
distance if and only if the sequence $(\tilde{u}_{n})$ converges
weakly to $\tilde{u}$ in $\Pi_{0}$. \end{thm}

The result says that a convergence in value-based topology to a countable
structure is equivalent to the convergence in distribution of finite-order
hierarchies of beliefs. Informally, around countable structures, the
higher-order beliefs have diminishing importance.

The intuition of the proof is as follows: If $u$ is finite, we surround
the hierarchies $\tilde{c}$ for $c\in C$ by sufficiently small and
disjoint neighborhoods, so that all hierarchies in the neighborhood
of $\tilde{c}$ have similar beliefs about the state and the opponent.
We do the same for the other player. Weak convergence ensures that
the converging structures assign large probability to the neighborhoods.
We show that any information about a player's hierarchy beyond the
neighborhood to which it belongs is almost conditionally independent
(in the sense of Section \ref{subsec:Value-of-joint}) from information
about the state and the opponents' neighborhoods. By Proposition \ref{prop: joint info},
only information about neighborhoods matters, and the latter is similar
to the information in limit structure $u$. If $u$ is countable,
we also show that it can be appropriately approximated by finite structures.

There are two reasons why Theorem \ref{THM4WEAKTOPOLOGY} is surprising:
It seems to (a) convey a message that is opposite to the literature
on strategic (dis)continuities, and (b) contradict our discussion
of Theorem \ref{THM3LARGESPACE}. We deal with these two issues in
order.

\subsubsection{Strategic discontinuities}

For an illustration of the first issue, consider email-game information
structures $u$ from \citet{rubinstein_electronic_1989}\footnote{The email game information structure $u_\varepsilon$ is as follows: there are two states $0$ and $1$, the latter with probability $p$. Player 1
knows the state. If the state is $1$, a message is sent to the other
player, who, upon receiving it, immediately sends it back. The message
travels back and forth until it is lost, which happens with i.i.d.
probability $\varepsilon>0$ each time it travels. The signal of each player in $u_\varepsilon$ is the number of messages she received. In \cite{rubinstein_electronic_1989}, a non-zero-sum coordination game is considered, and shown to have the property that the set of (Bayesian Nash) equilibrium payoffs with $u_\varepsilon$ does not converge to the set of equilibrium payoffs with $u_0$, where $u_0$ is simply the structure corresponding to common knowledge of the state.}. Player 1
always knows the state. Player 2's first-order belief attaches probability
of at least $\frac{1}{1+\varepsilon\frac{p}{1-p}}$ to one of the
states, where $p<1$ is the initial probability of one of the states
and $\varepsilon$ is the probability of losing the message. It is
well-known that, as $\varepsilon\rightarrow0$, the Rubinstein type
spaces converge in the weak topology to the common knowledge of the
state. Theorem \ref{THM4WEAKTOPOLOGY} implies that the Rubinstein
type spaces also converge under value-based distance.

We can make the last point somehow more general. An information structure
$u\in\Delta\left(K\times C\times D\right)$ exhibits $\varepsilon$-knowledge
of the state if there is a mapping $\kappa:C\cup D\rightarrow K$
such that 
\[
u\Big(\left\{ u(\{k=\kappa(c)\}|c)\geq1-\varepsilon\right\} \Big)\geq1-\varepsilon\text{ and }u\Big(\left\{ u(\{k=\kappa(d)\}|d)\geq1-\varepsilon\right\} \Big)\geq1-\varepsilon.
\]
In other words, the probability that any player assigns at least $1-\varepsilon$
to some state is at least $1-\varepsilon$.

\begin{prop} \label{PROP: APPROXIMATE KNOWLEDGE}Suppose that $u$
exhibits $\varepsilon$-knowledge of the state and that $v\in\Delta\left(K\times K_{C}\times K_{D}\right)$,
where $K_{C}=K_{D}=K$, $\marg_{K}v=\marg_{K}u$, and $v\left(k=k_{C}=k_{D}\right)=1$.
(In other words, $v$ is a common knowledge structure with the only
information about the state.) Then, 
\[
\dist\left(u,v\right)\leq20\varepsilon.
\]
\end{prop}

Therefore, approximate knowledge structures are close to common knowledge
structures. The convergence of approximate knowledge type spaces to
common knowledge is a consequence of Theorem \ref{THM4WEAKTOPOLOGY}.
The metric bound stated in Proposition \ref{PROP: APPROXIMATE KNOWLEDGE}
requires a separate (simple) proof based on Proposition \ref{prop: joint info}.

The above results seem to go against the main message of the strategic
discontinuities literature (\citet{rubinstein_electronic_1989}, \citet{dekel_topologies_2006},
\citet{weinstein2007structure}, \citet{ely_critical_2011}, etc.),
where the convergence of finite-order hierarchies does not imply strategic
convergence even around finite structures. There are three important
ways in which our setting differs. First, we rely on the ex-ante equilibrium
concept, rather than interim rationalizability. We are also interested
in payoff comparison rather than behavior. Second, we restrict attention
to zero-sum games. Finally, we only work with common prior type spaces.

We believe that each of these differences is important. First, if
we worked with rationalizability, an argument due to \citet{weinstein2007structure}
applies, and, assuming sufficient richness, it can be used to show
that the resulting topology is strictly finer than weak topology\footnote{We are grateful to Satoru Takahashi for clarifying this point.}.
Further, the ex-ante focus and payoff comparison (but without restriction
to zero-sum games) lead to a topology that is significantly finer
than weak topology (in fact, so fine that it can be useless - see
Section \ref{sec:Payoff distance NZS} for a detailed discussion).
The role of common prior is less clear. On the one hand, \citet{lipman:03}
implies that, at least from the interim perspective, a common prior
does not generate significant restrictions on finite-order hierarchies.
On the other hand, we rely on the ex-ante perspective, and common
prior is definitely important for Proposition \ref{prop: joint info},
which plays an important role in the proof.

Let us also mention that Proposition \ref{PROP: APPROXIMATE KNOWLEDGE} is also related to \cite{kajiimorris97robust}(and additional results in \cite{MORRISUI} Morris-Ui and \cite{oyama2010robust}). In that paper, the authors fix a
game with complete information and study the robustness of equilibria
to perturbation of complete information in similar way to our notion
of $\varepsilon$-knowledge. They show that if the game has a unique
correlated equilibrium, then for sufficiently small $\varepsilon$,
nearby games have an equilibrium with behavior close to the equilibrium
of the complete information game. Proposition \ref{PROP: APPROXIMATE KNOWLEDGE} says that, for zero-sum
games that may have multiple equilibria, equilibrium payoffs are robust to small amounts of incomplete information. 

\subsubsection{Relation to Theorem \ref{THM3LARGESPACE}}

For the second issue, recall that Theorem \ref{THM3LARGESPACE} exhibits
a sequence of countable information structures such that the hierarchies
of beliefs converge in the weak topology along the sequence, but the
sequence does not converge in the value-based distance. The limiting
structure, namely the distribution of the realizations of the infinite
Markov chain, is \emph{uncountable}. On the other hand, Theorem \ref{THM4WEAKTOPOLOGY}
says that convergence in weak topology to a \emph{countable} information
structure is equivalent to convergence in value-based distance. Together,
the two results imply that, although weak and value-based topologies
are equivalent around countable structures $\CU^{*}$, they differ
beyond $\CU^{*}$. The impact of higher-order beliefs becomes significant
only for uncountable information structures.

Another way to illustrate the relation between two results is to observe
that, although the two topologies coincide on $\CU^{*}\simeq\Pi_{0}$,
and the latter has compact closure $\Pi$ under weak topology, the
completion of $\CU^{*}$ with respect to $\dist$ is not compact.
This should not be confusing, as the ``completion'' is metric specific
and not a purely topological notion, and different metrics that induce
the same topology can have different completions.

\subsection{Pointwise value-based topology and completions}

An alternative way to define a topology on the space of information
structures would be through the convergence of values. Say that a
sequence of information structures $(u_{n})$ converges pointwise
to $u$ if, for all payoff functions $g\in\CG$, $\lim_{n\rightarrow\infty}\val(u_{n},g)=\val(u,g)$.
Clearly, if $(u_{n})$ converges to $u$ for value-based distance,
then it also converges to $u$ pointwise.

The topology of pointwise convergence is the weakest topology that
makes the value mappings continuous. Since $\val(\mu,g)$ is also
well defined for $\mu$ in $\Pi$, pointwise convergence is also well-defined
on $\Pi$. Moreover by Theorem 12 of \citet{gossner_value_2001},
the topology of pointwise convergence coincides with the topology
of weak convergence on $\Pi$. Using Theorem \ref{THM4WEAKTOPOLOGY},
we obtain the following corollary: \begin{cor} On set $\CU^{*}$,
the topology induced by value-based distance, the topology of weak
convergence, and the topology of pointwise convergence coincide. In
particular, let $u$ be in $\CU^{*}$ and $(u_{n})$ be in $\CU^{*}$.Then
$(u_{n})$ converges to $u$ for value-based distance if and only
if for every $g$ in $\CG$, $\val(u_{n},g)\xrightarrow[n\to\infty]{}\val(u,g)$.
\end{cor}

A standard way to define a metric compatible with pointwise topology
is the following. Consider any sequence $(g_{n})_{n}$ that is dense
in the set of payoff functions $\cup_{L\geq1}[-1,1]^{K\times L^{2}}$
in the sense\footnote{To construct such a sequence, one can for instance proceed as follows.
For each positive integer $L$, consider a finite grid approximating
$[-1,1]^{K\times L^{2}}$ up to $1/L$ , then define $(g_{n})_{n}$
by collecting the elements of all grids.} that, for each $g$ in $[-1,1]^{K\times L^{2}}$ and $\varepsilon>0$,
there exists $n$ such that $|g(k,i,j)-g_{n}(k,i,j)|\leq\varepsilon$
for all $(k,i,j)\in K\times L^{2}$. The particular choice of $(g_{n})_{n}$
will play no role in the sequel. Define now the distance $\dist_{W}$
on $\CU^{*}$ by: 
\[
\dist_{W}(u,v)=\sum_{n=1}^{\infty}\frac{1}{2^{n}}|\val(u,g_{n})-\val(v,g_{n})|.
\]
By density of $(g_{n})_{n}$, we have $\dist_{W}(u_{l},u)\xrightarrow[l\to\infty]{}0$
if and only if, for all $g$, $\val(u_{l},g)\xrightarrow[l\to\infty]{}\val(u,g)$.
$\CU^{*}$ equipped with $\dist_{W}$ is a metric space, and we denote
by $\CV$ its completion for $\dist_{W}$. For this distance, $\CU^{*}$
is isometric to a dense subset of $\CV$, so that $\CV$ can be seen
as the closure of $\CU^{*}$. Using Theorem 12 of \citet{gossner_value_2001},
we have the following result.

\begin{thm}\label{thm5completion} $\CV$ is homeomorphic to the
space $\Pi$, endowed with weak topology.\end{thm}

\begin{proof} Define similarly distance $\dist_{W}$ on $\Pi$ as
$\dist_{W}(\mu,\nu)=\sum_{n=1}^{\infty}\frac{1}{2^{n}}|\val(\mu,g_{n})-\val(\nu,g_{n})|.$
By construction, the mapping $(u\mapsto\tilde{u})$ from $\CU^{*}$
to $\Pi_{0}$ is an isometry for $\dist_{W}$. So $\CV$ is isometric
to the completion of $\Pi_{0}$ for $\dist_{W}$. But on $\Pi$, the
topology induced by $\dist_{W}$ is weak topology, and for this topology,
$\Pi$ is the closure of $\Pi_{0}$. So the completion of $\Pi_{0}$
for $\dist_{W}$ is $\Pi$. \end{proof}

As a consequence, $\CV$ is compact and does not depend on the choice
of $(g_{n})$. It contains not only the information structures with
countably many types, but also the information structures with a continuum
of signals, obtained as limits of sequences of information structures
with countably many types.

The main point of interest in Theorem \ref{thm5completion} is that
we can now view $\Pi$ as set $\CV$. We can recover the exact space
$(\Pi,weak)$ using values of zero-sum Bayesian games and the completion
of a metric space\footnote{We could have worked from the beginning with possibly uncountable
information structures, i.e., with Borel probabilities over $K\times[0,1]\times[0,1]$.
Endowing this set with distance $\dist_{W}$ yields a metric space
directly homeomorphic to $\Pi$, with no need to go to completion
since the space would already be complete. See online material.}. This may be seen as a duality result between games and information:
$\Pi$ is defined with hierarchies of beliefs but with no reference
to games and payoffs, whereas $\CV$ is defined by values of zero-sum
games, with no explicit reference to belief hierarchies. In particular,
restricting attention to the values of zero-sum games is still sufficient
to obtain the full space $\Pi$ with weak topology. Now the construction
of $\CV$ yields a new, alternative, interpretation of $\Pi$, and
one might possibly hope to deduce properties of ($\Pi$, weak) by
transferring, via the homeomorphism, properties first proven on $\CV$.

Finally, although $\dist_{W}$ and our value-based distance $\dist$
induce the same topology on $\CU^{*}$, their completions differ.
Theorem \ref{THM3LARGESPACE} implies that completion $\CW$ of $\CU^{*}$
for $\dist$ is not compact. Space $\CW$ also contains information
structures with a continuum of signals and represents a new space
of incomplete information structures with strong foundations based
on the suprema of differences between values of Bayesian games.

\section{Payoff-based distance}

\label{sec:Payoff distance NZS}

In this section, we consider a version of distance (\ref{eq:ZS distance})
where the supremum is taken over all (including non-zero-sum) games.
Rubinstein's e-mail game (\citep{rubinstein_electronic_1989}) shows the relevance of almost conditional independent information for non-zero-sum games (in the sense of Section \ref{subsec:Value-of-joint}). Therefore similar statements as Theorem \ref{THM4WEAKTOPOLOGY} and Proposition \ref{PROP: APPROXIMATE KNOWLEDGE} do not hold when considering non-zero-sum games. We show the stronger result that such a payoff-based distance between information structures is mostly trivial.

A \emph{non-zero sum payoff function} is a mapping $g:K\times I\times J\rightarrow\left[-1,1\right]^{2}$
where $I,J$ are finite sets. Let $\text{Eq}\left(u,g\right)\subseteq\R^{2}$
be the set of Bayesian Nash Equilibrium (BNE) payoffs in game $g$
on information structure $u$. Assume that the space $\R^{2}$ is
equipped with the maximum norm $\left\Vert x-y\right\Vert _{\max}=\max_{i=1,2}\left|x_{i}-y_{i}\right|$
and that the space of compact subsets of $\R^{2}$ is equipped with
the induced Hausdorff distance $\dist_{\max}^{H}$. Let 
\begin{equation}
\dist_{NZS}\left(u,v\right)=\sup_{g\text{ is a non-zero-sum payoff function}}\dist_{\max}^{H}\left(\text{Eq}\left(u,g\right),\text{Eq}\left(v,g\right)\right).\label{eq:NZSdistance}
\end{equation}
Then, clearly as in our original definition, $0\leq\dist_{NZS}\left(u,v\right)\leq2.$\footnote{A very similar approach to closeness of information structures is
taken in \citet{monderer_samet_1996} and \citet{kajmorr:97}. \citet{monderer_samet_1996}
define a notion of distance $d_{MS}$ on common prior information
structures that relies on closeness of common belief events. They
show that if two information structures are $\varepsilon$-close,
then any equilibrium on one of them is $\varepsilon$-close to an
$\varepsilon$-equilibrium on the other structure. \citet{kajmorr:97}
use the last property as their definition. They say that an information
structure is $\varepsilon$-close to another one if, for all bounded
games, any equilibrium of one is $\varepsilon$-close to an $\varepsilon$-equilibrium
of the other. Thus, the main difference between the approach on these
two papers and our metric $\dist_{NZS}$ is that we require $\varepsilon$-closeness
to a (proper) equilibrium. Our metric $\dist_{NZS}$ is closer in
spirit to the value-based distance defined using zero-sum games, and,
arguably, it is easier to interpret for values of $\varepsilon$ that
are far away from 0. }

Contrary to the value in the zero-sum game, the BNE payoffs on information
structure $u$ cannot be factorized through distribution $\tilde{u}\in\Pi$
over the hierarchies of beliefs induced by $u$. For this reason,
we only restrict our analysis to information structures that are non-redundant, which means that two different signals (occurring with positive probability) induce tow different hierarchies of beliefs.  
We do so because
the dependence of the BNE on redundant information is not yet well-understood\footnote{See \citet{sadzik2008beliefs}. An alternative approach would be to
take an equilibrium solution concept that can be factorized through
the hierarchies of beliefs. An example is Bayes Correlated Equilibrium
from \citet{bergemann_bayes_2015}.}.
For convenience, we also restrict ourselves to information structures with finite support.

Let $u\in\Delta\left(K\times C\times D\right)$ be an information
structure with finite support. A subset $A\subseteq K\times C\times D$ is a \emph{proper
common knowledge component} if $u\left(A\right)\in\left(0,1\right)$
and for each signal $s\in C\cup D$, $u\left(A|s\right)\in\left\{ 0,1\right\}$.
An information structure is \emph{simple} if it does not have a proper
common knowledge component. As it follows from Lemma III.2.7 in \citet{mertens_sorin_zamir_2015}, each non-redundant information structure $u$ with finite support has a representation as a finite convex combination of (non-redundant) simple information structures\footnote{Let us sketch the argument  for this result. For each signal (type) $s$ of a player in the support of $u$, we can define $N_1(s)$ as the support of $u(.|s)$. Then, we repeat the construction for every signal in $N_1(s)$ and define $N_2(s)$ as the union of $N_1(s)$ and all the sets obtained this way, $N_2(s)=N_1(s)\cup (\cup_{\tilde s \in N_1(s)} N_1(\tilde s))$. Repeating this process, the sequence is eventually stationary, i.e. $N_{t+1}(s)=N_t(s)$ for some integer $t$. We obtain a finite set $N(s)=N_t(s)$ having the property that the conditional distribution of $u(.|N(s))$ is a (non-redundant) simple information structure with support $N(s)$. There are finitely many different sets $N(s)$ when $s$ ranges through all signals in the support of $u$ and they form a partition of the support of $u$. The representation as a convex combination follows directly from the construction.} $u=\sum_{\alpha}p_{\alpha}u_{\alpha}$, where $\sum p_{\alpha}=1,p_{\alpha}\geq0$.

\begin{thm}\label{THM:NZS DISTANCE}Suppose that $u,v$ are non-redundant
information structures with finite support. If $u$ and $v$ are simple, then 
\begin{align*}
\dist_{NZS}\left(u,v\right) & =\begin{cases}
0, & \text{if }\tilde{u}=\tilde{v},\\
2 & \text{otherwise. }
\end{cases}
\end{align*}
More generally, suppose that $u=\sum p_{\alpha}u_{\alpha}$ and $v=\sum q_{\alpha}v_{\alpha}$
are the decompositions into simple information structures. We can
always choose the decompositions so that $\tilde{u}_{\alpha}=\tilde{v}_{\alpha}$
for each $\alpha$. Then, 
\[
\dist_{NZS}\left(u,v\right)=\sum_{\alpha}\left|p_{\alpha}-q_{\alpha}\right|.
\]
\end{thm}

The distance between the two non-redundant simple information structures
is binary, either 0 if the information structures are equivalent,
or 2 if they are not. In particular, the distance between all simple
information structures that do not have the same hierarchies of beliefs
is trivially equal to its maximum possible value 2. The distance between
two non-redundant but not necessarily simple information structures,
$\dist_{NZS}$, depends on the similarity of their simple components
after decomposition. Theorem \ref{THM:NZS DISTANCE} implies that
(\ref{eq:NZSdistance}) is too fine a measure of distance between
information structures to be useful.

The proof in the case of two non-redundant and simple structures $u$
and $v$ is straightforward. Let $\tilde{u}\neq\tilde{v}$. Earlier results have shown that there exists a finite game $g:K\times I\times J\rightarrow[-1,1]^{2}$
in which each type of player 1 in the support of $\tilde{u}$ and
$\tilde{v}$ reports her hierarchy of beliefs as the unique rationalizable
action (see lemma 4 in \cite{dekel_topologies_2006} and lemma 11 in \cite{ely_critical_2011}).
Second, Lemma III.2.7 in \citet{mertens_sorin_zamir_2015}
(or Corollary 4.7 in \citet{mertzam:85}) shows that the supports
of distributions $\tilde{u}$ and $\tilde{v}$ must be disjoint (it
is also a consequence of the result in \citet{samet1998iterated}).
Therefore, we can construct a game in which, additionally to the first
game, player 2 chooses between two actions $\left\{ u,v\right\} $
and it is optimal for her to match the information structure to which
player 1's reported type belongs. Finally, we multiply the resultant
game by $\varepsilon>0$ and construct a new game in which, additionally,
player 1 receives payoff $1-\varepsilon$ if player 2 chooses $u$
and a payoff of $-1+\varepsilon$ if player 2 chooses $v$. Hence,
the payoff distance between the two information structures is at least
$2-\varepsilon$, where $\varepsilon$ is arbitrarily small. The resultant
game has a BNE in the unique rationalizable profile. \footnote{This construction relies creating new games by adding externality
to payoffs of one player that depend only on the actions of the other
player. Such techniques are available with non-zero-sum games, but
not with zero-sum games. We are grateful to a referee for pointing
it out.}

\section{Conclusion}

In this paper, we have introduced and analyzed value-based distance
on the space of information structures. The main advantage of the
definition is that it has a simple and useful interpretation as the
tight upper bound on the loss or gain from moving between two information
structures. This allows us to apply it directly to numerous questions
about the value of information, the relation between games and single-agent
problems, a comparison of information structures, etc. Additionally,
we show that the distance contains interesting topological information.
On the one hand, the topology induced on countable information structures
is equivalent to the topology of weak convergence of consistent probabilities
over coherent hierarchies of beliefs. On the other hand, the set of
countable information structures is not entirely bounded for value-based
distance, which negatively solves the last open question raised in
\citet{mertens_repeated_1986}, with deep implications for stochastic
games.

By restricting our attention to zero-sum games, we were able to re-examine
the relevance of many phenomena observed and discussed in the strategic
discontinuities literature. While the distinction between approximate
knowledge and approximate common knowledge is not important in situations
of conflict, higher order beliefs may matter on some potentially uncountably
large structures. More generally, we believe that the discussion of
the strategic phenomena on particular classes of games can be a fruitful
line of future research. It is not the case that each problem must
involve coordination games. Interesting classes of games to study
could be common interest games, potential games, etc.\footnote{As an example of work in this direction, \citet{yamashita2018} studies
an order on hierarchies and types induced by payoffs in supermodular
games.}

\appendix

\section{Proof of Theorem 1}

The proof of Theorem 1 relies on two main aspects: the two interpretations
of garbling (deterioration of signals and strategy) and the minmax
theorem.

\emph{Part 1.} We start with general considerations and first identify
payoff functions with particular infinite matrices. For $1\leq L<\infty$,
let $G(L)$ be the set of maps from $K\times\N\times\N$ to $[-1,1]$
such that $g(k,i,j)=-1$ if $i\geq L,j<L$, $g(k,i,j)=1$ if $i<L,j\geq L$,
and $g(k,i,j)=0$ if $i>L,j>L$. Elements in $G(L)$ correspond to
payoff functions with action set $\N$ for each player, with any strategy
$\geq L$ that is weakly dominated. We define $G=G(\infty)=\bigcup_{L\geq1}G(L)$;
for each $u$, $v$ in $\CU$ the values $\val(u,g)$ and $\val(v,g)$
are well defined; and $\dist(u,v)=\sup_{g\in G}|\val(u,g)-\val(v,g)|$.

For $u\in\CU$ and $g\in G$, let $\gamma_{u,g}(q_{1},q_{2})$ denote
the payoff of player 1 in the zero-sum game $\Gamma(u,g)$ when player
1 plays $q_{1}\in\CQ$ and player 2 plays $q_{2}\in\CQ$. Extending
$g$ to mixed actions, as usual, we have $\gamma_{u,g}(q_{1},q_{2})=\sum_{k,c,d}u(k,c,d)g(k,q_{1}(c),q_{2}(d))$.
Notice that the scalar product $\langle g,u\rangle=\sum_{k,c,d}g(k,c,d)u(k,c,d)$
is well defined and corresponds to payoff $\gamma_{u,g}(Id,Id)$,
where $Id\in\CQ$ is the strategy that plays the signal received with
probability one. A straightforward computation leads to $\gamma_{u,g}(q_{1},q_{2})=\langle g,q_{1}.u.q_{2}\rangle$.
Consequently, 
\[
\val(u,g)=\max_{q_{1}\in\CQ}\min_{q_{2}\in\CQ}\langle g,q_{1}.u.q_{2}\rangle=\min_{q_{2}\in\CQ}\max_{q_{1}\in\CQ}\langle g,q_{1}.u.q_{2}\rangle.
\]
For $L=1,2,...,+\infty$ and $g\in G(L)$, the max and min can be
obtained by elements of $\CQ(L)$. Since both players can play the
$Id$ strategy in $\Gamma({u,g})$, we have for all $u\in\CU$ and
$g\in G(L)$ that $\inf_{q_{2}\in\CQ(L)}\langle g,u.q_{2}\rangle\leq\val(u,g)\leq\sup_{q_{1}\in\CQ(L)}\langle g,q_{1}.u\rangle$.
Notice also that for all $u$, $v$ in $\CU(L)$, $\|u-v\|=\sup_{g\in G(L)}\langle g,u-v\rangle$.

\emph{Part 2.} We now prove Theorem \ref{thm1}. Fix $u$, $v$ in
$\CU(L)$, with $L=1,2,...,+\infty$. For $g\in G(L)$, we have $\inf_{q_{1},q_{2}\in\CQ(L)}\langle g,v.q_{2}-q_{1}.u\rangle\leq\val(v,g)-\val(u,g)$,
so 
\begin{equation}
\sup_{g\in G(L)}\left(\val(v,g)-\val(u,g)\right)\geq\sup_{g\in G(L)}\inf_{q_{1},q_{2}\in\CQ(L)}\langle g,v.q_{2}-q_{1}.u\rangle.\label{eq2}
\end{equation}
For $g\in G$, $q_{1},q_{2}\in\CQ\left(L\right)$, by monotonicity
of the value with respect to information, we have $\val(v.q_{2},g)\geq\val(v,g)\text{ and }\val(u,g)\geq\val(q_{1}.u,g)$.
So $\val(v,g)-\val(u,g)\leq\dist\left(q_{1}.u,v.q_{2}\right)\leq\|q_{1}.u-v.q_{2}\|.$
Therefore, 
\begin{equation}
\sup_{g\in\CG}\left(\val(v,g)-\val(u,g)\right)\leq\inf_{q_{1},q_{2}\in\CQ\left(L\right)}\|q_{1}.u-v.q_{2}\|=\inf_{q_{1},q_{2}\in\CQ\left(L\right)}\sup_{g\in G(L)}\langle g,v.q_{2}-q_{1}.u\rangle.\label{eq3}
\end{equation}
We are now going to show that 
\begin{equation}
\sup_{g\in G(L)}\inf_{q_{1},q_{2}\in\CQ(L)}\langle g,v.q_{2}-q_{1}.u\rangle=\min_{q_{1},q_{2}\in\CQ\left(L\right)}\sup_{g\in G(L)}\langle g,v.q_{2}-q_{1}.u\rangle.\label{eq1}
\end{equation}
Together with inequalities \ref{eq2} and \ref{eq3}, it will give
\[
\sup_{g\in\CG}\left(\val(v,g)-\val(u,g)\right)=\sup_{g\in G(L)}\left(\val(v,g)-\val(u,g)\right)=\min_{q_{1},q_{2}\in\CQ\left(L\right)}\|q_{1}.u-v.q_{2}\|.
\]

To prove \ref{eq1}, we will apply a variant of Sion's theorem (see,
e.g., \citet{mertens_sorin_zamir_2015} Proposition I.1.3) to the
zero-sum game with strategy spaces $G(L)$ for the maximizer, $\CQ(L)^{2}$
for the minimizer, and payoff $h(g,(q_{1},q_{2}))=\langle g,v.q_{2}-q_{1}.u\rangle$.
Strategy sets $G(L)$ and $\CQ\left(L\right)^{2}$ are convex, and
$h$ is bilinear.

Case 1: $L<+\infty$. Then, $\Delta(\{0,...,L-1\})$ is compact, and
$\CQ\left(L\right)^{2}$ is compact for the product topology. Moreover,
$h$ is continuous, so by Sion's theorem, \ref{eq1} holds. Furthermore,
$\sup_{g\in G(L)}(\val(v,g)-\val(u,g))$ is achieved, since $G(L)$
is compact.

Case 2: $L=+\infty$. We are going to modify the topology on $\CQ$
to have compact $\CQ\left(L\right)^{2}$ and lower semi-continuous
$h$ on $(q_{1},q_{2})$. The idea is to identify 0 and $+\infty$
in $\N$. Formally, given $q\in\Delta(\N)$ and a sequence $(q_{n})_{n}$
of probabilities over $\N$, we define: $(q_{n})_{n}$ converges to
$q$ if and only if: $\forall c\geq1$, $\lim_{n\to\infty}q_{n}(c)=q(c)$.
It implies $\limsup_{n}q_{n}(0)\leq q(0)$.

$\Delta(\N)$ is now compact, and we endow $\CQ$ with the product
topology so that $\CQ\left(L\right)^{2}$ is itself compact. Fix $g\in G$.
We finally show that $\langle g,q.u\rangle$ is u.s.c. in $q\in\CQ$
and $\langle g,v.q\rangle$ is l.s.c. in $q\in\CQ$. For this, we
take advantage of the particular structure of $G$: there exists $L'$
such that $g\in G(L')$.

For each $q$ in $\Delta(\N)$, we have for each $k$ in $K$ and
$d$ in $\N$: 
\begin{eqnarray*}
g(k,q,d) & = & \sum_{c\in\N}g(c)g(k,c,d)\\
\, & = & g(k,0,d)+\sum_{c=1}^{L'-1}(g(k,c,d)-g(k,0,d))q(c)+\sum_{c\geq L'}(g(k,c,d)-g(k,0,d))q(c).
\end{eqnarray*}
For each $c\geq L'$, we have $g(k,c,d)-g(k,0,d)\leq0$. If $(q_{n})_{n}$
converges to $q$ for our new topology, $\lim_{n}\sum_{c=1}^{L'-1}(g(k,c,d)-g(k,0,d))q_{n}(c)=\sum_{c=1}^{L'-1}(g(k,c,d)-g(k,0,d))q(c)$
and, by Fatou's lemma, $\limsup_{n}\sum_{c\geq L'}(g(k,c,d)-g(k,0,d))q_{n}(c)\leq\sum_{c\geq L'}(g(k,c,d)-g(k,0,d))q(c)$.
As a consequence, $\limsup_{n}g(k,q_{n},d)\leq g(k,q,d)$. This is
true for each $k$ and $d$, and we easily obtain that $\langle g,q.u\rangle=\sum_{k,c,d}u(k,c,d)g(k,q(c),d)$
is u.s.c. in $q\in\CQ$.

Similarly, for each $q\in\Delta(\N)$, $k\in K$, and $c\in\N$, we
can write $g(k,c,q)=g(k,c,0)+\sum_{d=1}^{L'-1}(g(k,c,d)-g(k,c,0))q(c)+\sum_{d\geq L'}(g(k,c,d)-g(k,c,0))q(c),$
with $g(k,c,d)-g(k,c,0)\geq0$ for $d\geq L'$, and show that $\langle g,v.q\rangle$
is l.s.c. in $q\in\CQ$.

\section{Proofs of Section 4}

\subsection{Proof of Proposition \ref{prop:diameter}}

We prove the lower bound of (\ref{eq:diameter bounds}). Let $g\left(k\right)=\1_{p_{k}>q_{k}}-\1_{p_{k}\leq q_{k}}$.
Then, 
\[
\dist\left(u,v\right)\geq\val\left(u,g\right)-\val\left(v,g\right)=\sum_{k\in K}\left(p_{k}-q_{k}\right)g\left(k\right)=\sum_{k\in K}\left|p_{k}-q_{k}\right|.
\]
Now, let us prove the upper bound of (\ref{eq:diameter bounds}).
Define $\bar{u}$ and $\underline{v}$ in $\Delta(K\times K_{C}\times K_{D})$
with $K=K_{C}=K_{D}$ such that $\bar{u}(k,c,d)=p_{k}\1_{c=k}\1_{d=k_{0}}$
for some fixed $k_{0}\in K$ (complete information for player 1, trivial
information for player 2, and the same prior about $k$ as $u$) and
$\underline{v}(k,c,d)=q_{k}\1_{c=k_{0}}\1_{d=k}$ for all $(k,c,d)$
(trivial information for player 1, complete information for player
2, and the same beliefs about $k$ as $v$). Since the value of a
zero-sum game is weakly increasing with player 1's information and
weakly decreasing with player 2's information, we have 
\[
\sup_{g\in\CG}(\val(u,g)-\val(v,g))\leq\sup_{g\in\CG}(\val(\bar{u},g)-\val(\underline{v},g))=\min_{q_{1}\in\CQ,q_{2}\in\CQ}\|\bar{u}.q_{2}-q_{1}.\underline{v}\|,
\]
where, according to Theorem \ref{thm1}, the minimum in the last expression
is attained for garblings with values in $\Delta K$. Since player
2 has a unique signal in $\bar{u}$, only $q_{2}(.|k_{0})\in\Delta K$
matters. We denote it by $q'=q_{2}\left(.|k_{0}\right)$. Similarly,
we define $p^{\prime}=q_{1}(.|k_{0})\in\Delta(K)$. Then, 
\[
\begin{split}\|\bar{u}.q_{2}-q_{1}.\underline{v}\| & =\sum_{(k,c,d)\in K^{3}}|p_{k}\1_{c=k}q{}_{d}^{\prime}-q_{k}\1_{d=k}p{}_{c}^{\prime}|\\
 & =\sum_{k\in K}|p_{k}q{}_{k}^{\prime}-q_{k}p{}_{k}^{\prime}|+p_{k}(1-q{}_{k}^{\prime})+q_{k}(1-p{}_{k}^{\prime})\\
 & =2+\sum_{k\in K}|p_{k}q_{k}^{\prime}-q_{k}p_{k}^{\prime}|-p_{k}q_{k}^{\prime}-q_{k}p_{k}^{\prime}=2\left(1-\sum_{k\in K}\min\left(p_{k}q_{k}^{\prime},q_{k}p_{k}^{\prime}\right)\right).
\end{split}
\]
A similar inequality holds by inverting the roles of $u$ and $v$,
and the upper bound follows from the fact that one can choose $p^{\prime},q^{\prime}$
arbitrarily.

If $p=q$, then $\sum_{k\in K}\min\left(p_{k}q_{k}^{\prime},q_{k}p_{k}^{\prime}\right)=\sum_{k\in K}p_{k}\min\left(q_{k}^{\prime},p_{k}^{\prime}\right)\leq\sum_{k\in K}p_{k}p_{k}^{\prime}\leq\max_{k\in K}p_{k}$,
where the latter is attained by $p_{k}^{\prime}=q_{k}^{\prime}=\1_{\{k=k^{*}\}}$
for some $k^{*}\in K$ such that $p_{k^{*}}=\max_{k\in k}p_{k}$.

\subsection{Proof of Proposition \ref{prop: SA=00003D00003D00003D00003D00003D00003D00003D00003D00003D00003DGames}}

Let us start with general properties of $\dist_{1}$. Let us define
the set of single-agent information structures as $\CU_{1}=\Delta(K\times\N)$
using the same convention that countable sets are identified with
subsets of $\N$. Note that, given $u\in\Delta(K\times C\times D)$,
$\marg_{K\times C}u\in\CU_{1}$. Let $\CG'_{1}=\{g':K\times I\rightarrow\R\,|\,I\text{ finite}\}$
be the set of single-agent decision problems, and define for $u',v'\in\CU_{1}$,
$\dist'_{1}(u',v')=\sup_{g'\in\CG'_{1}}|\val(v',g')-\val(u',g')|$.
It is easily seen that, for any $u,v\in\Delta(K\times C\times D)$,
\begin{equation}
\dist_{1}(u,v)=\dist'_{1}(u',v')=\max\{\min_{q\in\CQ}\|u'-q.v'\|,\min_{q\in\CQ}\|q.u'-v'\|\},\label{eq:d_1_marg}
\end{equation}
where $u'=\marg_{K\times C}u$, $v'=\marg_{K\times C}v$, $q.u'(k,c)=\sum_{s\in C}u'(k,s)q(s)(c)$
and where the last equality can be obtained by mimicking (and simplifying)
the arguments of the proof of Theorem \ref{thm1}.

We now prove Proposition \ref{prop: SA=00003D00003D00003D00003D00003D00003D00003D00003D00003D00003DGames}.
Using the assumptions, we have $u(k)=v(k)$, $u\left(c,d|k\right)=u\left(c|k\right)u\left(d|k\right)$,
and $v\left(c^{\prime},d|k\right)=v\left(d|k\right)v\left(c^{\prime}|k\right)=u(d|k)v\left(c^{\prime}|k\right)$.
For any pair of garblings $q_{1},q_{2}$, 
\[
\begin{split}\left\Vert u.q_{2}-q_{1}.v\right\Vert  & =\sum_{k,c,d}\left|\sum_{\beta}u\left(k,c,\beta\right)q_{2}\left(d|\beta\right)-\sum_{\alpha}v\left(k,\alpha,d\right)q_{1}\left(c|\alpha\right)\right|\\
 & =\sum_{k,c}u\left(k\right)\sum_{d}\left|u\left(c|k\right)\sum_{\beta}u\left(\beta|k\right)q_{2}\left(d|\beta\right)-\left(\sum_{\alpha}v\left(\alpha|k\right)q_{1}\left(c|\alpha\right)\right)u\left(d|k\right)\right|\\
 & =\sum_{k,c}u\left(k\right)\sum_{d}\left|u\left(d|k\right)\Gamma\left(k,c\right)+\Delta\left(k,d\right)u\left(c|k\right),\right|
\end{split}
\]
where $\Delta\left(k,d\right)=u\left(d|k\right)-\sum_{\beta}u\left(\beta|k\right)q_{2}\left(d|\beta\right)$,
and $\Gamma\left(k,c\right)=\sum_{\alpha}v\left(\alpha|k\right)q_{1}\left(c|\alpha\right)-u\left(c|k\right)$.
Because $\left|x+y\right|\geq|x|+\text{sgn}(x)y$ for each $x,y\in\R$,
we have 
\[
\begin{split} & \sum_{d}\left|u\left(d|k\right)\Gamma\left(k,c\right)+\Delta\left(k,d\right)u\left(c|k\right)\right|\\
\geq & \sum_{d}u\left(d|k\right)\left|\Gamma\left(k,c\right)\right|+\text{sgn}\left(\Gamma\left(k,c\right)\right)u\left(c|k\right)\sum_{d}\Delta\left(k,d\right)=\sum_{d}u\left(d|k\right)\left|\Gamma\left(k,c\right)\right|.
\end{split}
\]
where the last equality comes from the fact that $\sum_{d}\Delta\left(k,d\right)=0$.
Therefore, we obtain 
\[
\begin{split}\left\Vert u.q_{2}-q_{1}.v\right\Vert  & \geq\sum_{k,c,d}u\left(k\right)\left|u\left(d|k\right)\Gamma\left(k,c\right)\right|\\
 & =\sum_{k,c,d}u\left(k\right)\left|u\left(d|k\right)u\left(c|k\right)-\sum_{\alpha}u\left(d|k\right)v\left(\alpha|k\right)q_{1}\left(c|\alpha\right)\right|=\left\Vert u-q_{1}.v\right\Vert .
\end{split}
\]

We deduce that $\min_{q_{1},q_{2}}\left\Vert u.q_{2}-q_{1}.v\right\Vert =\min_{q_{1}}\left\Vert u-q_{1}.v\right\Vert .$
Inverting the roles of the players, we also have $\min_{q_{1},q_{2}}\left\Vert v.q_{2}-q_{1}.y\right\Vert =\min_{q_{1}}\left\Vert v-q_{1}.u\right\Vert .$
We conclude that 
\begin{align*}
\dist(u,v) & =\max\{\min_{q_{1},q_{2}}\left\Vert u.q_{2}-q_{1}.v\right\Vert ;\min_{q_{1},q_{2}}\left\Vert v.q_{2}-q_{1}.y\right\Vert \}\\
 & =\max\{\min_{q_{1}}\left\Vert u-q_{1}.v\right\Vert ;\min_{q_{1}}\left\Vert v-q_{1}.u\right\Vert \}=\dist_{1}(u,v),
\end{align*}
where the last equality follows from (\ref{eq:d_1_marg}) together
with the fact that $\marg_{K\times D}u=\marg_{K\times D}v$.

\subsection{Proof of Proposition \ref{prop. Info substitutes}}

Because $u\succeq v$, 
\[
\dist\left(u,v\right)=\min_{q_{2}\in\CQ}\min_{q_{1}\in\CQ}\left\Vert u.q_{2}-q_{1}.v\right\Vert \leq\min_{q_{2}\in\CQ}\min_{q_{1}:C\rightarrow\Delta\left(C\times C_{2}\right)}\left\Vert u.q_{2}-\hat{q}_{1}.v\right\Vert ,
\]
where in the right-hand side of the inequality, we use a restricted
set of player 1's garblings. Precisely, for every garbling $q_{1}:C\rightarrow\Delta\left(C\times C_{2}\right)$,
we associate garbling $\hat{q}_{1}$ defined by $\hat{q}_{1}(c',c'_{1},c'_{2}|c,c_{1})=\1_{\{c_{1}\}}(c'_{1})q_{1}(c',c'_{2}|c)$.
Further, for any such $q_{1}$ and an arbitrary garbling $q_{2}$,
we have 
\[
\begin{split} & \left\Vert u.q_{2}-\hat{q}_{1}.v\right\Vert =\sum_{k,c,c_{1},c_{2},d}\left|\sum_{\beta}u\left(k,c,c_{1},c_{2},\beta\right)q_{2}\left(d|\beta\right)-\sum_{\alpha}u\left(k,\alpha,c_{1},d\right)q_{1}\left(c,c_{2}|\alpha\right)\right|\\
 & \;\;=\sum_{k,c,c_{1},c_{2},d}u\left(k,c_{1})\right)\left|\sum_{\beta}u\left(c,c_{2},\beta|k,c_{1}\right)q_{2}\left(d|\beta\right)-\sum_{\alpha}u\left(\alpha,d|k,c_{1}\right)q_{1}\left(c,c_{2}|\alpha\right)\right|.
\end{split}
\]
Because of the conditional independence assumption, the above is equal
to 
\[
\begin{split}= & \sum_{k,c,c_{2},d}\left(\sum_{c_{1}}u\left(k,c_{1}\right)\right)\left|\sum_{\beta}u\left(c,c_{2},\beta|k\right)q_{2}\left(d|\beta\right)-\sum_{\alpha}u\left(\alpha,d|k\right)q_{1}\left(c,c_{2}|\alpha\right)\right|\\
= & \sum_{k,c,c_{2},d}\left|\sum_{\beta}u\left(k,c,c_{2},\beta\right)q_{2}\left(d|\beta\right)-\sum_{\alpha}u\left(k,\alpha,d\right){q}_{1}\left(c,c_{2}|\alpha\right)\right|=\left\Vert u^{\prime}.q_{2}-{q}_{1}.v{}^{\prime}\right\Vert .
\end{split}
\]
Therefore, $\dist\left(u,v\right)\leq\min_{q_{2}}\min_{q_{1}:C\rightarrow\Delta\left(C\times C_{2}\right)}\left\Vert u^{\prime}.q_{2}-q_{1}.v{}^{\prime}\right\Vert =\dist\left(u^{\prime},v{}^{\prime}\right)$.

\subsection{Proof of Proposition \ref{prop. Info complements}}

We have $\dist\left(u^{\prime},v{}^{\prime}\right)=\dist_{1}\left(u^{\prime},v{}^{\prime}\right)=\dist_{1}\left(u,v\right)\leq\dist\left(u,v\right).$
The first equality comes from Proposition \ref{prop: Games vs Single agent};
the second comes from the fact that $u$ and $u'$ ($v$ and $v'$,
respectively) induce the same distribution on player $1$'s first-order
beliefs, and the inequality from the definition of the two distances.

\subsection{Proof of Proposition \ref{prop: joint info}}

It is sufficient to show that, if $c_{1}$ is $\varepsilon$-conditionally
independent from $\left(k,d\right)$ given $c$, then $\sup_{g\in\CG}\val\left(u,g\right)-\val\left(v,g\right)\leq\varepsilon.$

For this, let $q_{2}:D\times D_{1}\rightarrow D$ be defined as $q_{2}\left(d,d_{1}\right)(d')=\1_{d'=d}$.
Let $q_{1}:C\rightarrow C\times C_{1}$ be defined as $q_{1}\left(c,c_{1}|c\right)=u\left(c_{1}|c\right)$.
Then, 
\begin{align*}
\left\Vert u.q_{2}-q_{1}.v\right\Vert  & =\sum_{k,c,c_{1},d}\left|u\left(k,c,c_{1},d\right)-u\left(k,c,d\right)u\left(c_{1}|c\right)\right|\\
 & =\sum_{c}u\left(c\right)\sum_{k,c_{1},d}\left|u\left(k,c_{1},d|c\right)-u\left(k,d|c\right)u\left(c_{1}|c\right)\right|\leq\varepsilon.
\end{align*}
The claim follows from Theorem \ref{thm1}.

\section{Proof of Theorem \ref{THM3LARGESPACE}}

$N$ is a very large even-valued integer to be fixed later, and we
write $A=C=D=\{1,...,N\}$, with the idea of using $C$ while speaking
of the actions or signals of player 1 and using $D$ while speaking
of the actions and signals of player 2. We fix $\varepsilon$ and
$\alpha$, to be used later, such that $0<\varepsilon<\frac{1}{10(N+1)^{2}}\text{ and }\alpha=\frac{1}{25}$.
We will consider a Markov chain with law $\nu$ on $A$, satisfying
the following:

$\bullet$ the law of the first state of the Markov chain is uniform
on $A$;

$\bullet$ given the current state, the law of the next state is uniform
on a subset of size $N/2$;

$\bullet$ and a few more conditions, to be defined later.

A sequence $(a_{1},...,a_{l})$ of length $l\geq1$ is said to be
\textit{nice} if it is in the support of the Markov chain: $\nu(a_{1},...,a_{l})>0$.
For instance, any sequence of length 1 is nice, and $N^{2}/2$ sequences
of length 2 are nice.

The remainder of the proof is split into 3 parts: we first define
information structures $(u^{l})_{l\geq1}$ and payoff structures $(g^{p})_{p\geq1}$.
Then, we define two conditions $UI1$ and $UI2$ on the information
structures and show that they imply the conclusions of Theorem \ref{THM3LARGESPACE}.
Finally, we show, via the probabilistic method, the existence of a
Markov chain $\nu$ satisfying all our conditions.

\subsection{Information and payoff structures $(u^{l})_{l\protect\geq1}$ and
$(g^{l})_{l\protect\geq1}$}

For $l\geq1$, define the information structure $u^{l}\in\Delta(K\times C^{l}\times D^{l})$
so that for each state $k$ in $K$, signal $c=(c_{1},...,c_{l})$
in $C^{l}$ of player 1 and signal $d=(d_{1},...,d_{l})$ in $D^{l}$
for player 2, 
\[
u^{l}(k,c,d)=\nu(c_{1},d_{1},c_{2},d_{2},...,c_{l},d_{l})\left(\frac{c_{1}}{N+1}\mathbf{1}_{k=1}+\left(1-\frac{c_{1}}{N+1}\right)\mathbf{1}_{k=0}\right).
\]
The following interpretation of $u^{l}$ holds: first select $(a_{1},a_{2},...,a_{2l})=(c_{1},d_{1},...,c_{l},d_{l})$
in $A^{2l}$ according to Markov chain $\nu$ (i.e., uniformly among
the nice sequences of length $2l$), then tell $(c_{1},c_{2},...,c_{l})$
(the elements of the sequence with odd indices) to player 1 and $(d_{1},d_{2},...,d_{l})$
(the elements of the sequence with even indices) to player 2. Finally,
choose state $k=1$ with probability $c_{1}/(N+1)$ and state $k=0$
with the complement probability $1-c_{1}/(N+1)$.

Notice that the definition is not symmetric among players: player
1's first signal $c_{1}$ is uniformly distributed and plays a particular
role. The marginal of $u^{l}$ on $K$ is uniform, and the marginal
of $u^{l+1}$ over $(K\times C^{l}\times V^{l})$ is equal to $u^{l}$.

Consider a sequence of elements $(a_{1},...,a_{l})$ of $A$ that
is not nice (i.e., such that $\nu(a_{1},...,a_{l})=0$). We say that
the sequence is\emph{ not nice because of player 1} if $\min\{t\in\{1,...,l\},\nu(a_{1},...,a_{t})=0\}$
is odd and \emph{not nice because of player 2} if $\min\{t\in\{1,...,l\},\nu(a_{1},...,a_{t})=0\}$
is even. Sequence $(a_{1},...,a_{l})$ is now nice, or not nice because
of player 1, or not nice because of player 2. A sequence of length
2 is either nice, or not nice because of player 2.

For $p\geq1$, define payoff structure $g^{p}:K\times C^{p}\times D^{p-1}\to[-1,1]$
such that, for all $k$ in $K$, $c'=(c'_{1},...,c'_{p})$ in $C^{p}$,
$d'=(d'_{1},...,d'_{p-1})$ in $D^{p-1}$: 
\begin{eqnarray*}
g^{p}(k,c',d') & = & g_{0}(k,c'_{1})+h^{p}(c',d'),\text{ where }g_{0}(k,c'_{1})=-{\left(k-\frac{c'_{1}}{N+1}\right)}^{2}+\frac{N+2}{6(N+1)},\\
h^{p}(c',d') & = & \left\{ \begin{array}{ccl}
\varepsilon & \mbox{if} & (c'_{1},d'_{1},...,c'_{p})\text{ is nice},\\
5\varepsilon & \mbox{if} & (c'_{1},d'_{1},...,c'_{p})\text{ is not nice because of player 2,}\\
-5\varepsilon & \mbox{if} & (c'_{1},d'_{1},...,c'_{p})\text{ is not nice because of player 1}.
\end{array}\right.
\end{eqnarray*}

One can check that $|g^{p}|\leq5/6+5\varepsilon\leq8/9$. Regarding
the $g_{0}$ part of the payoff, consider a decision problem for player
1 where $c_{1}$ is selected uniformly in $A$ and the state is selected
to be $k=1$ with probability $c_{1}/(N+1)$ and $k=0$ with probability
$1-c_{1}/(N+1)$. Player 1 observes $c_{1}$ but not $k$, and she
chooses $c'_{1}$ in $A$ and receives payoff $g_{0}(k,c'_{1})$.
We have $\frac{c_{1}}{N+1}g_{0}(1,c'_{1})+(1-\frac{c_{1}}{N+1})g_{0}(0,c'_{1})$
$=$ $\frac{1}{(N+1)^{2}}(c'_{1}(2c_{1}-c'_{1})+(N+1)((N+2)/6-c_{1}))$.
To maximize this expected payoff, it is well known that player 1 should
play her belief on $k$, i.e. $c'_{1}=c_{1}$. Moreover, if player
1 chooses $c'_{1}\neq c_{1}$, her expected loss from not having chosen
$c_{1}$ is at least $\frac{1}{(N+1)^{2}}\geq10\varepsilon$. Furthermore,
the constant $\frac{N+2}{6(N+1)}$ has been chosen such that the value
of this decision problem is 0.

Consider now $l\geq1$ and $p\geq1$. By definition, the Bayesian
game $\Gamma(u^{l},g^{p})$ is played as follows: first, $(c_{1},d_{1},...,c_{l},d_{l})$
is selected according to law $\nu$ of the Markov chain, player 1
learns $(c_{1},...,c_{l})$, player 2 learns $(d_{1},...,d_{l})$,
and the state is $k=1$ with probability $c_{1}/(N+1)$ and $k=0$
otherwise. Then, player 1 chooses $c'$ in $C^{p}$ and player 2 chooses
$d'$ in $D^{p-1}$ \textit{simultaneously}, and finally, player 1's
payoff is $g^{p}(k,c',d')$. Notice that, by the previous paragraph
about $g_{0}$, it is always strictly dominant for player 1 to truthfully
report her first signal, i.e. choose $c'_{1}=c_{1}$. We will show
in the next section that if $l\geq p$ and player 1 simply plays the
sequence of signals she has received, player 2 cannot do better than
also truthfully reporting his own signals, leading to a value not
lower than the payoff for nice sequences, which is $\varepsilon$.
On the contrary, in game $\Gamma(u^{l},g^{l+1})$, player 1 has to
report not only the $l$ signals she has received but also an extra-signal
$c'_{l+1}$ that she has to guess. In this game, we will prove that,
if player 2 truthfully reports his own signals, player 1 will incur
payoff of $-5\varepsilon$ with a probability of at least (approximately)
1/2, and this will result in a low value. These intuitions will prove
correct in the next section, under conditions $UI1$ and $UI2$.

\subsection{Conditions UI and values}

To prove that the intuition of the previous paragraph is correct,
we need to ensure that players have incentives to report their true
signals, so we need additional assumptions on the Markov chain. \\

\noindent \textbf{Notations and definition:} Let $l\geq1$, $m\geq0$,
$c=(c_{1},...,c_{l})$ in $C^{l}$, and $d=(d_{1},...,d_{m})$ in
$D^{m}$. We write

\centerline{%
\begin{tabular}{lcll}
$a^{2q}(c,d)$  & =  & $(c_{1},d_{1},....,c_{q},d_{q})\in A^{2q}$  & {for each } $q\leq\min\{l,m\}$,\tabularnewline
$a^{2q+1}(c,d)$  & =  & $(c_{1},d_{1},....,c_{q},d_{q},c_{q+1})\in A^{2q+1}$  & { for each }$q\leq\min\{l-1,m\}$.\tabularnewline
\end{tabular}} For $r\leq\min\{2l,2m+1\}$, we say that $c$ and $d$ are \emph{nice
at level }$r$, and we write $c\smile_{r}d,$ if $a^{r}(c,d)$ is
nice.

In the next definition, we consider information structure $u^{l}\in\Delta(K\times C^{l}\times D^{l})$
and let ${\tilde{c}}$ and ${\tilde{d}}$ denote the respective random
variables of the signals of player 1 and player 2. \begin{defn} \label{defUI}We
say that the \emph{conditions $UI1$ are satisfied} if for all $l\geq1$,
all ${c}=({c}_{1},...,{c}_{l})$ in $C^{l}$ and $c^{\prime}=(c_{1}^{\prime},...,c_{l+1}^{\prime})$
in $C^{l+1}$ such that ${c}_{1}=c_{1}^{\prime}$, we have 
\begin{equation}
u^{l}\left(c^{\prime}\smile_{2l+1}\tilde{d}\;\big|\;\tilde{c}={c},c^{\prime}\smile_{2l}\tilde{d}\right)\in[1/2-\alpha,1/2+\alpha]\label{eq61}
\end{equation}
and for all $m\in\{1,...,l\}$ such that ${c}_{m}\neq c_{m}^{\prime}$,
for $r=2m-2,2m-1$, 
\begin{equation}
u^{l}\left(c^{\prime}\smile_{r+1}\tilde{d}\;\big|\;\tilde{c}=c,c^{\prime}\smile_{r}\tilde{d}\right)\in[1/2-\alpha,1/2+\alpha].\label{eq62}
\end{equation}
\end{defn} We say that the\emph{ conditions $UI2$ are satisfied}
if for all $1\leq p\leq l$, for all ${d}\in D^{l}$, for all $d^{\prime}\in D^{p-1}$,
for all $m\in\{1,...,p-1\}$ such that ${d}_{m}\neq d_{m}^{\prime}$,
for $r=2m-1,2m$ 
\begin{equation}
u^{l}\left(\tilde{c}\smile_{r+1}d^{\prime}|\tilde{d}={d},\tilde{c}\smile_{r}d^{\prime}\right)\in[1/2-\alpha,1/2+\alpha].\label{eq63}
\end{equation}

To understand the conditions $UI1$, consider the Bayesian game $\Gamma(u^{l},g^{l+1})$,
and assume that player 2 truthfully reports his sequence of signals
and that player 1 has received signals $(c_{1},...,c_{l})$ in $C^{l}$.
Equation (\ref{eq61}) states that, if the sequence of reported signals
$(c'_{1},\tilde{d}_{1},...,c'_{l},\tilde{d}_{l})$ is nice at level
$2l$, then whatever the last reported signal $c'_{l+1}$ is, the
conditional probability that $(c'_{1},\tilde{d}_{1},...,c'_{l},\tilde{d}_{l},c'_{l+1})$
is still nice is in $[1/2-\alpha,1/2+\alpha]$, (i.e., close to 1/2).
Regarding (\ref{eq62}), first notice that if $c'=c$, then by construction
$(c'_{1},\tilde{d}_{1},...,c'_{l},\tilde{d}_{l})$ is nice and $u^{l}\left(c^{\prime}\smile_{r+1}\tilde{d}\;\big|\;\tilde{c}=c,c^{\prime}\smile_{r}\tilde{d}\right)=u^{l}\left(c\smile_{r+1}\tilde{d}\;\big|\;\tilde{c}=c\right)=1$
for each $r=1,...,2l-1$. Assume now that, for some $m=1,...,l$,
player 1 misreports her $m^{th}$-signal (i.e., reports $c'_{m}\neq c_{m}$).
Equation (\ref{eq62}) requires that, given that the reported signals
were nice thus far (at level $2m-2$), the conditional probability
that the reported signals are not nice at level $2m-1$ (integrating
$c'_{m}$) is close to 1/2, and moreover, if the reported signals
are nice at level $2m-1$, adding the next signal $\tilde{d}_{m}$
for player 2 has a probability close to 1/2 of keeping the reported
sequence nice. Conditions $UI2$ have a similar interpretation, considering
the Bayesian games $\Gamma(u^{l},g^{p})$ for $p\leq l$, assuming
that player 1 truthfully reports her signals and that player 2 plays
$d'$ after having received $d$ signals. \begin{prop} \label{pro2}
Conditions $UI1$ and $UI2$ imply 
\begin{eqnarray}
\forall l\geq1,\forall p\in\{1,...,l\}, & \val(u^{l},g^{p})\geq\varepsilon.\label{eq64}\\
\forall l\geq1, & \val(u^{l},g^{l+1})\leq-\varepsilon.\label{eq66}
\end{eqnarray}
\end{prop} As a consequence of this proposition, under the existence
of a Markov chain satisfying conditions $UI1$ and $UI2$, we obtain
Theorem \ref{THM3LARGESPACE}: 
\[
{\rm If}\;l>p,\;\;{\rm then}\;\;d(u^{l},u^{p})\geq\val(u^{l},g^{p+1})-\val(u^{p},g^{p+1})\geq2\varepsilon.
\]

\textbf{Proof of Proposition \ref{pro2}.} We assume that $UI1$ and
$UI2$ hold. We fix $l\geq1$, work on probability space $K\times C^{l}\times D^{l}$
equipped with probability $u^{l}$, and let $\tilde{c}$ and $\tilde{d}$
denote the random variables of the signals received by the players.

\noindent 1) We first prove (\ref{eq64}). Consider the game $\Gamma(u^{l},g^{p})$
with $p\in\{1,...,l\}$. We assume that player 1 chooses the truthful
strategy. Fix $d=(d_{1},...,d_{l})$ in $D^{l}$ and $d'=(d'_{1},...,d'_{p-1})$
in $D^{p-1}$, and assume that player 2 has received signal $d$ and
chooses to report $d'$. Define the non-increasing sequence of events:
$A_{n}=\{\tilde{c}\smile_{n}d'\}.$ We will prove by backward induction
that 
\begin{equation}
\forall n=1,...,p,\;\;\mathbb{E}[h^{p}(\tilde{c},d')|\tilde{d}=d,A_{2n-1}]\geq\varepsilon.\label{eq65}
\end{equation}

If $n=p$, $h^{p}(\tilde{c},d')=\varepsilon$ on event $A_{2p-1}$,
implying the result. Assume now that, for some $n$ such that $1\leq n<p$,
we have $\mathbb{E}[h^{p}(\tilde{c},d')|\tilde{d}=d,A_{2n+1}]\geq\varepsilon.$
Since we have a non-increasing sequence of events, $\1{}_{A_{2n-1}}=\1_{A_{2n+1}}+\1_{A_{2n-1}}\1_{A_{2n}^{c}}+\1_{A_{2n}}\1_{A_{2n+1}^{c}},$
so by definition of the payoffs, $h^{p}(\tilde{c},d')\1_{A_{2n-1}}=h^{p}(\tilde{c},d')\1_{A_{2n+1}}+5\varepsilon\1_{A_{2n-1}}\1_{A_{2n}^{c}}-5\varepsilon\1_{A_{2n}}\1_{A_{2n+1}^{c}}.$

First, assume that $d'_{n}=d_{n}$. By construction of the Markov
chain, $u^{l}(A_{2n+1}|A_{2n-1},\tilde{d}=d)=1$, implying that $u^{l}(A_{2n+1}^{c}|A_{2n-1},\tilde{d}=d)=u^{l}(A_{2n}^{c}|A_{2n-1},\tilde{d}=d)=0$.
As a consequence, 
\begin{eqnarray*}
\mathbb{E}[h^{p}(\tilde{c},d')|\tilde{d}=d,A_{2n-1}] & = & \mathbb{E}[h^{p}(\tilde{c},d')\1_{A_{2n+1}}|\tilde{d}=d,A_{2n-1}]\\
 & = & \mathbb{E}[\mathbb{E}[h^{p}(\tilde{c},d')|\tilde{d}=d,A_{2n+1}]\1_{A_{2n+1}}|\tilde{d}=d,A_{2n-1}]\geq\varepsilon.
\end{eqnarray*}
Assume now that $d'_{n}\neq d_{n}$. Assumption UI2 implies that 
\begin{eqnarray*}
u^{l}(A_{2n}^{c}|A_{2n-1},\tilde{d}=d) & \geq & 1/2-\alpha,\\
u^{l}(A_{2n}\cap A_{2n+1}^{c}|A_{2n-1},\tilde{d}=d) & \leq & (1/2+\alpha)^{2},\\
u^{l}(A_{2n+1}|A_{2n-1},\tilde{d}=d) & \geq & (1/2-\alpha)^{2}.
\end{eqnarray*}
It follows that 
\[
\begin{split} & \mathbb{E}[h^{p}(\tilde{c},d'|\tilde{d})=d,A_{2n-1}]\\
= & \mathbb{E}[\mathbb{E}[h^{p}(\tilde{c},d')|\tilde{d}=d,A_{2n+1}]\1_{A_{2n+1}}|\tilde{d}=d,A_{2n-1}]\\
 & +5\varepsilon u^{l}(A_{2n}^{c}|A_{2n-1},\tilde{d}=d)-5\varepsilon u^{l}(A_{2n}\cap A_{2n+1}^{c}|A_{2n-1},\tilde{d}=d)\\
\geq & \varepsilon(\frac{1}{4}-\alpha+\alpha^{2})+5\varepsilon(\frac{1}{2}-\alpha)-5\varepsilon(\frac{1}{4}+\alpha+\alpha^{2})=\varepsilon(\frac{3}{2}-11\alpha-4\alpha^{2})\geq\varepsilon,
\end{split}
\]
and (\ref{eq65}) follows by backward induction.

Since $A_{1}$ is an event that holds almost surely, we deduce that
$\mathbb{E}[h^{p}(\tilde{c},d')|\tilde{d}=d]\geq\varepsilon.$ Therefore,
player 1's truthful strategy guarantees payoff $\varepsilon$ in $\Gamma(u^{l},g^{p})$.

2) We now prove (\ref{eq66}). Consider the game $\Gamma(u^{l},g^{l+1})$.
We assume that player 2 chooses the truthful strategy. Fix $c=(c_{1},...,c_{l})$
in $C^{l}$ and $c'=(c'_{1},...,c'_{l-1})$ in $C^{l-1}$, and assume
that player 1 has received signal $c$ and chooses to report $c'$.
We will show that player 1's expected payoff is no larger than $-\varepsilon$,
and assume w.l.o.g. that $c'_{1}=c_{1}$. Consider the non-increasing
sequence of events $B_{n}=\{c^{\prime}\smile_{n}\tilde{d}\,\}.$ We
will prove by backward induction that $\forall n=1,...,l,\;\;\mathbb{E}[h^{l+1}(c^{\prime},\tilde{d})|\tilde{c}=c,B_{2n}]\leq-\varepsilon$.

If $n=l$, we have $\1_{B_{2l}}=\1_{B_{2l+1}}+\1_{B_{2l}}\1_{B_{2l+1}^{c}}$,
and $h^{l+1}(c^{\prime},\tilde{d})\1_{B_{2l}}=\varepsilon\1_{B_{2l+1}}-5\varepsilon\1_{B_{2l}}\1_{B_{2l+1}^{c}}.$
UI1 implies that $|u^{l}(B_{2l+1}|\tilde{c}=c,B_{2l})-\frac{1}{2}|\leq\alpha$
, and it follows that 
\begin{align*}
\mathbb{E}[h^{l+1}(c^{\prime},\tilde{d})|\tilde{c}=c,B_{2l}] & =\varepsilon\,u^{l}(B_{2l+1}|\tilde{c}=c,B_{2l})-5\varepsilon\,u^{l}(B_{2l+1}^{c}|u=\hat{u},B_{2l})\\
 & \leq\varepsilon\,(\frac{1}{2}+\alpha)-5\varepsilon\,(\frac{1}{2}-\alpha)\leq-\varepsilon.
\end{align*}

Assume now that, for some $n=1,...,l-1$, we have $\mathbb{E}[h^{l+1}(c^{\prime},\tilde{d})|\tilde{c}=c,B_{2n+2}]\leq-\varepsilon.$
We have $\1_{B_{2n}}=\1_{B_{2n+2}}+\1_{B_{2n}}\1_{B_{2n+1}^{c}}+\1_{B_{2n+1}}\1_{B_{2n+2}^{c}},$
and by definition of $h^{l+1}$, 
\[
h^{l+1}(c^{\prime},\tilde{d})\1_{B_{2n}}=h^{l+1}(c^{\prime},\tilde{d})\1_{B_{2n+2}}-5\varepsilon\1_{B_{2n}}\1_{B_{2n+1}^{c}}+5\varepsilon\1_{B_{2n+1}}\1_{B_{2n+2}^{c}}.
\]
First, assume that $c_{n+1}^{\prime}=c_{n+1}$, then $u^{l}(B_{2n+2}|B_{2n},\tilde{c}=c)=1$.
Then 
\begin{eqnarray*}
\mathbb{E}[h^{l+1}(c^{\prime},\tilde{d})|\tilde{c}=c,B_{2n}] & = & \mathbb{E}[h^{l+1}(c^{\prime},\tilde{d})\1_{B_{2n+2}}|\tilde{c}=c,B_{2n}],\\
 & = & \mathbb{E}[\mathbb{E}[h^{l+1}(c^{\prime},\tilde{d})|\tilde{c}=c,B_{2n+2}]\1_{B_{2n+2}}|\tilde{c}=c,B_{2n}]\leq-\varepsilon.
\end{eqnarray*}

Assume on the contrary that $c_{n+1}^{\prime}\neq c_{n+1}$. Assumption
UI1 implies that 
\begin{eqnarray*}
u^{l}(B_{2n+1}^{c}|B_{2n},\tilde{c}=c) & \geq & 1/2-\alpha,\\
u^{l}(B_{2n+1}\cap B_{2n+2}^{c}|B_{2n},\tilde{c}=c) & \leq & (1/2+\alpha)^{2},\\
u^{l}(B_{2n+2}|B_{2n},\tilde{c}=c) & \geq & (1/2-\alpha)^{2}.
\end{eqnarray*}
It follows that 
\[
\begin{split}\mathbb{E}[h^{l+1}(c^{\prime},\tilde{d})|\tilde{c}=c,B_{2n}] & =\mathbb{E}[\mathbb{E}[h^{l+1}(c^{\prime},\tilde{d})|\tilde{c}=c,B_{2n+2}]\1_{B_{2n+2}}|\tilde{c}=c,B_{2n}]\\
 & \quad-5\,\varepsilon\,u^{l}(B_{2n+1}^{c}|B_{2n},\tilde{c}=c)+5\,\varepsilon\,u^{l}(B_{2n+1}\cap B_{2n+2}^{c}|B_{2n},\tilde{c}=c)\\
 & \leq-\,\varepsilon\,(\frac{1}{4}-\alpha+\alpha^{2})-5\,\varepsilon\,(\frac{1}{2}-\alpha)+5\,\varepsilon\,(\frac{1}{4}+\alpha+\alpha^{2})\leq-\varepsilon.
\end{split}
\]
By induction, we obtain $\mathbb{E}[h^{l+1}(c^{\prime},\tilde{d})|\tilde{c}=c,B_{2}]\leq-\varepsilon$.
Since $B_{2}$ holds almost surely here, we get $\mathbb{E}[h^{l+1}(c^{\prime},\tilde{d})|\tilde{c}=c]\leq-\varepsilon,$
showing that player 2's truthful strategy guarantees that the payoff
of the maximizer is less than or equal to $-\varepsilon$, which concludes
the proof.

\subsection{Existence of an appropriate Markov chain\label{subsec:Existence-of-Markov chain}\label{existence}}

Here we conclude the proof of Theorem \ref{THM3LARGESPACE} by showing
the existence of an even-valued integer $N$ and a Markov chain with
law $\nu$ on $A=\{1,...,N\}$ satisfying our conditions

$1)$ the law of the first state of the Markov chain is uniform on
$A$,

$2)$ for each $a$ in $A$, there are exactly $N/2$ elements $b$
in $A$ such that $\nu(b|a)=2/N$, and

$3)$ $UI1$ and $UI2$.

Letting $P=(P_{a,b})_{(a,b)\in A^{2}}$ denote the transition matrix
of the Markov chain, we must prove the existence of $P$ satisfying
$2)$ and $3)$. The proof is nonconstructive and uses the following
probabilistic method, where we select (independently for each $a$
in $A$) the set $\{b\in A,P_{a,b}>0\}$ uniformly among the subsets
of $A$ with cardinal $N/2$. We will show that, when $N$ goes to
infinity, the probability of selecting an appropriate transition matrix
becomes strictly positive and, in fact, converges to 1.

Formally, let $\mathcal{S}_{A}$ denote the collection of all subsets
$S\subseteq A$ with cardinality $\left\vert S\right\vert =\frac{1}{2}N$.
We consider a collection $\left(S_{a}\right)_{a\in A}$ of i.i.d.
random variables uniformly distributed over $\mathcal{S}_{A}$ defined
on probability space $(\Omega_{N},\mathcal{F}_{N},\mathbb{P}_{N})$.
For all $a$, $b$ in $A$, let $X_{a,b}=\1_{\{b\in S_{a}\}}$ and
$P_{a,b}=\frac{2}{N}X_{a,b}$. By construction, $P$ is a transition
matrix satisfying $2)$. Theorem \ref{THM3LARGESPACE} will now follow
from the following proposition.

\begin{prop} \label{pro3} 
\[
\mathbb{P}_{N}\left(\text{ }P\text{ induces a Markov chain satisfying UI1 and UI2 }\right)\xrightarrow[N\to\infty]{}1.
\]
In particular, this probability is strictly positive for all sufficiently
large $N$.

\end{prop}

The remainder of this section is devoted to the proof of Proposition
\ref{pro3}. We start with probability bounds based on Hoeffding's
inequality.

\begin{lem} \label{lemT1}For any $a\neq b,$ each $\gamma>0$ 
\[
\mathbb{P}_{N}\left(\left\vert |S_{a}\cap S_{b}|-\frac{1}{4}N\right\vert \geq\gamma N\right)\leq\frac{1}{2}e^{4}Ne^{-2\gamma^{2}N}.
\]
\end{lem}

\begin{proof} Consider a family of i.i.d. Bernoulli variables $(\widetilde{X}_{i,j})_{i=a,b,\,j\in A}$
of parameter $\frac{1}{2}$ defined on space $(\Omega,\mathcal{F},\mathbb{P})$.
For $i=a,b$, define events $\widetilde{L}_{i}=\{\sum_{j\in A}\widetilde{X}_{i,j}=\frac{N}{2}\}$
and set-valued variables $\widetilde{S}_{i}=\{j\in A\,|\,\widetilde{X}_{i,j}=1\}$.
It is straightforward to check that the conditional law of $(\widetilde{S}_{a},\widetilde{S}_{b})$
given $\widetilde{L}_{a}\cap\widetilde{L}_{b}$ under $\mathbb{P}$
is the same as the law of $(S_{a},S_{b})$ under $\mathbb{P}_{N}$.
It follows that 
\begin{align*}
\mathbb{P}_{N}\left(\left\vert |S_{a}\cap S_{b}|-\frac{1}{4}N\right\vert \geq\gamma N\right) & =\mathbb{P}\left(\left\vert |\widetilde{S}_{a}\cap\widetilde{S}_{b}|-\frac{1}{4}N\right\vert \geq\gamma N\,\Big|\,\widetilde{L}_{a}\cap\widetilde{L}_{b}\right)\\
 & \leq\frac{\mathbb{P}\left(\left\vert |\widetilde{S}_{a}\cap\widetilde{S}_{b}|-\frac{1}{4}N\right\vert \geq\gamma N\right)}{\mathbb{P}\left(\widetilde{L}_{a}\cap\widetilde{L}_{b}\right)}.
\end{align*}
Using Hoeffding's inequality, we have 
\[
\mathbb{P}\left(\left\vert |\widetilde{S}_{a}\cap\widetilde{S}_{b}|-\frac{1}{4}N\right\vert \geq\gamma N\right)=\mathbb{P}\left(\left\vert \sum_{j\in A}\widetilde{X}_{a,j}\widetilde{X}_{b,j}-\frac{1}{4}N\right\vert \geq\gamma N\right)\leq2e^{-2\gamma^{2}N}.
\]
On the other hand, using Stirling's approximation\footnote{We have $n^{n+\frac{1}{2}}e^{-n}\leq n!\leq en^{n+\frac{1}{2}}e^{-n}$
for each $n$.}, we have 
\[
\mathbb{P}\left(\widetilde{L}_{a}\cap\widetilde{L}_{b}\right)=\left(\frac{1}{2^{N}}\frac{N!}{\left(\frac{N}{2}!\right)^{2}}\right)^{2}\geq\left(\frac{2^{N+1}N^{-\frac{1}{2}}}{2^{N}e^{2}}\right)^{2}=\frac{4}{Ne^{4}}.
\]
We deduce that $\mathbb{P}_{N}\left(\left\vert |S_{a}\cap S_{b}|-\frac{1}{4}N\right\vert \geq\gamma N\right)\leq\frac{1}{2}e^{4}Ne^{-2\gamma^{2}N}.$
\end{proof}

\begin{lem} \label{lemT2}For each $a\neq b,$ for any subset $S\subseteq A$
and any $\gamma\geq\frac{1}{2N-2}$, 
\[
\mathbb{P}_{N}\left(\left\vert \sum_{i\in S}X_{i,a}-\frac{1}{2}\left\vert S\right\vert \right\vert \geq\gamma N\right)\leq2e^{-2N\gamma^{2}},\;{\mathnormal{a}nd}\;\mathbb{P}_{N}\left(\left\vert \sum_{i\in S}X_{i,a}X_{i,b}-\frac{1}{4}\left\vert S\right\vert \right\vert \geq\gamma N\right)\leq2e^{-\frac{1}{2}N\gamma^{2}}.
\]
\end{lem}

\begin{proof} For the first inequality, notice that $X_{i,a}$ are
i.i.d. Bernoulli random variables with parameter $\frac{1}{2}$. Hoeffding's
inequality implies that 
\[
\mathbb{P}_{N}\left(\left\vert \sum_{i\in S}X_{i,a}-\frac{1}{2}\left\vert S\right\vert \right\vert \geq\gamma N\right)\leq2e^{-2\gamma^{2}\frac{N^{2}}{|S|}}\leq2e^{-2N\gamma^{2}}.
\]
\end{proof}

For the second inequality, let $Z_{i}=X_{i,a}X_{i,b}.$ Notice that
all variables $Z_{i}$ are i.i.d. Bernoulli random variables with
parameter $p=\frac{1}{2}\left(\frac{\frac{N}{2}-1}{N-1}\right)=\frac{1}{4}-\frac{1}{4N-4}$.
Hoeffding's inequality implies that 
\begin{eqnarray*}
\mathbb{P}_{N}\left(\left\vert \sum_{i\in S}Z_{i}-\frac{1}{4}\left\vert S\right\vert \right\vert \geq\gamma N\right) & \leq & \mathbb{P}_{N}\left(\left\vert \sum_{i\in S}Z_{i}-p\left\vert S\right\vert \right\vert \geq\frac{1}{2}\gamma N\right)\leq2e^{-2\gamma^{2}\frac{N^{2}}{|S|}}\leq2e^{-\frac{1}{2}N\gamma^{2}},
\end{eqnarray*}
where we used $|S||p-\frac{1}{4}|\leq\frac{N}{4N-4}\leq\frac{\gamma N}{2}$
for the first inequality.

For each $a\neq b$ and $c\neq d,$ each $\gamma>0,$ define 
\begin{center}
\begin{tabular}{lll}
$Y_{a}=2\sum_{i\in A}X_{i,a},$  & $Y^{c}=2\sum_{i\in A}X_{c,i}=N$,  & \;\tabularnewline
$Y_{a,b}=4\sum_{i\in A}X_{i,a}X_{i,b},$  & $Y_{a}^{c}=4\sum_{i\in A}X_{i,a}X_{c,i},$  & $Y^{c,d}=4\sum_{i\in A}X_{c,i}X_{d,i},$\tabularnewline
$Y_{a,b}^{c}=8\sum_{i\in A}X_{i,a}X_{i,b}X_{c,i},$  & $Y_{a}^{c,d}=8\sum_{i\in A}X_{i,a}X_{c,i}X_{d,i},$  & $Y_{a,b}^{c,d}=16\sum_{i\in A}X_{i,a}X_{i,b}X_{c,i}X_{d,i}.$\tabularnewline
\end{tabular}
\par\end{center}

\begin{lem} \label{lemT3}For each $a\neq b$ and $c\neq d,$ each
$\gamma\geq64/N,$ each of the variables 
\[
Z\in\{Y_{a},Y^{c},Y_{a,b},Y^{c,d},Y_{a}^{c},Y_{a,b}^{c},Y_{a}^{c,d},Y_{a,b}^{c,d}\},
\]
\[
\mathbb{P}_{N}\left(\left\vert Z-N\right\vert \geq\gamma N\right)\leq e^{4}Ne^{-\frac{N}{32}{(\frac{\gamma}{10})}^{2}}.
\]
\end{lem}

\begin{proof} In case $Z=Y_{a}$ or $Y_{a,b},$ the bound follows
from Lemma \ref{lemT2} (for $S=A)$. If case $Z=Y^{c},$ the bound
is trivially satisfied. If $Z=Y^{c,d},$ the bound follows from Lemma
\ref{lemT1}.

In case $Z=Y_{a,b}^{c,d}$, notice that $Y_{a,b}^{c,d}=16\sum\limits _{i\in S_{c}\cap S_{d}}Z_{i}$
where $Z_{i}=X_{i,a}X_{i,b}$. All variables $Z_{i}$ are i.i.d. Bernouilli
random variables with parameter $p=\frac{1}{4}-\frac{1}{4N-4}$. Moreover,
$\left\{ Z_{i}\right\} _{i\neq c,d}$ are independent of $S_{c}\cap S_{d}$.
Enlarging the probability space, we can construct a new collection
of i.i.d. Bernoulli random variables $Z_{i}^{\prime}$ such that $Z_{i}^{\prime}=Z_{i}$
for all $i\neq c,d$ and such that $\left\{ (Z_{i}^{\prime})_{i\in A},S_{c}\cap S_{d}\right\} $
are all independent. Then, $\left\vert Y_{a,b}^{c,d}-16\sum\limits _{i\in S_{c}\cap S_{d}}Z_{i}^{\prime}\right\vert \leq32,$
and, because $\frac{1}{2}\gamma N\geq32,$ we have 
\[
\mathbb{P}_{N}\left(\left\vert Y_{a,b}^{c,d}-N\right\vert \geq\gamma N\right)\leq\mathbb{P}_{N}\left(\left\vert \sum\limits _{i\in S_{c}\cap S_{d}}Z_{i}^{\prime}-\frac{1}{16}N\right\vert \geq\frac{1}{32}\gamma N\right).
\]
Define the events 
\[
A=\left\{ \left\vert \frac{1}{4}\left\vert S_{c}\cap S_{d}\right\vert -\frac{N}{16}\right\vert \geq\frac{1}{160}\gamma N\right\} ,\;\;\;B=\left\{ \left\vert \sum\limits _{i\in S_{c}\cap S_{d}}Z_{i}^{\prime}-\frac{1}{4}\left\vert S_{c}\cap S_{d}\right\vert \right\vert \geq\frac{1}{40}\gamma N\right\} .
\]

Then, the probability can be further bounded by 
\[
\leq\mathbb{P}_{N}\left(A\right)+\mathbb{P}_{N}\left(B\right)\leq\frac{1}{2}e^{4}Ne^{-2N\left(\frac{1}{40}\gamma\right)^{2}}+2e^{-\frac{1}{2}N\left(\frac{1}{40}\gamma\right)^{2}}\leq e^{4}Ne^{-\frac{N\gamma^{2}}{3200}},
\]
where the first bound comes from Lemma \ref{lemT1} and the second
bound comes from the second bound in Lemma \ref{lemT2}.

The remaining bounds have proofs similar to (and simpler than) the
case $Z=Y_{a,b}^{c,d}$. We omit the details in the interest of space.
\end{proof}

Finally, we describe an event $E$ that collects these bounds. Recall
that $\alpha=1/25$, and define for each $a\neq b$ \ and $c\neq d$,
\begin{eqnarray*}
E_{a,b,c,d} & = & \left\{ \left\vert \frac{Y_{a,b}}{Y_{a}}-1\right\vert \leq2\alpha\right\} \cap\left\{ \left\vert \frac{Y_{a,b}^{c}}{Y_{a}^{c}}-1\right\vert \leq2\alpha\right\} \cap\left\{ \left\vert \frac{Y_{a}^{c,d}}{Y_{a}^{c}}-1\right\vert \leq2\alpha\right\} \cap\left\{ \left\vert \frac{Y_{a,b}^{c,d}}{Y_{a}^{c,d}}-1\right\vert \leq2\alpha\right\} \\
 & \cap & \left\{ \left\vert \frac{Y^{c,d}}{Y^{c}}-1\right\vert \leq2\alpha\right\} \cap\left\{ \left\vert \frac{Y_{a}^{c}}{Y^{c}}-1\right\vert \leq2\alpha\right\} \cap\left\{ \left\vert \frac{Y_{a}^{c,d}}{Y^{c,d}}-1\right\vert \leq2\alpha\right\} .
\end{eqnarray*}
Finally, let $E=\bigcap\limits _{a,b,c,d:a\neq b\ \text{and }c\neq d}E_{a,b,c,d}.$

\begin{lem} \label{lembound} We have 
\[
\mathbb{P}_{N}(E)>1-7e^{4}N^{5}e^{-\frac{N}{2163200}}\xrightarrow[n\to\infty]{}1.
\]
\end{lem}

\begin{proof} Take $\gamma=\frac{\alpha}{1+\alpha}=\frac{1}{26}$
and let 
\[
F_{a,b,c,d}=\bigcap\limits _{Z\in\{Y_{a},Y_{a,b},Y^{c,d},Y^{c,d},Y_{a}^{c},Y_{a,b}^{c},Y_{a}^{c,d},Y_{a,b}^{c,d}\}}\left\{ \left\vert Z-N\right\vert \leq\gamma N\right\} .
\]
It is easy to see that $F_{a,b,c,d}\subseteq E_{a,b,c,d}.$ The probability
that $F_{a,b,c,d}$ holds can be bounded from Lemma \ref{lemT3} (as
soon as $N\geq\frac{64}{\gamma}=1664$), as 
\[
\mathbb{P}_{N}\left(F_{a,b,c,d}\right)\geq1-7e^{4}Ne^{-\frac{N}{32.(260)^{2}}}.
\]
The result follows since there are fewer than $N^{4}$ ways of choosing
$(a,b,c,d)$. \end{proof}

Computations using the bound of Lemma \ref{lembound} show that $N=52.10^{6}$
is sufficient for the existence of an appropriate Markov chain. Therefore,
one can take $\varepsilon=3.10^{-17}$ in the statement of Theorem
\ref{THM3LARGESPACE}. We conclude the proof of Proposition \ref{pro3}
by showing that event $E$ implies conditions $UI1$ and $UI2.$

\begin{lem} \label{lemT4} If event $E$ holds, then conditions $UI1,UI2$
are satisfied. \end{lem}

\begin{proof} We fix law $\nu$ of the Markov chain on $A$ and assume
that it has been induced, as explained at the beginning of Section
\ref{existence}, by a transition matrix $P$ satisfying $E$. For
$l\geq1$, we forget about the state in $K$ and still let $u^{l}$
denote the marginal of $u^{l}$ over $C^{l}\times D^{l}$. If $c=(c_{1},...,c_{l})\in C^{l}$
and $d=(d_{1},...,d_{l})\in D^{l}$, we have $u^{l}(c,d)=\nu(c_{1},d_{1},...,c_{l},d_{l})$.

Let us begin with condition UI2, which we recall here: for all $1\leq p\leq l$,
for all ${d}\in D^{l}$, for all $d^{\prime}\in D^{p-1}$, for all
$m\in\{1,...,p-1\}$ such that ${d}_{m}\neq d_{m}^{\prime}$, for
$r=2m-1,2m$, 
\[
u^{l}\left(\tilde{c}\smile_{r+1}d^{\prime}|\tilde{d}={d},\tilde{c}\smile_{r}d^{\prime}\right)\in[1/2-\alpha,1/2+\alpha],(\ref{eq63})
\]
where $(\tilde{c},\tilde{d})$ is a random variable selected according
to $u^{l}$. The quantity $u^{l}\left(\tilde{c}\smile_{r+1}d^{\prime}|\tilde{d}={d},\tilde{c}\smile_{r}d^{\prime}\right)$
is thus the conditional probability of the event $(\tilde{c}$ and
$d'$ are nice at level $r+1$) given that they are nice at level
$r$ and that the signal received by player 2 is $d$. We divide the
problem into different cases.

\underline{Case $m>1$ and $r=2m-1$. } The events $\{\tilde{c}\smile_{2m}d'\}$
and $\{\tilde{c}\smile_{2m-1}d'\}$ can be decomposed as follows:
\begin{eqnarray*}
\{\tilde{c}\smile_{2m-1}d'\} & = & \{\tilde{c}\smile_{2m-2}d'\}\cap\{X_{d'_{m-1},\tilde{c}_{m}}=1\},\\
\{\tilde{c}\smile_{2m}d'\} & = & \{\tilde{c}\smile_{2m-2}d'\}\cap\{X_{d_{m-1}^{\prime},\tilde{c}_{m}}=1\}\cap\{X_{\tilde{c}_{m},d'_{m}}=1\}.
\end{eqnarray*}
So $u^{l}\left(\tilde{c}\smile_{2m}d^{\prime}|\tilde{d}={d},\tilde{c}\smile_{2m-1}d^{\prime}\right)=u^{l}\left(X_{\tilde{c}_{m},d'_{m}}=1|\tilde{d}={d},\tilde{c}\smile_{2m-1}d^{\prime}\right)$,
and the Markov property gives 
\[
\begin{split}u^{l}\left(\tilde{c}\smile_{2m}d^{\prime}|\tilde{d}={d},\tilde{c}\smile_{2m-1}d^{\prime}\right) & =u^{l}\left(X_{\tilde{c}_{m},d'_{m}}=1|X_{d'_{m-1},\tilde{c}_{m}}=1,X_{d_{m-1},\tilde{c}_{m}}=1,X_{\tilde{c}_{m},d_{m}}=1\right)\\
 & =\frac{\sum_{i\in U}X_{i,d_{m}^{\prime}}X_{d'_{m-1},i}X_{d_{m-1},i}X_{i,d_{m}}}{\sum_{i\in U}X_{d'_{m-1},i}X_{d_{m-1},i}X_{i,d_{m}}}.
\end{split}
\]
This is equal to $\frac{1}{2}\frac{Y_{d_{m},d_{m}^{\prime}}^{d_{m-1},d_{m-1}^{\prime}}}{Y_{d_{m}}^{d_{m-1},d_{m-1}^{\prime}}}$
if $d'_{m-1}\neq d_{m-1}$, and to $\frac{1}{2}\frac{Y_{d_{m},d_{m}^{\prime}}^{d_{m-1}}}{Y_{d_{m}}^{d_{m-1}}}$
if $d_{m-1}^{\prime}=d_{m-1}$. In both cases, $E$ implies (\ref{eq63}).

\underline{Case $r=2m$.}

We have $u^{l}\left(\tilde{c}\smile_{2m+1}d'|\tilde{d}=d,\tilde{c}\smile_{2m}d^{\prime}\right)=u^{l}\left(X_{d_{m}^{\prime},\tilde{c}_{m+1}}=1|\tilde{d}=d,\tilde{c}\smile_{2m}d^{\prime}\right)$,
and by the Markov property 
\begin{eqnarray*}
 &  & u^{l}\left(\tilde{c}\smile_{2m+1}d'|\tilde{d}=d,\tilde{c}\smile_{2m}d^{\prime}\right)\\
 & = & u^{l}\left(X_{d_{m}^{\prime},\tilde{c}_{m+1}}=1|X_{{d_{m}},\tilde{c}_{m+1}}=1,X_{\tilde{c}_{m+1},d_{m+1}}=1\right)\\
 & = & \frac{\sum_{i\in U}X_{d_{m}^{\prime},i}X_{d_{m},i}X_{i,d_{m+1}}}{\sum_{i\in U}X_{d_{m},i}X_{i,d_{m+1}}}=\frac{1}{2}\frac{Y_{d_{m+1}}^{d'_{m},d_{m}}}{Y_{d_{m+1}}^{d_{m}}}\;\in[1/2-\alpha,1/2+\alpha].
\end{eqnarray*}

\underline{Case $m=1$, $r=1$.} 
\begin{eqnarray*}
 &  & u^{l}\left(\tilde{c}\smile_{2}d'|\tilde{d}=d,\tilde{c}\smile_{1}d^{\prime}\right)\\
 & = & u^{l}\left(\tilde{c}\smile_{2}d'|\tilde{d}=d\right)=u^{l}\left(X_{\tilde{c}_{1},d_{1}^{\prime}}=1|X_{\tilde{c}_{1},d_{1}}=1\right),\\
 & = & \frac{\sum_{i\in U}X_{i,d'_{1}}X_{i,d_{1}}}{\sum_{i\in U}X_{i,d_{1}}}=\frac{1}{2}\frac{Y_{d_{1},d_{1}^{\prime}}}{Y_{d_{1}}}\;\in[1/2-\alpha,1/2+\alpha].
\end{eqnarray*}

The proof of condition $UI1$ being similar, it is omitted here.\end{proof}

\section{Proofs of Theorem \ref{THM4WEAKTOPOLOGY}}

\subsection{Theorem \ref{THM4WEAKTOPOLOGY}: weak topology is contained in value-based
topology}

Assume that $u_{n}\in\Delta\left(K\times C_{n}\times D_{n}\right)$
and $u\in\Delta\left(K\times C\times D\right)$ are information structures
such that $\dist\left(u_{n},u\right)\rightarrow0$. Then, for all
games $g$ in $\CG$, $|\val(\tilde{u_{n}},g)-\val(\tilde{u},g)|=|\val(u_{n},g)-\val(u,g)|\rightarrow0$.
By Theorem 12 in \citet{gossner_value_2001}, the functions $(\val(.,g))_{g}$
span the topology on $\Pi$. So $(\tilde{u}_{n})_{n}$ converges weakly
to $\tilde{u}$.

\subsection{Theorem \ref{THM4WEAKTOPOLOGY}: value-based topology is contained
in weak topology}

Assume that $u_{n}\in\Delta\left(K\times C_{n}\times D_{n}\right)$
and $u\in\Delta\left(K\times C\times D\right)$ are information structures
such that $\tilde{u}_{n}$ converges to $\tilde{u}$ in the weak topology.
We will prove that 
\begin{equation}
\limsup_{n\rightarrow\infty}\sup_{g\in\CG}\left(\val\left(u_{n},g\right)-\val\left(u,g\right)\right)\leq0.\label{eq:topologies 1}
\end{equation}
Because we can switch the roles of players, this will suffice to establish
that $\dist\left(u_{n},u\right)\rightarrow0$.

\emph{Partitions of unity}. We can assume without loss of generality
that $u$ is non-redundant and that all signals $c$ and $d$ have
positive probability. We can associate signals $c\in C\subseteq\N$
and $d\in D\subseteq\N$ with the corresponding hierarchies of beliefs
in $\Theta_{1}$ and $\Theta_{2}$. In other words, we identify $C\subseteq\Theta_{1}$
as the (countable) support of $\tilde{u}$ and $D\subseteq\Theta_{2}$
as the smallest countable set such that, for each $c\in C$, $\phi_{1}\left(K\times D|c\right)=1$
(i.e., $D$ is the union of countable supports of all beliefs of hierarchies
in $C$). For each $c\in C$ and $d\in D$, we denote the corresponding
hierarchies under $u$ as $\tilde{c}$ and $\tilde{d}$. Also, let
$C^{m}=C\cap\left\{ 0,...,m\right\} $ and $D^{m}=D\cap\left\{ 0,...,m\right\} $.

Because $\Theta_{2}$ is Polish, for each $m\in\N$ and each $d\in D^{m}$,
we can find continuous functions $\kappa_{d}^{m}:\Theta_{2}\rightarrow\left[0,1\right]$
for $m\in\N,d\in\{0,...,m\}$ such that $\kappa_{d}^{m}\left(\tilde{d}\right)=1\text{ for each }d\in D^{m},$
$\kappa_{d}^{m}\equiv0$ if $d\notin D,$ and $\sum_{d=0}^{m}\kappa_{d}^{m}\left(\theta_{2}\right)=1$
for each $\theta_{2}\in\Theta_{2}.$ In other words, for each $m$,
$\left\{ \kappa_{d}^{m}\right\} _{0\leq d\leq m}$ is a continuous
partition of unity on space $\Theta_{2}$ with the property that,
for each $d\in D^{m}$, $\kappa_{d}^{m}$ peaks at hierarchy $\tilde{d}$.
Notice that, for each $c\in C$ and each $d\in D^{p}$, we have $\E_{\phi_{1}(\tilde{c})}[\1_{\{k\}}(.)\kappa_{d}^{p}(.)]\geq u\left(k,d|c\right),$
and 
\[
\sum_{k\in K}\sum_{d=0}^{p}\left|\E_{\phi_{1}(\tilde{c})}[\1_{\{k\}}(.)\kappa_{d}^{p}(.)]-u\left(k,d|c\right)\right|=u(D\setminus D^{p}|c).
\]
Because all hierarchies $\tilde{c},c\in C$ are distinct, there exists
$p^{m}<\infty$ and $\varepsilon^{m}\in\left(0,\frac{1}{m}\right)$
for each $m$ such that, for any $c,c'\in C^{m}$ such that $c\neq c{}^{\prime}$,
\[
\sum_{k\in K}\sum_{d=0}^{p^{m}}\left|\E_{\phi_{1}(\tilde{c})}[\1_{\{k\}}\kappa_{d}^{p^{m}}]-\E_{\phi_{1}(\tilde{c}^{\prime})}[\1_{\{k\}}\kappa_{d}^{p^{m}}]\right|\geq2\varepsilon^{m}.
\]
\[
\text{ Let }\;h_{c}^{m}\left(\theta_{1}\right)=\sum_{k}\sum_{d=0}^{p^{m}}\left|\E_{\phi_{1}(\theta_{1})}[\1_{\{k\}}\kappa_{d}^{p^{m}}]-\E_{\phi_{1}(\tilde{c})}[\1_{\{k\}}\kappa_{d}^{p^{m}}]\right|.
\]
Then, $h_{c}^{m}$ is a continuous function such that $h_{c}^{m}\left(\tilde{c}\right)=0$
and such that, if $h_{c}^{m}\left(\theta_{1}\right)\leq\varepsilon^{m}$
for some $c\in C^{m}$, then $h_{c^{\prime}}^{m}\left(\theta_{1}\right)\geq\varepsilon^{m}$
for any $c'\in C^{m}$ such that $c^{\prime}\neq c$. For $0\leq c\leq m+1$,
define continuous functions 
\begin{align*}
\kappa_{c}^{m}\left(\theta_{1}\right) & =\max\left(1-\frac{1}{\varepsilon^{m}}h_{c}^{m}\left(\theta_{1}\right),0\right)\text{ for }c\in C_{m},\\
\kappa_{c}^{m} & \equiv0\text{ if }c\notin C,\text{ and }\kappa_{m+1}^{m}\left(\theta_{1}\right)=1-\sum_{c=0}^{m}\kappa_{c}^{m}\left(\theta_{1}\right).
\end{align*}
Then, for each $m$, $\sum_{c=0}^{m+1}\kappa_{c}^{m}\equiv1$, and
$\kappa_{c}^{m}\left(\theta_{1}\right)\in\left[0,1\right]$ for each
$c=0,...,m+1$, which implies that $\left\{ \kappa_{c}^{m}\right\} _{0\leq c\leq m+1}$
is a continuous partition of unity on space $\Theta_{1}$ such that,
for each $c\in C^{m}$, $\kappa_{c}^{m}(\tilde{c})=1$.

\emph{Conditional independence}. For each information structure $v\in\Delta\left(K\times C{}^{\prime}\times D{}^{\prime}\right)$,
define information structure $K^{m}v\in\Delta\left(K\times C{}^{\prime}\times\left\{ 0,...,m+1\right\} \times D{}^{\prime}\times\left\{ 0,...,p^{m}\right\} \right)$
so that $K^{m}v\left(k,c^{\prime},\hat{c},d^{\prime},\hat{d}\right)=v\left(k,c^{\prime},d^{\prime}\right)\kappa_{\hat{c}}^{m}\left(\tilde{c}^{\prime}\right)\kappa_{\hat{d}}^{p^{m}}\left(\tilde{d}^{\prime}\right).$
Let $\delta^{m}v=2\varepsilon^{m}+K^{m}v\left(\hat{c}=m+1\right).$
We are going to show that, under $K^{m}v$, signal $c^{\prime}$ is
$\delta^{m}v$-conditionally independent from $\left(k,\hat{d}\right)$
given $\hat{c}$. Notice first that, if $K^{m}v\left(k,d^{\prime},\hat{d},c{}^{\prime},\hat{c}\right)>0$
for some $\hat{c}\in C^{m}$, then $h_{\hat{c}}^{m}\left(\tilde{c}^{\prime}\right)\leq\varepsilon^{m}.$
It follows that 
\begin{align*}
 & \sum_{k}\sum_{\hat{d}=0}^{p^{m}}\left|K^{m}v\left(k,\hat{d}|\hat{c},c{}^{\prime}\right)-\E_{\phi_{1}\left(\tilde{\hat{c}}\right)}[\1_{\{k\}}\kappa_{\hat{d}}^{p^{*}}]\right|\\
 & =\sum_{k}\sum_{\hat{d}=0}^{p^{m}}\left|K^{m}v\left(k,\hat{d}|c{}^{\prime}\right)-\E_{\phi_{1}\left(\tilde{\hat{c}}\right)}[\1_{\{k\}}[\kappa_{\hat{d}}^{p^{m}}]\right|\\
 & =\sum_{k}\sum_{\hat{d}=0}^{p^{m}}\left|\E_{\phi_{1}(\tilde{c}^{\prime})}[\1_{\{k\}}\kappa_{\hat{d}}^{p^{m}}]-\E_{\phi_{1}\left(\tilde{\hat{c}}\right)}[\1_{\{k\}}\kappa_{\hat{d}}^{p^{m}}]\right|=h_{\hat{c}}^{m}\left(\tilde{c}^{\prime}\right)\leq\varepsilon^{m}.
\end{align*}
On the other hand, 
\begin{align*}
 & \sum_{k}\sum_{\hat{d}=0}^{p^{m}}\left|K^{m}v\left(k,\hat{d}|\hat{c}\right)-\E_{\phi_{1}\left(\tilde{\hat{c}}\right)}[\1_{\{k\}}\kappa_{\hat{d}}^{p^{*}}]\right|\\
 & =\sum_{k}\sum_{\hat{d}=0}^{p^{m}}\left|\frac{1}{K^{m}v(\hat{c})}\sum_{c'\in C'}K^{m}v(c',\hat{c})K^{m}v\left(k,\hat{d}|\hat{c},c{}^{\prime}\right)-\E_{\phi_{1}\left(\tilde{\hat{c}}\right)}[\1_{\{k\}}[\kappa_{\hat{d}}^{p^{m}}]\right|\\
 & \leq\sum_{c'\in C'}\frac{K^{m}v(c',\hat{c})}{K^{m}v(\hat{c})}\sum_{k}\sum_{\hat{d}=0}^{p^{m}}\left|K^{m}v\left(k,\hat{d}|\hat{c},c{}^{\prime}\right)-\E_{\phi_{1}\left(\tilde{\hat{c}}\right)}[\1_{\{k\}}[\kappa_{\hat{d}}^{p^{m}}]\right|=h_{\hat{c}}^{m}\left(\tilde{c}^{\prime}\right)\leq\varepsilon^{m}.
\end{align*}
\begin{align*}
\text{ Hence, }\; & \sum_{\hat{c}=1}^{m+1}\sum_{c^{\prime}}K^{m}v\left(\hat{c},c{}^{\prime}\right)\sum_{k,\hat{d}}\left|K^{m}v\left(k,\hat{d}|\hat{c},c{}^{\prime}\right)-K^{m}v\left(k,\hat{d}|\hat{c}\right)\right|\\
 & \leq2\varepsilon^{m}\sum_{\hat{c}=1}^{m}K^{m}v\left(\hat{c}\right)+K^{m}v\left(\hat{c}=m+1\right)\leq\delta^{m}v.
\end{align*}
Define information structure $L^{m}v=\marg_{K\times\left\{ 0,...,p^{m}\right\} \times\left\{ 0,...,m+1\right\} }K^{m}v.$
Then, because $d\left(K^{m}v,v\right)=0$, the proof of Proposition
\ref{prop: joint info} implies that

\noindent $\sup_{g\in\CG}\left(\val\left(v,g\right)-\val\left(L^{m}v,g\right)\right)$
$\leq\delta^{m}v.$

\emph{Proof of claim (\ref{eq:topologies 1})}. Observe that, for
each $k,\hat{c},\hat{d}$, 
\[
\left(L^{m}u_{n}\right)\left(k,\hat{c},\hat{d}\right)=\E_{\tilde{u}_{n}}\left(\kappa_{\hat{c}}^{m}\left(\theta_{1}\right)\E_{\phi_{1}\left(\theta_{1}\right)}[\1_{\{k\}}\kappa_{\hat{d}}^{p^{m}}]\right).
\]
Because all the functions in brackets above are continuous, weak convergence
$\tilde{u}_{n}\rightarrow\tilde{u}$ implies that $\left(L^{m}u_{n}\right)\left(k,\hat{c},\hat{d}\right)\rightarrow\left(L^{m}u\right)\left(k,\hat{c},\hat{d}\right)$
for each $k,\hat{c},\hat{d}$. Because the information structures
$L^{m}u_{n}$ and $L^{m}u$ are described on the same and finite spaces
of signals, the pointwise convergence implies $\dist\left(L^{m}u_{n},L^{m}u\right)\leq\left\Vert L^{m}u_{n}-L^{m}u\right\Vert \rightarrow0\text{ as \ensuremath{n\rightarrow\infty}.}$
Moreover, if $\hat{c}\in C^{m}$ and $\hat{d}\in D^{p^{m}}$, the
definitions imply that $\left(L^{m}u\right)\left(k,\hat{c},\hat{d}\right)\geq u\left(k,\hat{c},\hat{d}\right)$.
Therefore, 
\[
\dist\left(L^{m}u,u\right)\leq\|L^{m}u-u\|\leq2\left(u\left(C\backslash C^{m}\right)+u\left(D\backslash D^{p^{m}}\right)\right)\underset{n\rightarrow\infty}{\longrightarrow}0.
\]
It follows that $\delta^{m}u_{n}=\left(K^{m}u_{n}\right)\left(\hat{c}=m+1\right)\underset{n\rightarrow\infty}{\longrightarrow}\left(L^{m}u\right)\left(\hat{c}=m+1\right)$,
and 
\[
\left(L^{m}u\right)\left(\hat{c}=m+1\right)=1-(L^{m}u)(C^{m}\times D^{p^{m}})\leq1-u(C^{m}\times D^{p^{m}})\leq u\left(C\backslash C^{m}\right)+u\left(D\backslash D^{p^{m}}\right).
\]
Together, we obtain for each $m,n$ 
\[
\begin{split}\sup_{g\in\CG}\left(\val\left(u_{n},g\right)-\val\left(u,g\right)\right) & \leq\sup_{g\in\CG}\left(\val\left(u_{n},g\right)-\val\left(L^{m}u_{n},g\right)\right)\\
 & +\sup_{g\in\CG}\left(\val\left(L^{m}u_{n},g\right)-\val\left(L^{m}u\right)\right)+\sup_{g\in\CG}\left(\val\left(L^{m}u\right)-\val\left(u,g\right)\right)\\
\leq & \delta^{m}u_{n}+\left\Vert L^{m}u_{n}-L^{m}u\right\Vert +\left(u\left(C\backslash C^{m}\right)+u\left(D\backslash D^{p^{m}}\right)\right).
\end{split}
\]
\[
\text{Therefore,}\;\lim\sup_{n\rightarrow\infty}\sup_{g\in\CG}\left(\val\left(v,g\right)-\val\left(L^{m}v,g\right)\right)\leq3\left(u\left(C\backslash C^{m}\right)+u\left(D\backslash D^{p^{m}}\right)\right).
\]
When $m\rightarrow\infty$, the right-hand side converges to $0$
as well.

\section{Proof of Proposition \ref{PROP: APPROXIMATE KNOWLEDGE}}

Let $u^{\prime}\in\Delta\left(K\times(K_{C}\times C)\times(K_{D}\times D)\right)$
be defined so that $u=\marg_{K\times C\times D}u{}^{\prime}$ and
$u^{\prime}(\{k_{C}=\kappa\left(c\right),k_{D}=\kappa\left(d\right)\})=1$.
Because $u^{\prime}$ does not have any new information, we verify
(for instance, using Proposition \ref{prop: joint info}) that $\dist\left(u,u^{\prime}\right)=0$.
We are going to show that $C$ is $16\varepsilon$-conditionally independent
from $K\times K_{D}$ given $K_{C}$. Notice that, because $u$ exhibits
$\varepsilon$-knowledge, 
\begin{align*}
u^{\prime}\left\{ k_{C}\neq k\text{ or }k_{D}\neq k\right\}  & \leq u^{\prime}\left\{ k_{C}\neq k\right\} +u^{\prime}\left\{ k_{D}\neq k\right\} \\
 & \leq2\varepsilon+2\varepsilon=4\varepsilon.
\end{align*}
Therefore, 
\begin{align*}
 & \sum_{k,k_{C},k_{D}}u{}^{\prime}\left(k_{C}\right)\sum_{c}\left|u^{\prime}\left(k,k_{D},c|k_{C}\right)-u{}^{\prime}\left(k,k_{D}|k_{C}\right)u{}^{\prime}\left(c|k_{C}\right)\right|\\
= & \sum_{k,k_{C},k_{D}}u{}^{\prime}\left(k,k_{C},k_{D}\right)\sum_{c}\left|u^{\prime}\left(c|k,k_{C},k_{D}\right)-\sum_{k^{\prime},k_{D}{}^{\prime}}u{}^{\prime}\left(c|k^{\prime},k_{C},k{}_{D}^{\prime}\right)u{}^{\prime}\left(k^{\prime},k{}_{D}^{\prime}|k_{C}\right)\right|\\
\leq & \sum_{k}u{}^{\prime}\left(k,k,k\right)\sum_{c}\left|u^{\prime}\left(c|k,k,k\right)-\sum_{k^{\prime},k_{D}{}^{\prime}}u{}^{\prime}\left(c|k^{\prime},k_{C}=k,k{}_{D}^{\prime}\right)u{}^{\prime}\left(k^{\prime},k{}_{D}^{\prime}|k_{C}=k\right)\right|\\
 & +2u^{\prime}\left\{ k_{C}\neq k\text{ or }k_{D}\neq k\right\} \\
\leq & \sum_{k}u{}^{\prime}\left(k,k,k\right)\sum_{c}\left|u^{\prime}\left(c|k,k,k\right)-u{}^{\prime}\left(c|k,k,k\right)\frac{u^{\prime}\left(k,k,k\right)}{u^{\prime}\left(k_{C}=k\right)}\right|\\
 & +\sum_{k}u{}^{\prime}\left(k,k,k\right)\sum_{c}\sum_{k^{\prime}\neq k\text{, or }k_{D}^{\prime}\neq k}\left|u^{\prime}\left(c|k^{\prime},k_{C}=k,k_{D}^{\prime}\right)u{}^{\prime}\left(k^{\prime},k_{D}^{\prime}|k_{C}=k\right)\right|\\
 & +2u^{\prime}\left\{ k_{C}\neq k\text{ or }k_{D}\neq k\right\} \\
\leq & \sum_{k}u{}^{\prime}\left(k,k,k\right)\left|1-\frac{u^{\prime}\left(k,k,k\right)}{u^{\prime}\left(k_{C}=k\right)}\right|+3u^{\prime}\left\{ k_{C}\neq k\text{ or }k_{D}\neq k\right\} \\
\leq & \sum_{k}\left|u^{\prime}\left(k_{C}=k\right)-u^{\prime}\left(k,k,k\right)\right|+3u^{\prime}\left\{ k_{C}\neq k\text{ or }k_{D}\neq k\right\} \\
\leq & 4u^{\prime}\left\{ k_{C}\neq k\text{ or }k_{D}\neq k\right\} \leq16\varepsilon.
\end{align*}

Because an analogous result applies to the information of the other
player, Proposition \ref{prop: joint info} shows that 
\[
\dist\left(u^{\prime},v^{\prime}\right)\leq16\varepsilon,
\]
where $v^{\prime}=\marg_{K\times K_{C}\times K_{D}}u^{\prime}$. Because
\begin{align*}
\dist\left(v,v^{\prime}\right)\leq & \sum_{k,k_{C},k_{D}}\left|v\left(k,k_{C},k_{D}\right)-v{}^{\prime}\left(k,k_{C},k_{D}\right)\right|\\
\leq & 2v^{\prime}\left\{ k_{C}\neq k\text{ or }k_{D}\neq k\right\} =2u{}^{\prime}\left\{ k_{C}\neq k\text{ or }k_{D}\neq k\right\} \leq4\varepsilon,
\end{align*}
the triangle inequality implies that 
\[
\dist\left(u,v\right)\leq\dist\left(u,u^{\prime}\right)+\dist\left(u^{\prime},v^{\prime}\right)+\dist\left(v,v^{\prime}\right)\leq20\varepsilon.
\]

\section{Proof of Theorem \ref{THM:NZS DISTANCE}}

Suppose that $u$ and $v$ are two simple and non-redundant information
structures with finite support. Let $\tilde{u}$ and $\tilde{v}$ be the associated probability
distributions over player 1's belief hierarchies. It is easy to show
that, if two non-redundant information structures induce the same
distributions over hierarchies of beliefs $\tilde{u}=\tilde{v}$,
then they are equivalent from any strategic point of view, and, in
particular, they induce the same set of ex-ante BNE payoffs. Hence,
we assume that $\tilde{u}\neq\tilde{v}$.

Let $H_{u}=\text{supp}\tilde{u}$ and $H_{v}=\text{supp}\tilde{v}$.
Lemma III.2.7 in \citet{mertens_sorin_zamir_2015} implies that the
sets $H_{u}$ and $H_{v}$ are disjoint.

By adapting the construction made in Lemma 4 of \cite{dekel_topologies_2006} (see also Lemma 11 in \cite{ely_critical_2011}), there exists a non-zero sum payoff function
$g^{\left(0\right)}:K\times\left(I\times I_{0}\right)\times J\rightarrow\left[-1,1\right]^{2}$
such that $I_{0}=H_{u}\cup H_{v}$ and such that the set of player
1's rationalizable actions of type $c\in C$ with hierarchy $h\left(c\right)$
is contained in set $I\times\left\{ h\left(c\right)\right\} $. In
particular, in a BNE, each type of player 1 will report its hierarchy.
Construct game $g^{\left(1\right)}:K\times\left(I\times I_{0}\right)\times\left(J\times\left\{ u,v\right\} \right)\rightarrow\left[-1,1\right]^{2}$
with payoffs 
\begin{align*}
g_{1}^{\left(1\right)}\left(k,i,i_{0},j,j_{0}\right) & =g_{1}^{\left(0\right)}\left(k,i,i_{0},j\right),\\
g_{2}^{\left(1\right)}\left(k,i,i_{0},j,j_{0}\right) & =\frac{1}{2}g_{2}^{\left(0\right)}\left(k,i,i_{0},j\right)+\begin{cases}
\frac{1}{2}, & \text{if }j_{0}=u\text{ and }i_{0}\in H_{u}\\
-\frac{1}{2}, & \text{if }j_{0}=u\text{ and }i_{0}\notin H_{u},\\
0, & \text{if }j_{0}=v.
\end{cases}
\end{align*}
Then, the rationalizable actions of player 2 in game $g^{\left(1\right)}$
are contained in $J\times\left\{ u\right\} $ for any type in type
space $u$ and in $J\times\left\{ v\right\} $ for any type in type
space $v$.

Finally, for any $\varepsilon\in\left(0,1\right),$construct game
$g^{\varepsilon}:K\times\left(I\times I_{0}\right)\times\left(J\times\left\{ u,v\right\} \right)\rightarrow\left[-1,1\right]^{2}$
with payoffs 
\begin{align*}
g_{1}^{\varepsilon}\left(k,i,i_{0},j,j_{0}\right) & =\varepsilon g_{1}^{\left(0\right)}\left(k,i,i_{0},j,j_{0}\right)+\left(1-\varepsilon\right)\begin{cases}
1, & \text{if }j_{0}=u,\\
-1, & \text{if }j_{0}=v,
\end{cases},\\
g_{2}^{\varepsilon} & \equiv g_{2}^{\left(1\right)}.
\end{align*}
Then, the BNE payoff of player belongs to $\left[1-\varepsilon,1\right]$
on structure $u$ and $\left[-1,-1+\varepsilon\right]$ on structure
$v$. It follows that the payoff distance between the two type spaces
is at least $2-2\varepsilon$, for arbitrary $\varepsilon>0$.

Next, suppose that $u$ and $v$ are two non-redundant information
structures with decomposition $u=\sum_{\alpha}p_{\alpha}u_{\alpha}$
and $v=\sum_{\alpha}q_{\alpha}v_{\alpha}$ such that $\tilde{u}_{\alpha}=\tilde{v}_{\alpha}$
for each $\alpha$. Let $g$ be a non-zero sum payoff function. Let
$\sigma_{\alpha}$ be an equilibrium on $u_{\alpha}$ with payoffs
$g_{\alpha}=g\left(\sigma_{a}\right)\in\R^{2}$. Let $s_{\alpha}$
be the associated equilibrium on $v_{\alpha}$ (that can be obtained
by mapping the hierarchies of beliefs through an appropriate bijection)
with the same payoffs $g_{\alpha}$. The distance between payoffs
is bounded by 
\begin{align*}
\left\Vert \sum p_{\alpha}g\left(\sigma_{\alpha}\right)-q_{\alpha}g\left(s_{\alpha}\right)\right\Vert _{\max} & =\left\Vert \sum\left(p_{\alpha}-q_{\alpha}\right)g_{a}\right\Vert _{\max}\\
 & \leq\sum\left|p_{\alpha}-q_{\alpha}\right|\left\Vert g_{\alpha}\right\Vert _{\max}\leq\sum\left|p_{\alpha}-q_{\alpha}\right|,
\end{align*}
where the last inequality comes from the fact that payoffs are bounded.

On the other hand, let $A=\left\{ \alpha:p_{\alpha}>q_{\alpha}\right\} $.
Using a similar construction as above, we can construct game $g^{\left(1\right)}$
such that player 2's actions have form $J\times\left\{ u_{A},u_{B}\right\} $,
and his rationalizable actions are contained in set $J\times\left\{ u_{A}\right\} $
for any type in type space $u_{\alpha},\alpha\in A$ and in $J\times\left\{ u_{B}\right\} $
otherwise. Further, we construct game $g^{\left(\varepsilon\right)}$
as above. Then, any player $1$'s equilibrium payoff $g_{1,\alpha}^{(\varepsilon)}$
is at least $1-\varepsilon$ for any type in type space $u_{\alpha},\alpha\in A$,
and $-1+\varepsilon$ for any type in type space $u_{\alpha}$ for
$\alpha\notin A$. Denoting player 2's equilibrium payoff as $g_{2,\varepsilon}^{\varepsilon}$,
the payoff distance in game $g^{\varepsilon}$ is at least 
\begin{align*}
 & \max\left(\left|\sum_{\alpha}\left(p_{\alpha}-q_{\alpha}\right)g_{1,\alpha}\right|,\left|\sum_{\alpha}\left(p_{\alpha}-q_{\alpha}\right)g_{2,\alpha}\right|\right)\geq\left|\sum_{\alpha}\left(p_{\alpha}-q_{\alpha}\right)g_{1,\alpha}\right|\\
\geq & \left[\sum_{\alpha\in A}\left(p_{\alpha}-q_{\alpha}\right)-\sum_{\alpha\notin A}\left(p_{\alpha}-q_{\alpha}\right)\right]\left(1-\varepsilon\right)\geq\left(1-\varepsilon\right)\sum\left|p_{\alpha}-q_{\alpha}\right|.
\end{align*}
Because $\varepsilon>0$ is arbitrary, the two above inequalities
show that the payoff distance is equal to $\sum\left|p_{\alpha}-q_{\alpha}\right|$.

\section{Examples and Counterexamples\label{sec:Examples-and-Counterexamples}}

This section contains various examples and counterexamples.

\subsection{Computing the distance}

We begin with two examples of computations of the value-based distance
based on the characterization from Theorem \ref{thm1}.

We assume that there are two states, $K=\{\textcolor{blue}{Blue},\textcolor{red}{Red}\}$.
We consider information structures which are uniform over a finite
subset of $K\times\N\times\N$.

%
%
%
%
%
%

The next picture illustrates information structure $u_{1}$: 
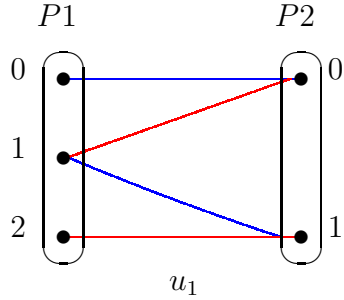
\begin{center}
\begin{picture}(50,100) \textcolor{blue}{{} \put(20,60){\line(1,0){90}}}
\textcolor{blue}{{} \qbezier(20,30)(55,15)(100,00)} \textcolor{red}{{}
\qbezier(15,30) (15,30)(100,60) } \textcolor{red}{{} \put(10,00){\line(1,0){90}}}
\put(10, 30){\oval(15, 80)} \put(100, 30){\oval(15, 80)} \put(10,60){\circle*{5}}
\put(10,30){\circle*{5}} \put(10,00){\circle*{5}} \put(100,60){\circle*{5}}
\put(100,00){\circle*{5}} \put(0,80){$P1$} \put(90,80){$P2$}
\put(-10,60){$0$} \put(-10,30){$1$} \put(-10,00){$2$}
\put(110,60){$0$} \put(110,00){$1$} \put(50,-20){$u_{1}$}
\end{picture} 
\par\end{center}

\bigskip{}
 \bigskip{}

Each line corresponds to a pair of signals from the support of the
information structure; each such a pair has equal probability. The
color of the line represents the state. For example, conditional on
both players receiving signal 0, the state is \textcolor{blue}{Blue}.
Note that signals 0 and 1 convey different information for player
2. After receiving signal 0, player 2 knows that if the state is \textcolor{blue}{Blue},
then player 1 knows it, and if the state is \textcolor{red}{Red},
player 1's belief on the state is uniform. Whereas after receiving
signal 1, player 2 knows that : if the state is \textcolor{blue}{Blue},
then player 1's belief on the state is uniform, and if the state is
\textcolor{red}{Red}, player 1 knows it.

The next figure presents information structures $u_{2}$ and $u_{2}^{\prime}$: 
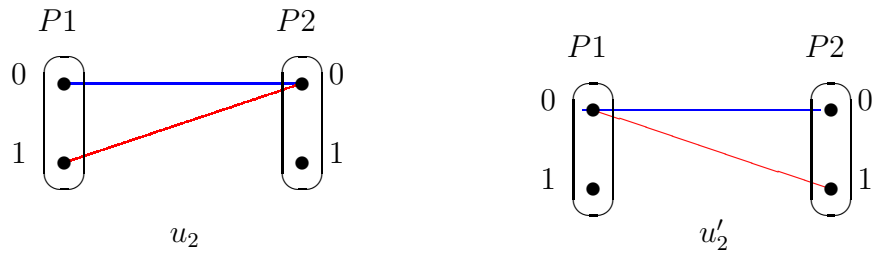
\begin{center}
\begin{picture}(50,90) \textcolor{blue}{\put(15,60){\line(1,0){90}}}
\textcolor{red}{{} \qbezier(10,30) (10,30)(100,60) } \put(10, 45){\oval(15,
50)} \put(100, 45){\oval(15, 50)} \put(10,60){\circle*{5}}
\put(10,30){\circle*{5}} \put(100,60){\circle*{5}} \put(100,30){\circle*{5}}
\put(0,80){$P1$} \put(90,80){$P2$} \put(-10,60){$0$} \put(-10,30){$1$}
\put(110,60){$0$} \put(110,30){$1$} \put(50,00){$u_{2}$}
\end{picture} \hspace{5cm} \begin{picture}(50,100) \textcolor{blue}{\put(10,50){\line(1,0){90}}}
\textcolor{red}{{} \put(10,50){\line(3,-1){90}}} \put(10,
35){\oval(15, 50)} \put(100, 35){\oval(15, 50)} \put(50,00){$u'_{2}$}
\put(10,50){\circle*{5}} \put(10,20){\circle*{5}} \put(100,50){\circle*{5}}
\put(100,20){\circle*{5}} \put(0,70){$P1$} \put(90,70){$P2$}
\put(-10,50){$0$} \put(-10,20){$1$} \put(110,50){$0$}
\put(110,20){$1$} \end{picture} 
\par\end{center}

\bigskip{}

Under information structure $u_{2}$, player 1 knows the state, but
player 2 does not. Under $u_{2}^{\prime}$, player 2 knows the state,
but player 1 does not. We have the following observation: 
\begin{prop}
$\dist\left(u_{1},u_{2}\right)=\frac{1}{2}$. $\dist\left(u_{1},u_{2}^{\prime}\right)=1.$ 
\end{prop}
It follows that $u_{1}$ is closer to $u_{2}$ than to $u_{2}^{\prime}$. 
\begin{proof}
We have $u_{2}\succeq u_{1}$, hence 
\[
\dist(u_{2},u_{1})=\min_{q_{1}\in\CQ,q_{2}\in\CQ}\|q_{1}.u_{1}-u_{2}.q_{2}\|.
\]
Define $q_{1}$ in $\CQ$ such that $q_{1}(0)=\delta_{0}$, $q_{1}(1)=q_{1}(2)=\delta_{1}$,
and $q_{2}$ in $\CQ$ satisfying $q_{2}(0)=1/2\,\delta_{0}+1/2\,\delta_{1}$.
The information structures $q_{1}.u_{1}$ and $u_{2}.q_{2}$ can be
represented as follows:

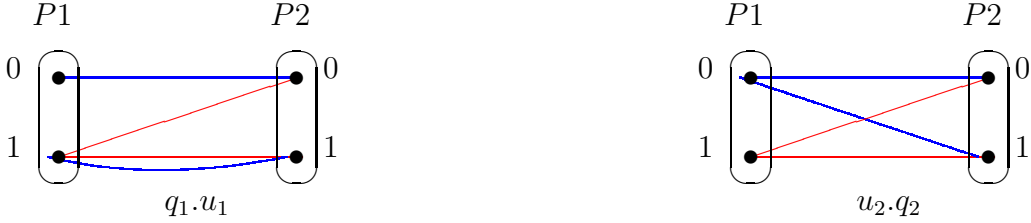
\begin{picture}(100,100) \textcolor{blue}{\put(20,50){\line(1,0){90}}}
\textcolor{red}{{} \put(20,20){\line(1,0){90}}} \textcolor{blue}{\qbezier(10,20)(50,10)(100,20)}
\textcolor{red}{{} \put(10,20){\line(3,1){90}}} \put(10, 35){\oval(15,
50)} \put(100, 35){\oval(15, 50)} \put(50,00){$q_{1}.u_{1}$}
\put(10,50){\circle*{5}} \put(10,20){\circle*{5}} \put(100,50){\circle*{5}}
\put(100,20){\circle*{5}} \put(0,70){$P1$} \put(90,70){$P2$}
\put(-10,50){$0$} \put(-10,20){$1$} \put(110,50){$0$}
\put(110,20){$1$} \textcolor{blue}{\put(270,50){\line(1,0){90}}}
\textcolor{red}{{} \put(270,20){\line(1,0){90}}} \textcolor{blue}{\qbezier(260,50)(260,50)(350,20)}
\textcolor{red}{{} \put(260,20){\line(3,1){90}}} \put(260,
35){\oval(15, 50)} \put(350, 35){\oval(15, 50)} \put(300,00){$u_{2}.q_{2}$}
\put(260,50){\circle*{5}} \put(260,20){\circle*{5}} \put(350,50){\circle*{5}}
\put(350,20){\circle*{5}} \put(250,70){$P1$} \put(340,70){$P2$}
\put(240,50){$0$} \put(240,20){$1$} \put(360,50){$0$}
\put(360,20){$1$} \end{picture}

\vspace{0.5cm}

\noindent Notice that $u_{2}.q_{2}\sim u_{2}$, whereas $q_{1}.u_{1}\preceq u_{1}.$
$\|q_{1}.u_{1}-u_{2}.q_{2}\|=1/2$, hence $d(u_{2},u_{1})\leq1/2$.

Consider now the payoff structure $g$ given by $\left\{ \textcolor{blue}{ \left(\begin{array}{cc}
0 & 1\\
0 & -1
\end{array} \right) }\;,\;\textcolor{red}{\left(\begin{array}{cc}
-1 & 0\\
1 & 0
\end{array} \right)}\right\} $, where the color corresponds to the payoffs in a given state. In
the game $(u_{2},g)$, it is optimal for player 1 to play Top if $0$
and Bottom if $1$, and $\val(u_{1},g)=1/2$. In the game $(u_{1},g)$
it is optimal for player 2 to play Left if $0$ and Right if $1$,
and $\val(u_{1},g)=0$. Consequently, $\dist(u_{2},u_{1})\geq1/2$,
and we obtain $\dist(u_{2},u_{1})=1/2.$

Notice that $u'_{2}\sim u''_{2}$, with $u''_{2}$ obtained from $u_{2}^{\prime}$
by exchanging the signals 0 and 1 for each player, and $\|u_{1}-u''_{2}\|=1$.
Considering the payoff structure given by $\left\{ \textcolor{blue}{ \left(\begin{array}{cc}
-1 & 1\\
-1 & 1
\end{array} \right) }\;,\;\textcolor{red}{\left(\begin{array}{cc}
1 & -1\\
1 & -1
\end{array} \right)}\right\} $ gives $\dist(u'_{2},u_{1})=1$. 
\end{proof}

\subsection{Impact of the marginal over states}

We illustrate an application of Proposition \ref{prop:diameter} in the binary
case $K=\left\{ 0,1\right\} $. Fix $p,q\in\Delta K$. In such a case,
one easily checks that the maximum in the right hand side of inequalities
\eqref{eq:diameter bounds} is attained by either $p^{\prime}=q{}^{\prime}=(1,0)$,
or $p^{\prime}=q{}^{\prime}=(0,1)$, or $p^{\prime}=p,q{}^{\prime}=q$.
It follows that for any two information structures $u,v$, 
\[
\dist\left(u,v\right)\leq2\left(1-\max\left(\min\left(p_{0},q_{0}\right),\min\left(p_{1},q_{1}\right),p_{0}q_{0}+p_{1}q_{1}\right)\right).
\]

\begin{example}
\label{ex: extreme distance}The bound is attained when $u\left(k,c,d\right)=p_{k}\1_{c=k}\1_{d=0}$
and $v\left(k,c,d\right)=q_{k}\1_{c=0}\1_{d=k}$ for each $k,c,d\in\left\{ 0,1\right\} $. 
\end{example}

\subsection{Value of additional information: games vs single-agent problems}

The next example illustrates that the thesis of Proposition \ref{prop: SA=00003D00003D00003D00003D00003D00003D00003D00003D00003D00003DGames} 
does not hold without a conditional independence. 
\begin{example}
\label{ex value added}Suppose that $C=\left\{ *\right\} ,K=D=C{}^{\prime}=\left\{ 0,1\right\} $,
and let $u\in\Delta(K\times(C\times C')\times D)$ be defined by 
\[
u\left(k,(*,c^{\prime}),d\right)=\begin{cases}
\frac{1}{4}\frac{k+c^{\prime}}{2}, & \text{if }d=1,\\
\frac{1}{4}\left(1-\frac{k+c^{\prime}}{2}\right), & \text{if }d=0.
\end{cases}
\]
Let $v$ denote the marginal of $u$ over $K\times C\times D$. Going
from $v$ to $u$, the new information $c^{\prime}$ is independent
from $k$, but $d$ is both correlated with $k$ and $c^{\prime}$.
Because $c^{\prime}$ is independent from the state, the value of
new information in player 1 single-agent problem is 0: 
\[
\dist_{1}\left(u,v\right)=0.
\]
However, signal $c^{\prime}$ provides non-trivial information about
the signal of the other player, hence it is valuable in some games,
which implies that $\dist(u,v)>0$. (Indeed, using Theorem 1, since
$u$ gives player $1$ more information, we have 
\[
\dist(u,v)=\min_{q_{1},q_{2}}\|u.q_{2}-q_{1}.v\|,
\]
where $q_{2}:\{0,1\}\rightarrow\Delta\{0,1\}$ and $q_{1}\in\Delta\{0,1\}$.
The existence of a pair $(q_{1},q_{2})$ such that $\|u.q_{2}-q_{1}.v\|=0$
is equivalent to the system of equations 
\[
\forall(k,d,c')\in\{0,1\}^{3},\;u.q_{2}(k,d,c')=v.q_{1}(k,d,c'),
\]
where the unknowns are $q_{1},q_{2}$, and one can check that this
system does not admit any solution. In other words, the information
that would be useless in a single-agent decision problem is valuable
in a strategic setting. 
\end{example}

\subsection{Informational substitutes}

Here, we discuss the assumptions of Proposition \ref{prop. Info substitutes}. The conditional
independence assumption is equivalent to two simpler assumptions (a)
$c_{1}$ is conditionally independent from $\left(c,c_{2}\right)$
given $\left(k,d\right)$, and (b) $c_{1}$ and $d$ are conditionally
independent given $k$. Both (a) and (b) are important as it is illustrated
in the two subsequent examples. 
\begin{example}
Violation of (a). Suppose that $C=D=\left\{ *\right\} $, $K=C_{1}=C_{2}=\left\{ 0,1\right\} $,
$c_{1}$ and $c_{2}$ are uniformly and independently distributed,
and $k=c_{1}+c_{2}\mod2$. Then, signal $c_{2}$ is itself independent
from the state, hence useless without $c_{1}$. Knowing $c_{1}$ and
$c_{2}$ means knowing the state, which is, of course, very valuable.
Thus, the value of $c_{2}$ increases when $c_{1}$ is also present. 
\end{example}
\medskip{}

\begin{example}
\label{ex17} Violation of (b). Suppose that $C=\left\{ *\right\} $,
$K=C_{1}=C_{2}=D=\left\{ 0,1\right\} $, $c_{1}$ and $d$ are uniformly
and independently distributed, $c_{2}=d$, and $k=c_{1}+d\mod2$.
Notice that part (a) of the assumption holds (given $\left(k,d\right)$,
both signals $c_{1}$ and $c_{2}$ are constant, hence, independent),
but part (b) is violated. Again, signal $c_{2}$ is useless alone,
but together with $c_{1}$ it allows the determine the state and the
information of the other player. 
\end{example}

\subsection{Informational complements}

The next example shows that the independence assumptions in Proposition \ref{prop. Info complements}
 are important. 
\begin{example}
Suppose that $C=\left\{ *\right\} $ and $K=D=D_{1}=C_{1}=\left\{ 0,1\right\} $.
The state is drawn uniformly. Signal $d$ is equal to the state with
probability $\frac{2}{3}$ and signal $d_{1}$ is equal to the state
for sure. Finally, $c_{1}=1$ iff $d=k$. 
\end{example}
In other words, player 2's signal $d$ is an imperfect information
about the state. Signal $c_{1}$ carries information about the quality
of the signal of player 1. Such signal is valuable in some games,
hence $\dist\left(u^{\prime},v{}^{\prime}\right)>0.$ In the same
time, if player 2 learns the state perfectly, signal $c$ becomes
useless, and $\dist\left(u,v\right)=0$. (The last claim can be formally
proven using Proposition \ref{prop: joint info}.)

\subsection{Convergence to simple structures}

Next, we provide an example of a convergent sequence of information
structures. 
\begin{example}
\label{exa6}Consider a sequence of information structures (which
are all uniform over a finite subset of $K\times\N\times\N$):

\vspace{0.5cm}
 
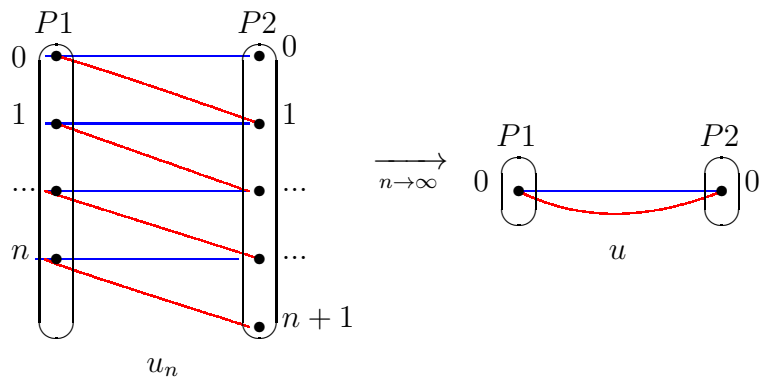
\begin{center}
\setlength{\unitlength}{.3mm} \begin{picture}(350,120) 

\textcolor{blue}{\put(35,30){\line(1,0){90}}}
\textcolor{blue}{\put(35,60){\line(1,0){90}}}
\textcolor{blue}{\put(30,90){\line(1,0){90}}} 
 \textcolor{blue}{\put(25,120){\line(1,0){90}}}

\textcolor{red}{{} \qbezier(20,30)(15,30)(110,00) }
 \textcolor{red}{{} \qbezier(15,60) (15,60)(110,30)}
\textcolor{red}{{} \qbezier(15,90) (15,90)(100,60) }
 \textcolor{red}{{} \qbezier(10,120)(32,114)(100,90)}

 \put(10, 60){\oval(15, 130)} \put(100, 60){\oval(15,
130)} \put(10,60){\circle*{5}} \put(10,30){\circle*{5}}
\put(10,90){\circle*{5}} \put(10,120){\circle*{5}} \put(100,90){\circle*{5}}
\put(100,120){\circle*{5}} \put(100,00){\circle*{5}} \put(100,60){\circle*{5}}
\put(100,30){\circle*{5}} \put(0,130){$P1$} \put(90,130){$P2$}
\put(-10,115){$0$} \put(-10,90){$1$} \put(-10,60){$...$}
\put(-10,30){$n$} \put(110,0){$n+1$} \put(110,60){$...$}
\put(110,30){$...$} \put(110,90){$1$} \put(110,120){$0$}
\put(50, -20){$u_{n}$} \put(150,70){$\xrightarrow[n\to\infty]{}$}
\textcolor{blue}{\put(215,60){\line(1,0){90}}} \textcolor{red}{{} \qbezier(210,60)
(250,40)(300,60)} \put(250, 30){$u$} \put(210, 60){\oval(15,
30)} \put(300, 60){\oval(15, 30)} \put(210,60){\circle*{5}}
\put(300,60){\circle*{5}} \put(200,80){$P1$} \put(290,80){$P2$}
\put(190,60){$0$} \put(310,60){$0$} \end{picture} 
\par\end{center}
\vspace{0.5cm}

Signals $0$ and $n+1$ perfectly reveal the state for player 2; no
other signals provide any information about the state. When $n$ is
large, with a high probability, none of the players learns anything
about the state (even if they learn about the beliefs of the other
player). Notice that the signal of each player is $\frac{1}{n+1}$-independent
from the state of the world. Proposition \ref{prop: joint info} implies that $\dist\left(u_{n},u\right)\leq\frac{2}{n+1}\rightarrow0$.
In particular, the information structures $u_{n}$ converge to the
structure $u$, where no player receive any information.

It is instructive to provide an elementary argument for the convergence.
Consider garblings $q_{1}$, $q_{2}$, such that $q_{1}(0)$ is uniform
on $\{0,...,n\}$, and $q_{2}(c)=\delta_{0}$ for each $c$. Then
$q_{1}.u=u_{n}.q_{2}$. We obtain $u\succeq u_{n}$, and $\dist(u,u_{n})=\min_{q'_{1},q'_{2}\in\CQ}\|q'_{1}.u_{n}-u.q'_{2}\|$.
Consider now $q'_{1}=q_{2}$ and $q'_{2}$ such that $q'_{2}(0)$
is uniform on $\{0,...,n+1\}$. We get $\|q'_{1}.u_{n}-u_{n}.q'_{2}\|\leq1/(n+1)\xrightarrow[n\to\infty]{}0.$ 
\end{example}

\section{Value of experiments\label{sec:Value-of-experiments}}

In this section, we compute the value of multiple (conditionally)
independent Blackwell experiments.

\subsection{Blackwell experiments}

We work with a binary case $k=\left\{ 0,1\right\} $ and, to simplify
the analysis, we suppose that two states of the world are equally
probable. A (binary Blackwell) experiment is defined as a probability
$p\in\left(\frac{1}{2},1\right)$, with the interpretation that the
agent observes signal $s=k$ with probability $p$ and signal $s=1-k$
with probability $1-p$.

For each $n,m\geq0$ and each $p,r\in\left(\frac{1}{2},1\right)$,
we define an information structure $u_{n,m}$ in which player 1 observes
outcomes of $n$ conditionally independent copies of experiment $p$
and player $2$ observes $m$ conditionally independent copies of
experiment $q$. Formally, for each $k=0,1,c\leq n,d\leq m$, let
\begin{align*}
u_{n,m}(k,c,d) & =\begin{cases}
\frac{1}{2}\left(p^{n-c}(1-p)^{c}\right)\left(r^{n-d}(1-r)^{d}\right), & k=0\\
\frac{1}{2}\left(p^{c}(1-p)^{n-c}\right)\left(r^{d}(1-r)^{n-d}\right) & k=1
\end{cases},
\end{align*}
Here, $c$ is the number of 1s observed by player 1 and $d$ is the
number of 1s observed by player 2. The goal of this section is to compute
\[
\dist(u_{l,m},u_{n,m})
\]
for some $l<n$ and $m$. By Proposition \ref{prop: SA=00003D00003D00003D00003D00003D00003D00003D00003D00003D00003DGames}, we have 
\begin{equation}
\dist(u_{l,m},u_{n,m})=\dist_{1}(u_{l},u_{n}),\label{eq:BExp}
\end{equation}
where $u_{l}$ and $u_{n}$ are defined from $u_{l,m}$ and $u_{n,m}$
by taking marginals over state and information of player 1.

\subsection{Result}

We define $S_{n}$ a random variable following a binomial $\mathcal{B}(n,p)$,
and $D_{n}=2S_{n}-n$. The interpretation is that $S_{n}$ represents
the successes of the Blackwell experiment $p$ (i.e., the number of
outcomes that are equal to the state) and $D_{n}=S_{n}-\left(n-S_{n}\right)$
is the difference between the number of successes and failures. It
tuns out that we can represent (\ref{eq:BExp}) in terms of the probability
distribution of $D_{n}$. 
\begin{prop}
\label{pro11} For each $n,m$, we have 
\begin{align}
\dist_{1}({u_{n},u_{l}}) & =\max_{d\in\{0,...,l\}}\gamma_{n,l}^{d},\text{ where}\label{eq:Blackwel exp char}\\
\gamma_{n,l}^{d} & =2(1-q_{d})(\Prob(D_{n}>d)-\Prob(D_{l}>d))-2q_{d}(\Prob(D_{n}<-d)-\Prob(D_{l}<-d)).\nonumber 
\end{align}
If either (a) $n$ and $l$ have the same parity or (b) $n$ odd and
$l$ even, the maximum in (\ref{eq:Blackwel exp char}) is achieved
for $d^{*}=0$: 
\[
\dist_{1}(u_{n},u_{l})=\Prob(D_{n}>0)-\Prob(D_{l}>0)-\Prob(D_{n}<0)+\Prob(D_{l}<0).
\]
Finally, if $n=2$ and $l=1$, then the maximum is obtained for $d^{*}=1$,
and 
\[
\dist_{1}(u_{1},u_{2})=2p(1-p)(2p-1).
\]
\end{prop}
We conjecture, although we are not able to show it, that if $n$ is
even and $l$ is odd, then the maximum in (\ref{eq:Blackwel exp char})
is achieved for $d^{*}=1$: 
\[
\dist_{1}(u_{n},u_{l})=2(1-p)(\Prob(D_{n}>1)-\Prob(D_{l}>1))-2p(\Prob(D_{n}<-1)-\Prob(D_{l}<-1)).
\]

\noindent We have some more explicit formulas for small number of
experiments. Quite surprisingly, although $u_{1}\preceq u_{2}\preceq u_{3}\preceq u_{4}$,
we have: 
\[
\dist_{1}(u_{2},u_{1})=\dist_{1}(u_{3},u_{1})=\dist_{1}(u_{3},u_{2})=\dist_{1}(u_{4},u_{2})=2p(1-p)(2p-1).
\]
And simple computations give: 
\begin{eqnarray*}
\dist_{1}(u_{4},u_{1}) & = & 2p(1-p)(2p-1)(1+3p-3p^{2}),\\
\dist_{1}(u_{4},u_{3}) & = & 6p^{2}(1-p)^{2}(2p-1),\\
\forall n,\dist_{1}(u_{n},u_{0}) & = & \Prob(D_{n}>0)-\Prob(D_{n}<0),\\
\forall n\geq1,\;\dist_{1}(u_{n},u_{1}) & = & 2(1-p)\Prob(D_{n}>1)-2p\Prob(D_{n}<-1).
\end{eqnarray*}

\subsection{Proof of Proposition \ref{pro11}}

We associate each element of $\Delta(K)$ with the probability $p\in[0,1]$
that the state is 1. Define for each integer $d$ in $\mathbb{Z}$,
\[
q_{d}=\frac{p^{d}}{p^{d}+(1-p)^{d}}=1-q_{-d}\in[0,1].
\]
$q_{d}$ is the belief of player 1 that $k=1$ after observing $s$
positive experiments and $s-d$ negative experiments, whatever is
$s$. Under $u_{n}$, the law of the difference $d$ is given by:
\[
\Prob_{n}(d)=\frac{1}{2}{n \choose \frac{n+d}{2}}p^{(n-d)/2}(1-p)^{(n-d)/2}(p^{d}+(1-p)^{d})=\Prob_{n}(-d)=\frac{1}{2}(\Prob(D_{n}=d)+\Prob(D_{n}=-d)),
\]
for all integers $d$ in $[-n,n]$ such that $n+d$ is even.

By \eqref{eq_d1}, we know that: 
\begin{equation}
\dist_{1}(u_{l},u_{n})=\max_{f\in{D}}\big|\tilde{u}_{n}(f)-\tilde{u}_{l}(f)\big|,\label{eq111}
\end{equation}
where $\tilde{u}_{n}(f)=\sum_{d=-n}^{n}f(q_{d})\Prob_{n}(d)$ is the
expectation of $f$ with respect to the probability induced by $u_{n}$
on the a posteriori. Similarly, $\tilde{u}_{l}(f)=\sum_{d=-n}^{n}f(q_{d})\Prob_{l}(d)$,
so $\tilde{u}_{n}(f)-\tilde{u}_{l}(f)=\sum_{d=-n}^{n}f(q_{d})(\Prob_{n}(d)-\Prob_{l}(d)).$
Let us first consider the optimization problem : 
\begin{equation}
e(u_{l},u_{n})=\max_{f\in{E}}\big|\sum_{d=-n}^{n}f(q_{d})(\Prob_{n}(d)-\Prob_{l}(d))\big|,\label{eq112}
\end{equation}
where $E$ is the set of convex 1-Lipschitz functions from $(\Delta(K),\|.\|_{1})$
to $\R$. Since $u_{n}\succeq u_{l}$, $e(u_{l},u_{n})=\max_{f\in{E}}\sum_{d=-n}^{n}f(q_{d})(\Prob_{n}(d)-\Prob_{l}(d))$.
W.l.o.g we can restrict attention to functions $f$ in $E$ such that
$f(1/2)=0$ and $f(q_{d})=f(q_{-d})$ for each $d$ (otherwise consider
$h(p)=(f(p)+f(1-p))/2$ for each $p$). Define for each $d$ in $\{0,...,n-1\}$,
the right slope of $f$ at $q_{d}$: 
\[
t_{d}=\frac{f(q_{d+1})-f(q_{d})}{q_{d+1}-q_{d}}.
\]
Since $n>l$, it is optimal to choose $t_{d}=1$ for $d\geq l$. The
quantity $\sum_{d=-n}^{n}f(q_{d})(\Prob_{n}(d)-\Prob_{l}(d))$ only
depends on $f$ though the slopes $t_{0},...,t_{n-1}$, and is affine
in $t_{0},...,t_{n-1}$. The constraints are $t_{d}\in[0,1]$ and
$t_{d+1}\geq t_{d}$ for each $d$, so for problem (\ref{eq112})
there is an optimal $f$ and $d^{*}\in\{0,...,l\}$ such that $t_{d}=0$
for $0\leq d<d^{*}$ and $t_{d}=1$ for $d\geq d^{*}$. Consequently,
we can restrict attention to the family of functions $(f_{d^{*}})$
from $\Delta(K)$ to $[0,1]$ such that: 
\[
f_{d^{*}}(q)=\left\{ \begin{array}{ccc}
|2q-1| & \mbox{if} & |q-1/2|\geq|q_{d^{*}}-1/2|\\
2q_{d^{*}}-1 & \mbox{if} & |q-1/2|\leq|q_{d^{*}}-1/2|
\end{array}\right.
\]
Since every function in the family belongs to $D$, we obtain: 
\begin{equation}
\dist_{1}(u_{l},u_{n})=\max_{d^{*}=0,...,l}\sum_{d=-n}^{n}f_{d^{*}}(q_{d})(\Prob_{n}(d)-\Prob_{l}(d)).\label{eq113}
\end{equation}
It turns out that $\sum_{d=-n}^{n}f_{d^{*}}(q_{d})\Prob_{n}(d)$ is
easy to compute: 
\begin{lem}
\label{lem111} 
\[
\sum_{d=-n}^{n}f_{d^{*}}(q_{d})\Prob_{n}(d)=(2q_{d^{*}}-1)+2\Prob(D_{n}>d^{*})(1-q_{d^{*}})-2q_{d^{*}}\Prob(D_{n}<-d^{*}).
\]
\end{lem}
\begin{proof}
\noindent 
\begin{eqnarray*}
\sum_{d=-n}^{n}f_{d^{*}}(q_{d})\Prob_{n}(d) & = & 2q_{d^{*}}-1+4\sum_{d=d^{*}+1}^{n}(q_{d}-q_{d^{*}})\Prob_{n}(d)\\
\; & = & 2q_{d^{*}}-1-2q_{d^{*}}\Prob_{n}(|d|>d^{*})+2\Prob(D_{n}>d^{*}),\\
\; & = & 2q_{d^{*}}-1-2q_{d^{*}}(\Prob(D_{n}>d^{*})+\Prob(D_{n}<-d^{*}))+2\Prob(D_{n}>d^{*}),\\
 & = & 2q_{d^{*}}-1+2\Prob(D_{n}>d^{*})(1-q_{d^{*}})-2q_{d^{*}}\Prob(D_{n}<-d^{*}).
\end{eqnarray*}
\end{proof}
\noindent \textbf{Remark:} Consider the decision problem where: $(k,c)$
is selected according to $u_{n}$, player 1 receives $c$ and has
to choose $i$ in $\{-1,0,1\}$, with payoff $g(k,0)=2q_{d^{*}}-1$,
$g(1,1)=g(0,0)=1$ and $g(1,0)=g(0,1)=-1$. The following strategy
is optimal: given $c$, compute the belief $q$ that $k=1$. Then
if $q\in[q_{-d^{*}},q_{d^{*}}]$ play the safe action $0$, if $q>q_{d^{*}}$
play $i=1$ and if $q<q_{-d^{*}}$ play $i=-1$. The payoff of this
strategy is precisely $\sum_{d=-n}^{n}f_{d^{*}}(q_{d})\Prob_{n}(d)$.
\\

The characterization (\ref{eq:Blackwel exp char}) follows from equation
(\ref{eq113}) and lemma \ref{lem111}.

\noindent Assume now that $n$ and $l$ have the same parity, that
is $n-l$ is even. For all $d$ in $\{-l,...,+l\}$, $\Prob_{n}(d)\leq\Prob_{l}(d)$,
so the maximum in \ref{eq113} is simply achieved for $d^{*}=0$.

Finally, suppose that $n$ is odd and $l$ is even. We show that the
maximum in equation (\ref{eq113}) is achieved for $d^{*}=0$. We
only need to consider $d$ even in $\{0,..,l\}$, and for each such
$d$ we define $h_{d}=f_{d}-f_{0}$. We have 
\[
\tilde{u}_{n}(h_{d})=2\sum_{d'=1}^{d-1}\Prob(|D_{n}|=d')(q_{d}-q_{d'}),\;\tilde{u}_{l}(h_{d})=2\sum_{d'=0}^{d-2}\Prob(|D_{l}|=d')(q_{d}-q_{d'}),
\]
and it is enough to show that $\tilde{u}_{n}(h_{d})\leq\tilde{u}_{l}(h_{d})$.
We are going to show that : 
\begin{equation}
\forall d'=0,2,...,d-2,\;\;\Prob(|D_{l}|=d')(q_{d}-q_{d'})\geq\Prob(|D_{n}|=d'+1)(q_{d}-q_{d'+1})\label{eq114}
\end{equation}
We have 
\begin{eqnarray*}
\frac{\Prob(|D_{l}|=d')(q_{d}-q_{d'})}{\Prob(|D_{n}|=d'+1)(q_{d}-q_{d'+1})} & = & \frac{q_{d}-q_{d'}}{q_{d}-q_{d'+1}}\frac{(p^{d'}+(1-p)^{d'})}{(p^{d'+1}+(1-p)^{d'+1})}{(p(1-p))}^{(l-n+1)/2}\frac{{l \choose (l+d')/2}}{{n \choose (n+d'+1)/2}},\\
 & \geq & \frac{1-q_{d'}}{1-q_{d'+1}}\frac{(p^{d'}+(1-p)^{d'})}{(p^{d'+1}+(1-p)^{d'+1})}{(p(1-p))}^{(l-n+1)/2}\frac{{l \choose (l+d')/2}}{{n \choose (n+d'+1)/2}},\\
 & = & (1-p)^{-1}{(p(1-p))}^{(l-n+1)/2}\frac{{l \choose (l+d')/2}}{{n \choose (n+d'+1)/2}}.
\end{eqnarray*}

\noindent Since $p\in[1/2,1]$, the minimum in $p$ is achieved for
$p=1/2$, so that: 
\[
\frac{\Prob(|D_{l}|=d')(q_{d}-q_{d'})}{\Prob(|D_{n}|=d'+1)(q_{d}-q_{d'+1})}\geq2^{n-l}\frac{{l \choose (l+d')/2}}{{n \choose (n+d'+1)/2}}.
\]
We finally show that $2^{n-l}{l \choose (l+d')/2}\geq{n \choose (n+d'+1)/2}$
for all $d'=0,...,l-2$. This inequality is true for $n=l+1$, and
$\frac{1}{2^{n}}{n \choose (n+d'+1)/2}$ decreases in $n$ (for $n$
even not smaller than $d'+1$).

\noindent Assume now that $n=2$ and $l=1$. The optimal $f$ in equation
\ref{eq113} should put maximal weight for $q=q_{0},q_{2},q_{-2}$
and minimal weight for $q_{-1},q_{1}$. One easily checks that the
maximum is obtained for $d^{*}=1$, and $d(u_{1},u_{2})=2p(1-p)(2p-1)$.

\section{Special cases\label{sec:Special-cases}}

We defined the distance $\dist$ between \emph{any} information
structures as a maximum of value difference across \emph{all} zero-sum
games. In this Section, we consider versions of the definition, where
we either restrict the space of games or information structures. Given
the restrictions, we obtain a tighter characterization of the distance.
Additionally, in all examples below, we show that the restricted spaces
of information structures have compact completion, which has important consequences for the limit theorems in repeated games.

\subsection{Single-agent decision problems}

\label{subsecsadp}

We introduced metric $\dist_{1}$ to analyze the
distance in single-agent decision problems. It is easy to see that
the distance $\dist_{1}$ depends only on the information of player
$1$: for any $u,u^{\prime},v,v^{\prime}$ such that $\marg_{K\times C}u=\marg_{K\times C}u^{\prime}$
with an analogous relation for $v$ and $v^{\prime}$, we have $\dist_{1}\left(u,v\right)=\dist_{1}\left(u^{\prime},v^{\prime}\right)$.

Let $\CU_{1}=\Delta(K\times\N)$ be the set of probabilities over
states and signals for player 1. From now on, we assume that $u,v$
are elements of $\CU_{1}$. Following the same method as in Theorem
\ref{thm1}, we show that 
\[
\dist_{1}(u,v)=\max\{\min_{q\in\CQ}\|q.u-v\|,\min_{q\in\CQ}\|q.v-u\|\},
\]
and the Blackwell characterization : $u\succeq v\Leftrightarrow\exists q\in Q,q.u=v$.

Further, notice that the distance depends only on the induced distributions
of conditional beliefs over $K$. Finally, we can show that if $D$
is the set of suprema of affine functions from $\Delta(K)$ to $[-1,1]$,
then 
\begin{equation}
\dist_{1}(u,v)=\sup_{f\in D}\left|\int_{p\in\Delta(K)}f(p)d{\tilde{u}}(p)-\int_{p\in\Delta(K)}f(p)d{\tilde{v}}(p)\right|,\label{eq_d1}
\end{equation}
where $\tilde{u}$ and $\tilde{v}$ denote the distributions of first-order
beliefs of Player 1 induced by the information structures $u$ and
$v$. Hence, $\CU_{1}$ under $\dist_{1}$ is totally bounded. Its
completion $\overline{\CU_{1}}$ is compact and 
\[
\overline{\CU_{1}}\simeq\Delta(\Delta(K)).
\]

\subsection{One-sided full information}

Let $\CU_{OF}$ be the subset of $\CU$ where with probability 1,
the signal of player 1 reveals both the state and the signal of player
2. For $u\in\CU_{OF}$, all what matters is the law $\tilde{u}$ of
the belief of player 2 about the state. From \citet{renault_venel_2017},
another characterization of $\dist$ can be given for elements of
${\CU_{OF}}$. 
\[
\dist(u,v)=\sup_{f\in D_{1}}\left(\int_{p\in\Delta(K)}f(p)d{\tilde{u}}(p)-\int_{p\in\Delta(K)}f(p)d{\tilde{v}}(p)\right),
\]
where 
\[
D_{1}=\{f,\forall p,q\in\Delta(K),\forall a,b\geq0,\;af(p)-bf(q)\leq\|ap-bq\|_{1}\}.
\]

Hence, $\CU_{OF}$ under $\dist$ is totally bounded; its completion
$\overline{\CU_{OF}}$ is compact and 
\[
\overline{\CU_{OF}}\simeq\Delta(\Delta(K)).
\]

\subsection{Public signals}

Set $\CU_{P}$ of information structures where both players receive
the same signal: $\overline{\CU_{P}}$ is compact, and homeomorphic
to $\Delta(\Delta(K))$. Here given $u$ in $\CU_{P}$, what matters
is the induced law $\tilde{u}$ on the common a posteriori of the
players on $K$.

\subsection{One player more informed}

\label{subsecopmi}

A generalization of the two previous cases is the set $\CU_{OM}$
of information structures where player 1 knows the signal of player
2, i.e. when the signal of player 1 is enough to deduce the signal
of player 2. $\overline{\CU_{OM}}$ is compact, and homeomorphic to
$\Delta(\Delta(\Delta(K)))$ (see \citet{mertens_repeated_1986}, \citet{gensbittel2014}).~

\subsection{Conditionally independent signals}

Set $\CU_{CI}$ of independent information structures : $\CU_{CI}$
is the set of $u$ in $\CU$ such that $u(c,d|k)=u(c|k)u(d|k)$ (the
signals $c$ and $d$ are conditionally independent given $k$). Here
$\overline{\CU_{CI}}$ is homeomorphic to $\Delta(\Delta(K)\times\Delta(L)).$~


\section{Value-based distance and bounds on equilibrium payoffs in non-zero-sum
games}

\label{seceqpayof}

In this section, we show that the value-based distance between information
structures contains information that is useful in some questions that
are relevant for non-zero-sum games. 

The section is divided into four parts. In the first part, we focus
on the distance between the sets of feasible payoffs induced by an
information structure and a game. More precisely, given an information
structure $u$ in $\CU$ and a non zero-sum payoff function $g:K\times I\times J\rightarrow[-1,1]^{2}$
(with $I$, $J$ finite), one naturally defines the non zero-sum Bayesian
game $\Gamma(u,g)$ and one can ask whether $\dist$ can be used to
measure the distance between the feasible payoffs in these games.
We provide a positive answer in three different cases (conditionally
independent signals, public signals and one-sided full information).

Next, we present two applications. First, we show that for any game
$g$ with common interests (common payoffs for both players), if structures
$u$ and $v$ belong to any of the three cases mentioned above, the
best equilibrium payoff in $\Gamma(u,g)$ and $\Gamma(v,g)$ differ
by at most $3\dist(u,v)$. 

The second application is concerned with the infinite repetition of
the games $\Gamma(u,g)$ and $\Gamma(v,g)$. As is well-known, in
such a repeated game the set of feasible and individually rational
payoffs is the limit, when $\delta\to1$, of the sets of $\delta$-discounted
Nash equilibrium payoffs (and under mild assumptions the limit of
the sets of $\delta$-discounted subgame-perfect Nash equilibrium
payoffs). The second corollary establishes a strong connection, depending
on $\dist(u,v)$, between the sets of feasible and individually rational
payoffs in $\Gamma(u,g)$ and $\Gamma(v,g)$.

Finally, we show by means of counterexample, that some assumptions
on the information structures are necessary to obtain bounds on the
non-zero-sum equilibrium payoffs in terms of value-based distance. 

\subsection{Bounds on distance between feasible payoffs}

We let $F(u,g)\subset\R^{2}$ be the set of feasible payoffs of $\Gamma(u,g)$
(the convex hull of the feasible payoffs with pure strategies), and
use the Hausdorff distance between non empty compact sets of $\R^{2}$,
endowed with $\|.\|_{\max}$.
\begin{prop}
\label{pro33} If the information structures $u$ and $v$ have conditionally
independent signals, the distance between $F(u,g)$ and $F(v,g)$
is at most $3\dist(u,v)$. 
\end{prop}
\begin{proof}

Write $\varepsilon=\dist(u,v)$. By Theorem \ref{thm1}, there exist garblings
$q_{1}$, $q_{2}$, $q_{3}$, $q_{4}$ such that $\|q_{1}.u-v.q_{2}\|\leq\varepsilon$
and $\|q_{3}.v-u.q_{4}\|\leq\varepsilon$. We first show that $u$
can be approximated by $q_{3}.v.q_{2}$.

Define $q_{3}.q_{1}$ for the garbling which given any signal $c$,
selects $c'$ according to $q_{1}(c)$ then finally $c''$ according
to $q_{3}(c')$. Similarly, let $q_{2}.q_{4}(d)=\sum_{d'}q_{4}(d)(d')q_{2}(d')$
for all $d$. Notice that $(q_{3}.q_{1}).v=q_{3}.(q_{1}.v)$, $u.(q_{2}.q_{4})=(u.q_{4}).q_{2}$
and $\|q.u'-q.v'\|\leq\|u'-v'\|$ for all $q$, $u'$, $v'$, we obtain:
\begin{eqnarray*}
\|q_{3}.q_{1}.u-u.(q_{2}.q_{4})\| & \leq & \|q_{3}.q_{1}.u-q_{3}.v.q_{2}\|+\|q_{3}.v.q_{2}-(u.q_{4}).q_{2}\|,\\
 & \leq & \|q_{1}.u-v.q_{2}\|+\|q_{3}.v-u.q_{4}\|,\\
 & \leq & 2\varepsilon.
\end{eqnarray*}
In particular, projecting over states and signals for player 1 gives
$\sum_{k,c}|q_{3}.q_{1}.u(k,c)-u(k,c)|\leq2\varepsilon$. Since $u$
has conditionally independent signals, we have: 
\begin{eqnarray*}
\|q_{3}.q_{1}.u(k,c)-u\| & = & \sum_{k,c,d}|q_{3}.q_{1}.u(k,c,d)-u(k,c,d)|,\\
 & = & \sum_{k,c,d}|q_{3}.q_{1}.u(k,c)u(d|k)-u(k,c)u(d|k)|,\\
 & = & \sum_{k,c}|q_{3}.q_{1}.u(k,c)-u(k,c)|\leq2\varepsilon.
\end{eqnarray*}
And we obtain $\|u-q_{3}.v.q_{2}\|\leq\|u-q_{3}.q_{1}.u\|+\|q_{3}.q_{1}.u-q_{3}.v.q_{2}\|\leq3\varepsilon$.
\\

Consider now a non zero-sum payoff function $g=(g_{1},g_{2})$ with
all payoffs in $[-1,1]$. We assume without loss of generality that
for some $L$, the set of actions for each player in $g$ is $I=J=\{0,...,L-1\}$.
Define the compact subsets $\CW_{L}(u)=\{q_{1}'.u.q'_{2},\;q'_{1},q'_{2}\in\CQ(L)\}$
and $\CW_{L}(v)=\{q'_{1}.v.q'_{2},q'_{1},\;q'_{2}\in\CQ(L)\}$ of
$\Delta(K\times I\times J)$. The set of feasible payoffs can be written:
\[
F(u,g)=\{(\langle g_{1},w\rangle,\langle g_{2},w\rangle),w\in\conv\CW_{L}(u)\},
\]
\[
\;F(v,g)=\{(\langle g_{1},w\rangle,\langle g_{2},w\rangle),w\in\conv\CW_{L}(v)\}.
\]
To show that the distance between $F(u,g)$and $F(v,g)$ is at most
3 $\varepsilon$, it is enough to prove that the distance (for $\|.\|=\|.\|_{1}$)
between $\conv\CW_{L}(u)$ and $\conv\CW_{L}(v)$, or simply between
$\CW_{L}(u)$ and $\CW_{L}(v)$, is at most 3 $\varepsilon$.

To conclude, let $u'$ be in $\CW_{L}(u)$, $u'$ can be written $u'=q'_{1}.u.q'_{2}$
for some $q'_{1}$, $q'_{2}$ in $\CQ(L)$. Since $\|u-q_{3}.v.q_{2}\|\leq3\varepsilon$,
we have $\|u'-q'_{1}.q_{3}.v.(q'_{2}.q_{2})\|\leq3\varepsilon$ and
the distance from $u'$ to $\CW_{L}(v)$ is\footnote{The proof shows that the distance between the sets of feasible payoffs
using mixed strategies in $\Gamma(u,h)$ and $\Gamma(v,h)$ also differ
by at most $3\dist(u,v)$.} at most 3 $\varepsilon$. 
\end{proof}
We have an analog of the proposition for information structures with
public signals.
\begin{prop}
\label{pro34} If the information structures $u$ and $v$ have public
signals, the distance between $F(u,g)$ and $F(v,g)$ is at most $\dist(u,v)$. 
\end{prop}
\begin{proof}
We assume that both players receive the same signals in each information
structure $u$, $v$. From the proof of proposition \ref{pro33},
it is sufficient to show that for each $L$ and $u'=q'_{1}.u.q'_{2}$
with $q'_{1},q'_{2}$ in $\CQ(L)$, the distance, with respect to
$\|.\|=\|.\|_{1}$, from $u'$ to the set $\conv\CW_{L}(v)$ is at
most $\dist(u,v)$.

We have the existence of garblings $q$, $q'$ such that $\|q.v-u.q'\|\leq\dist(u,v)$,
and projecting over states and signals for player 1 gives : $\sum_{k,c}|q.v(k,c)-u(k,c)|\leq\dist(u,v)$.
Define now the information structures $w$ and $u'$ by: 
\[
\forall k,c,d,\;w(k,c,d)=q_{.}v(k,c)\1_{d=c}\;\;{\rm and}\;\;v'=q'_{1}.w.q'_{2}.
\]

We have $\|u-w\|=\sum_{k,c}|w(k,c)-u(k,c)|\leq\dist(u,v)$, so $w$
is a good approximation of $u$. Now, $\|u'-v'\|\leq\|u-w\|\leq\dist(u,v)$,
and it is enough to show that $v'$ belongs to $\conv\CW_{L}(v)$.

We first assume that the set of signals having positive probability
under $v$ is finite and written $C$. The garbling $q$ selects,
independently for each signal $c$ in $C$, an element of $\N$ with
probability $q(c)$. Define now for any map $s:C\rightarrow\N$, 
\[
\lambda_{s}=\prod_{c\in C}q(c)(s(c)).
\]
$\lambda_{s}$ is the probability that for each $c$ in $C$, $q$
chooses $s(c)$ if the signal is $c$. Since $C$ is finite there
are countably many such maps $s$, and $\sum_{s}\lambda_{s}=1$. Define
also $v_{s}=s.v.s$ in $\CW_{\infty}(v)$ as: $(k,c',d')$ is selected
according to $v$ then the signals received by the players are $c=s(c'),d=s(d')$.
We have:
\begin{eqnarray*}
\sum_{s}\lambda_{s}v_{s}(k,c,d) & = & \sum_{s}\lambda_{s}\sum_{c'}v(k,c')\1_{c=d=s(c')},\\
 & = & \1_{c=d}\sum_{c'}v(k,c')\sum_{s}\lambda_{s}\1_{c=s(c')},\\
 & = & \1_{c=d}\sum_{c'}v(k,c')q(c')(c),\\
 & = & w(k,c,d).
\end{eqnarray*}
Then $v'=q'_{1}.w.q'_{2}=\sum_{s}\lambda_{s}\,q'_{1}.v_{s}.q'_{2}$.
For each $s$, $q'_{1}.v_{s}.q'_{2}$ belongs to $\CW_{L}(v)$, hence
$v'\in\conv\CW_{L}(v)$.

Finally, if the set of signals in $v$ is infinite, we consider a
sequence $(v_{n})$ with finite support such that $\|v-v_{n}\|\to0$,
and the corresponding sequence $(v'_{n})$ will have a limit point
$v'$ in $\conv\CW_{L}(v)$ satisfying $\|u'-v'\|\leq\dist(u,v)$. 
\end{proof}
We also have an analog of the proposition for information structures
with one-sided full information.
\begin{prop}
\label{pro35} If the information structures $u$ and $v$ have one-sided
full information, the distance between $F(u,g)$ and $F(v,g)$ is
at most $\dist(u,v)$. 
\end{prop}
\begin{proof}
Consider $u$, $v$ with one-sided full information for player 1.
As in the previous proofs, it is sufficient to prove that for each
$L$ and $q'_{1},q'_{2}$ in $\CQ(L)$, the distance between $u'=q'_{1}.u.q'_{2}$
and $\conv\CW_{L}(v)$ is at most $\dist(u,v)$, and we assume w.l.o.g.
that the set of signals for player 2 under $v$ is the finite set
$D$. There exist maps $f:\N\rightarrow K$, $h:\N\rightarrow D$
such that under $v$, if player 1's signal is $c$ then the state
is $k=f(c)$ and the signal of player 2 is $d=h(c)$. What is important
here is the law induced over states and signals for player 2, and
one might think as if player 1's signal was the pair (state, signal
for player 2).

By Theorem \ref{thm1} we have the existence of garblings $q$, $q'$ such that
$\|v.q-q'.u\|\leq\dist(u,v)$, and projecting gives : $\sum_{k,d}|v.q(k,d)-u(k,d)|\leq\dist(u,v)$.
Define now the information structures $w$ and $v'$ by: 
\[
\forall k,c,d,w(k,c,d)=v.q(k,d)\;u(c|k,d)\;\;{\rm and}\;\;v'=q'_{1}.w.q'_{2}.
\]

$\|u'-v'\|\leq\|u-w\|=\sum_{k,c,d}u(c|k,d)|u(k,d)-v.q(k,d)|\leq\dist(u,v)$,
and the goal is now to write $v'$ as an element of $\conv\CW_{L}(v)$.

Given a map $s:D\rightarrow\N$, define:

- the information structure $v_{s}$ in $\CW_{\infty}(v)$ by: first
select $(k,c',d')$ according to $v$, then the state is $k$, the
signal for player 2 is $s(d')$ and the signal of player 1 is chosen
according to $u$ conditionally on the state being $f(c')$ and the
signal of player 2 being $s(h(c'))$, i.e. player 1 receives the signal
$c$ with probability $u(c|f(c'),s(h(c'))$.

- the probability $\lambda_{s}=\prod_{d\in D}q(d)(s(d))$ that for
each $d$ in $D$, $q$ chooses $s(d)$ if the signal is $d$.

\begin{eqnarray*}
\sum_{s}\lambda_{s}v_{s}(k,c,d) & = & \sum_{s}\lambda_{s}\sum_{c',d'}v(k,c',d')\1_{d=s(d')}u(c|f(c'),s(h(c')),\\
 & = & \sum_{c',d'}v(k,c',d')\sum_{s}\lambda_{s}\1_{d=s(d')}u(c|k,d),\\
 & = & \sum_{d'}v(k,d')\;q(d')(d)\;u(c|k,d),\\
 & = & w(k,c,d).
\end{eqnarray*}
For each $s$, $q'_{1}.v_{s}.q'_{2}$ belongs to $\CW_{L}(v)$ and
we obtain $v'=\sum_{s}\lambda_{s}q'_{1}.v_{s}.q'_{2}\in\conv\CW_{L}(v)$,
completing the proof. 
\end{proof}

\subsection{Application to common interest games}

Propositions \ref{pro33}, \ref{pro34} and \ref{pro35} directly
imply the following result for games with common interests, where
the best equilibrium payoff and the best feasible (in pure, mixed
or correlated strategies) payoff coincide.
\begin{cor}
\label{cor34} Consider a non zero-sum payoff function $g=(g_{1},g_{1})$
with common payoffs for the players, and information structures $u$
and $v$ satisfying the assumptions of at least one of the propositions
\ref{pro33}, \ref{pro34}, \ref{pro35}. The best equilibrium payoff
for player 1 in $\Gamma(u,g)$ is at most $3\dist(u,v)$ from the
best equilibrium payoff for player 1 in $\Gamma(v,g)$. 
\end{cor}

\subsection{Application to repeated non-zero-sum games}

The propositions also imply that the sets of feasible and individually
rationals payoffs of $\Gamma(u,g)$ and $\Gamma(v,g)$ are closely
related if $\dist(u,v)$ is small. In the following corollary, we
denote by $m_{1}(u,g)=\val(u,g_{1})$ and $m_{2}(u,g)=-\val(u,-g_{2})$
the respective independent minmax of the players in the game $\Gamma(u,g)$.
\begin{cor}
\label{cor33} Consider a non zero-sum payoff function $g$, and information
structures $u$ and $v$ satisfying the assumptions of at least one
of the propositions \ref{pro33}, \ref{pro34}, \ref{pro35}. Let
$x=(x_{1},x_{2})$ be a feasible payoff in the game $\Gamma(u,g)$
satisfying $x_{i}\geq m_{i}(u,g)+4\dist(u,v)$ for $i=1,2$. Then
$x$ is $3\dist(u,v)$-close to a payoff which is feasible and individually
rational in $\Gamma(v,g)$. 
\end{cor}
\begin{proof}
By definition of the value-based distance, $|m_{i}(u,g)-m_{i}(v,g)|\leq\dist(u,v)$
for each player $i$. By one of the propositions, $x$ is $3\dist(u,v)$-close
to a payoff $y$ in $F(v,g)$. For each $i=1,2$, $y_{i}\geq x_{i}-3\dist(u,v)\geq m_{i}(u,g)+\dist(u,v)\geq m_{i}(v,g)$,
so $y$ is individually rational in the game $\Gamma(v,g)$. 
\end{proof}
\noindent \textbf{Remark:} From the proofs of the propositions, one
can see that the conclusion of corollary \ref{cor33} holds more generally
as soon as : $u$ has conditionally independent signals and $v$ is
arbitrary, or $v$ has one-sided full information and $u$ is arbitrary,
or both $u$ and $v$ have public signals. Same for corollary \ref{cor34}
with the conclusion: the best equilibrium payoff in $\Gamma(v,g)$
is at least the best equilibrium payoff in $\Gamma(u,g)$ minus $3\dist(u,v)$. 

\subsection{Counterexample}

We finally provide a counter-example to the three propositions and
the two corollaries when under $u$ player 1 is not more informed
and the signals are conditionally dependent.
\begin{example}
$K=C=D=\{0,1\}$. Under $u$, the signals $c$ and $d$ are uniformly
and independently distributed, and $k=c+d\mod2$. Under $v$, $c=d=0$,
and $k$ is uniformly selected. It is easy to see that $\dist(u,v)=0$.
However, consider a non zero-sum payoff function $h$ with $g_{1}(k,i,j)=g_{2}(k,i,j)=1$
if $k=i+j\mod2$ and $g_{1}(k,i,j)=g_{2}(k,i,j)=-1$ otherwise. The
payoff $x=(1,1)$ is feasible in the game $\Gamma(u,g)$, but no payoff
with positive coordinates is feasible in $\Gamma(v,g)$. 
\end{example}

\section{Other properties of characterization\label{sec:Other-properties-of}}

\subsection{Comparison of information}

The following Corollary is a direct corollary of Theorem \ref{thm1}: 
\begin{cor}
\label{cor: d charact comparison}For all information structures $u,v$,
\begin{equation}
\sup_{g\in\CG}\left(\val(v,g)-\val(u,g)\right)=\inf_{u'\preceq u,v'\succeq v}\|u'-v'\|.\label{eq:Thm1 as comparison}
\end{equation}
\end{cor}
This observation provides an additional interpretation to the characterization
from Theorem \ref{thm1}: the maximum gain from replacing information structure
$u$ by $v$ is equal to the minimum total variation distance between
the set of information structures that are worse than $u$ and those
that are better than $v$.

\subsection{Optimal strategies}

Another useful property of Theorem \ref{thm1} is that the garblings in equation
\ref{eq:d characterization} can be used to transform optimal strategies in one structure
to approximately optimal strategies on another structure. Moreover,
the transformation does not depend on the particular payoffs considered.

More precisely, we say that strategy $\sigma$ of player 1 is $\varepsilon$-optimal
in game $g$ on structure $u$ if for any strategy $\tau$ of player
2, the payoff of $\sigma$ against $\tau$ is no smaller than $\val\left(u,g\right)-\varepsilon$.
We similarly define $\varepsilon$-optimal strategies for player 2.

For a strategy $\sigma\in\CQ$ and a garbling $q_{1}\in\CQ$, define
$\sigma.q_{1}$ in $\CQ$ by $\sigma.q_{1}(c)=\sum_{c^{\prime}}q_{1}(c{}^{\prime}|c)\sigma(c{}^{\prime})$
for each signal $c$: player 1 receives signal $c$, then selects
$\chi$ according to $q_{1}(c)$ and plays $\sigma(\chi)$. We have 
\begin{prop}
\label{prop: optimal strategies}Fix $u,v$ in $\CU$ and let $q_{1}$
and $q_{2}$ in $\CQ$ satisfy 
\[
\sup_{g\in\CG}\left(\val(v,g)-\val(u,g)\right)=\|q_{1}.u-v.q_{2}\|.
\]
Then, if $\sigma$ is an optimal strategy in $g$ on $v$, then $\sigma.q_{1}$
is a $2\dist\left(u,v\right)$-optimal strategy in $g$ on $u$. Similarly,
if $\tau$ is optimal for player 2 in $g$ on $u$, then $\tau.q_{2}$
is $2\dist\left(u,v\right)$-optimal for player 2 in $g$ on $v$. 
\end{prop}
\begin{proof}
Consider $\sigma$ an optimal strategy in $g$ on $v$, and $\tau$
in $\CQ$ arbitrary. The payoff induced by $(\sigma.q_{1},\tau)$
in the game $(u,g)$ is: 
\begin{eqnarray*}
\sum_{k,c,d}u(k,c,d)g(k,\sigma.q_{1}(c),\tau(d)) & = & \sum_{k,c,d}q_{1}.u(k,c,d)g(k,\sigma(c),\tau(d)),\\
 & \geq & \sum_{k,c,d}v.q_{2}(k,c,d)g(k,\sigma(c),\tau(d))-\|q_{1}.u-v.q_{2}\|,\\
 & \geq & \sum_{k,c,d}v(k,c,d)g(k,\sigma(c),\tau.q_{2}(d))-d(u,v),\\
 & \geq & \val(v,g)-\dist(u,v),\\
 & \geq & \val(u,g)-2\dist(u,v).
\end{eqnarray*}
The dual property is proved similarly. 
\end{proof}

\section{Uncountable information structures\label{sec:Uncountable-information-structur}}

\subsection{Uncountable information structures}

Let $K$ be a compact metric space endowed with its Borel $\sigma$-algebra
$\mathcal{B}(K)$. An (uncountable) information structure over $K$
is defined as a pair of measurable spaces $(S_{1},\mathcal{A}_{1})$
and $(S_{2},\mathcal{A}_{2})$ and a probability measure $u$ over
the product measurable space $K\times S_{1}\times S_{2}$. Here, $S_{1}$
denotes the set of signals observed by Player 1, and $S_{2}$ the
set of signals observed by Player $2$.

Due to the construction of \citet{mertzam:85} of the universal belief
space $\Omega=K\times\Theta_{1}\times\Theta_{2}$, one can associate
to $u$ a unique consistent probability $P\in\Pi\subset\Delta(\Omega)$,
such that the induced canonical information structure on the universal
belief space (in which the signal of player $i$ is her hierarchy
of beliefs in $\Theta_{i}$) is equivalent to $u$ in the sense that
\[
\forall g\in\CG,\val(u,g)=\val(P,g).
\]
(see e.g. Theorem III.2.4 and Propositions III.4.2 and III.4.4 in
\citet{mertens_sorin_zamir_2015}. In particular, the value is well-defined
for any consistent probability and for any information structure.)
Thus, in order to capture all the equivalence classes of arbitrary
measurable information structures (possibly uncountable), it is sufficient
to consider canonical ones, with signals being a hierarchy of beliefs.
However, considering consistent probabilities is sometimes confusing
in the sense that when modifying a canonical information structure
with a garbling from $\Theta_{i}$ to $\Theta_{i}$, the modified
information structure is in $\Delta(\Omega)$ but not in general a
consistent probability. Instead of working with information structures
in $\Delta(\Omega)$ (as it is often done in \citet{mertens_sorin_zamir_2015}
or \citet{gossner_value_2001}) and in order to avoid any confusions,
we find convenient to use the set ${\CU}_{c}$ of information structures
with signals in $[0,1]$. This set is sufficiently large in the sense
that any measurable information structure is equivalent to an element
of ${\CU}_{c}$ (since it is equivalent to a canonical information
structure in $\Pi$, which is itself equivalent to an element of ${\CU}_{c}$
by relabeling the signals), and is stable by transformations based
on garblings from $[0,1]$ to $[0,1]$.

Define ${\CU}_{c}=\Delta(K\times[0,1]\times[0,1])$ as the set of
information structures where signals of the players are in $[0,1]$.
Let ${\CU}_{c}^{*}$ be the set of equivalence classes for the relation
\[
u\simeq v\Leftrightarrow\forall g\in\CG,\val(u,g)=\val(v,g).
\]
$\CU^{*}$ can be identified as the subset of ${\CU}_{c}^{*}$ made
by equivalence classes which contain an element with countable support.

Let $\Phi$ be the map from ${\CU}_{c}$ to $\Pi$ which associates
to the information structure $u$ the induced consistent probability
over the (Mertens-Zamir) universal belief space $\Omega$. This map
is constant over equivalence classes since the value $\val(u,g)$
depends only on $\Phi(u)$ for every $g\in\CG$, and thus $\Phi$
induces a map from ${\CU}_{c}^{*}$ to $\Pi$, also denoted $\Phi$.
Moreover, $\Phi$ is one-to-one since the value functions of finite
games separate points in $\Pi$ (see Theorem 12 in \citet{gossner_value_2001}).
We claim that this map is also onto. Indeed, for every $\mu\in\Pi$,
one can associate a canonical information structure. Then, since the
set of coherent belief hierarchies $\Theta_{i}$ is compact metric,
there exists a Borel isomorphism $\psi_{i}$ from $\Theta_{i}$ to
$[0,1]$\footnote{According to the Borel isomorphism theorem, all the uncountable standard
Borel spaces are Borel isomorphic. Standard Borel spaces being Borel
subsets of complete separable metric spaces, compact metric spaces
are standard Borel spaces.}, which allows to relabel the signals. The probability $u$ over $(k,\psi_{1}(\theta_{1}),\psi_{2}(\theta_{2}))$
is such that $\Phi(u)=\mu$ by construction and this proves the claim.

\subsection{Total variation norm and topologies on garblings\label{subsection_TV}}

Let us recall a few facts about the total variation distance. \p
Let $X$ denote a compact metric space, $\mathcal{F}_{1}(X)$ denote
the set of (Borel) measurable functions from $X$ to $[-1,1]$ and
$\mathcal{C}_{1}(X)$ denote the set of continuous functions from
$X$ to $[-1,1]$. Then, the total variation norm satisfies 
\begin{align*}
\forall\mu,\nu\in\Delta(X),\;\|\mu-\nu\|_{TV} & =\sup_{A\in\mathcal{B}(X)}|\mu(A)-\nu(A)|\\
 & =\frac{1}{2}\sup_{f\in\mathcal{F}_{1}(X)}\int_{X}fd(\mu-\nu)\\
 & =\frac{1}{2}\sup_{f\in\mathcal{C}_{1}(X)}\int_{X}fd(\mu-\nu)
\end{align*}
The first equality is the definition, the second is a classical exercise,
and the third is obtained by using a standard approximation argument
(e.g. by using that continuous functions are dense in $L^{1}(\mu+\nu)$
and a truncation argument). \p We also have a useful formula. Let
$\pi$ be another positive $\sigma$-finite measure such that $\mu$
and $\nu$ are absolutely continuous with respect to $\pi$ (e.g.,
choose $\pi=\mu+\nu$). Then 
\[
2\|\mu-\nu\|_{TV}=\int_{X}\left|\frac{d\mu}{d\pi}-\frac{d\nu}{d\pi}\right|d\pi.
\]
Indeed, the right-hand side is an upper bound since for any $f\in\mathcal{F}_{1}(X)$
\[
\int_{X}fd(\mu-\nu)=\int_{X}f\left(\frac{d\mu}{d\pi}-\frac{d\nu}{d\pi}\right)d\pi,
\]
and one may choose $f=sgn(\frac{d\mu}{d\pi}-\frac{d\nu}{d\pi})$ to
prove that it is attained. \p Remark: When $X$ is countable, we
have $\|\mu-\nu\|_{TV}=\frac{1}{2}\|\mu-\nu\|_{1}$. \p Let now $X$
and $Y$ denote compact metric spaces and fix $\mu\in\Delta(X)$.
Let $\mathcal{T}$ denote the set of equivalence classes of transitions
probabilities from $X$ to $Y$ with respect to the relation of equality
$\mu$ almost everywhere. Recall that a transition probability is
simply a Borel measurable map from $X$ to $\Delta(Y)$ when $\Delta(Y)$
is endowed with the weak topology. \p $\mathcal{T}$ is a compact
metrizable space when endowed with the topology $\tau_{\mu}$ defined
by 
\[
q_{n}\rightarrow q\Leftrightarrow\forall f\in\mathcal{L},\int_{X}\left(\int_{Y}f(x,y)dq^{n}(x)(y)\right)d\mu(x)\rightarrow\int_{X}\left(\int_{Y}f(x,y)dq(x)(y)\right)d\mu(x),
\]
where $\mathcal{L}$ denotes the set of bounded measurable maps on
$X\times Y$ that are continuous with respect to the second variable
(see e.g. Theorems 2.2 and 2.3 in \citet{balder1988generalized}).

\subsection{Extension of Theorem 1}

Let $C=D=[0,1]$, and let us consider the class of game payoffs $\CG_{c}$
of continuous maps from $K\times C\times D$ to $[-1,1]$. Recall
that from Proposition III.4.2 in \citet{mertens_sorin_zamir_2015},
for every $u\in{\CU}_{c}$ and every $g\in\CG_{c}$, the value $\val(u,g)$
is well defined and both players have optimal strategies. \p Let
$\CQ_{c}$ denote the set of measurable functions from $[0,1]$ to
$\Delta([0,1])$. \p Given $u\in{\CU}_{c}$ and $q\in\CQ_{c}$, define
$q.u$ and $u.q$ as before, i.e. as the unique probability measures
that satisfy for every bounded measurable function $f$: 
\[
\int_{K\times C\times D}fd(q_{1}.u)=\int_{K\times C\times D}\left(\int_{C}f(k,c',d)dq_{1}(c'|c)\right)du(k,c,d).
\]
\[
\int_{K\times C\times D}fd(u.q_{2})=\int_{K\times C\times D}\left(\int_{D}f(k,c,d')dq_{2}(d'|d)\right)du(k,c,d).
\]

\begin{thm}
\textbf{\label{thm:(Extension-of-Theorem}(Extension of Theorem \ref{thm1})}
For all $u,v\in{\CU}_{c}$: 
\begin{align*}
\sup_{g\in\CG}(\val(v,g)-\val(u,g)) & =\sup_{g\in\CG_{c}}(\val(v,g)-\val(u,g))\\
 & =2\min_{q_{1},q_{2}\in\CQ_{c}}\|q_{1}.u-v.q_{2}\|_{TV}.
\end{align*}
\end{thm}
As for Theorem \ref{thm1}, the result of \citet{peski_comparison_2008} can be extended
to arbitrary information structures. 
Precisely, for any $u,v\in\CU_c$,  write $u\succeq v$ if for all $g\in\CG$, $\val(u,g)-\val(v,g)\geq 0$.
By the monotony of the value with respect to information in zero-sum
games, we have $q.u\preceq u\preceq u.q$ for each garbling $q \in \CQ_c$ and 
Theorem \ref{thm:(Extension-of-Theorem} implies the following result. 
\begin{cor} \label{cor} For all $u,v\in \CU_c$,  $u\succeq v$
$\Longleftrightarrow$ there exists $q_{1}$, $q_{2}$ in $\CQ_c$ s.t.
$q_{1}.u=v.q_{2}$. \end{cor}

\begin{rem}
It is important to recall that signals sets can be replaced by any
other standard Borel space (by relabeling) for the following equality
to hold 
\[
\sup_{g\in\CG}(\val(v,g)-\val(u,g))=2\min_{(q_{1},q_{2})\in\CQ_{1}\times\CQ_{2}}\|q_{1}.u-v.q_{2}\|_{TV},
\]
where $\CQ_{1}$ and $\CQ_{2}$ are sets of transition probabilities
between appropriate spaces of signals. The same remark applies to Corollary \ref{cor}.
\end{rem}

\subsubsection{Proof of Theorem \ref{thm:(Extension-of-Theorem}}

We first prove the second equality. Note that for any $g\in\CG_{c}$,
the duality product $\langle g,u\rangle=\int_{K\times C\times D}gdu$
is well defined and corresponds to the payoff $\gamma_{u,g}(Id,Id)$,
where $Id\in\CQ_{c}$ is the strategy that plays with probability
one the signal received. A straightforward computation leads to 
\[
\gamma_{u,g}(q_{1},q_{2})=\langle g,q_{1}.u.q_{2}\rangle.
\]
Consequently, 
\[
\val(u,g)=\max_{q_{1}\in\CQ_{c}}\min_{q_{2}\in\CQ_{c}}\langle g,q_{1}.u.q_{2}\rangle=\min_{q_{2}\in\CQ_{c}}\max_{q_{1}\in\CQ_{c}}\langle g,q_{1}.u.q_{2}\rangle.
\]
Since both players can play the $Id$ strategy in $\Gamma({u,g})$,
we have for all $u\in{\CU}_{c}$ and $g\in\CG_{c}$ that 
\[
\inf_{q_{2}\in\CQ_{c}}\langle g,u.q_{2}\rangle\leq\val(u,g)\leq\sup_{q_{1}\in\CQ_{c}}\langle g,q_{1}.u\rangle.
\]
Notice also that for all $u$, $v$ in ${\CU}_{c}$, 
\[
2\|u-v\|_{TV}=\sup_{g\in\CG_{c}}\langle g,u-v\rangle.
\]
Fix $u$, $v$ in ${\CU}_{c}$. For $g\in\CG_{c}$, we have 
\[
\inf_{q_{1},q_{2}\in\CQ_{c}}\langle g,v.q_{2}-q_{1}.u\rangle\leq\val(v,g)-\val(u,g),
\]
so 
\begin{equation}
\sup_{g\in\CG_{c}}\left(\val(v,g)-\val(u,g)\right)\geq\sup_{g\in\CG_{c}}\inf_{q_{1},q_{2}\in\CQ_{c}}\langle g,v.q_{2}-q_{1}.u\rangle.\label{eeq2}
\end{equation}
For $g\in\CG_{c}$, $q_{1},q_{2}\in\CQ_{c}$, by monotony of the value
with respect to information, we have 
\[
\val(v.q_{2},g)\geq\val(v,g)\text{ and }\val(u,g)\geq\val(q_{1}.u,g),
\]
so that 
\[
\val(v,g)-\val(u,g)\leq d\left(q_{1}.u,v.q_{2}\right)\leq2\|q_{1}.u-v.q_{2}\|_{TV}.
\]
Hence 
\begin{align}
\sup_{g\in\CG_{c}}\left(\val(v,g)-\val(u,g)\right) & \leq\inf_{q_{1},q_{2}\in\CQ_{c}}2\|q_{1}.u-v.q_{2}\|_{TV}\\
 & =\inf_{q_{1},q_{2}\in\CQ_{c}}\sup_{g\in\CG_{c}}\langle g,v.q_{2}-q_{1}.u\rangle.\label{eeq3}
\end{align}
We are now going to show that 
\begin{equation}
\sup_{g\in\CG_{c}}\inf_{q_{1},q_{2}\in\CQ_{c}}\langle g,v.q_{2}-q_{1}.u\rangle=\inf_{q_{1},q_{2}\in\CQ_{c}}\sup_{g\in\CG_{c}}\langle g,v.q_{2}-q_{1}.u\rangle.\label{eeq1}
\end{equation}
Together with inequalities \ref{eeq2} and \ref{eeq3}, it will imply
the result. \p To prove \ref{eeq1}, we will apply a variant of Sion's
theorem (see e.g., Proposition I.1.3 in \citet{mertens_sorin_zamir_2015})
to the zero-sum game with strategy spaces $\CG_{c}$ for the maximizer,
$\CQ_{c}^{2}$ for the minimizer, and payoff $h(g,(q_{1},q_{2}))=\langle g,v.q_{2}-q_{1}.u\rangle$.
\p At first, note that the strategy sets $\CG_{c}$ and $\CQ_{c}^{2}$
are convex, and that $h$ is bilinear. \p To avoid any confusion,
let us denote $\CQ_{c}^{2}=Q_{1}\times Q_{2}$. Let $u_{C}$ denote
the marginal distribution of $u$ over the set $C$ of signals of
player $1$ and $v_{D}$ the marginal distribution of $v$ over the
set $D$ of signals of player $2$. Note that the payoff function
of the above game do not change if we replace $q_{1}$ by another
map which is equal $u_{C}$-almost everywhere to $q_{1}$, and a similar
remark holds for $q_{2}$. Therefore, we may consider that $Q_{1}$
is the set of equivalence classes of transitions w.r.t to equality
$u_{C}$ almost everywhere, endowed with the topology $\tau_{u_{C}}$
defined in section \ref{subsection_TV}, and that $Q_{2}$ is the
set of equivalence classes of transitions w.r.t to equality $v_{D}$
almost everywhere, endowed with the topology $\tau_{v_{D}}$. The
set $Q_{1}\times Q_{2}$ is thus metric compact for the associated
product topology.

It remains to check that for every $g\in\CG_{c}$, the map $(q_{1},q_{2})\rightarrow h(g,(q_{1},q_{2}))$
is continuous. Note first that is is the sum of a function of $q_{1}$
and a function of $q_{2}$. \p Let $\nu:C\rightarrow\Delta(K\times D)$
denote a version of the conditional law of $(k,d)$ given $c$ under
the probability $u$. Then, we have 
\begin{align*}
\langle g,q_{1}.u\rangle & =\int_{K\times C\times D}\left(\int_{C}g(k,c',d)dq_{1}(c'|c)\right)du(k,c,d)\\
 & =\int_{C}\int_{K\times D}\left(\int_{C}g(k,c',d)dq_{1}(c'|c)\right)d\nu(k,d|c)du_{C}(c)\\
 & =\int_{C}\int_{C}\left(\int_{K\times D}g(k,c',d)d\nu(k,d|c)\right)dq_{1}(c'|c)du_{C}(c)
\end{align*}
and the above expression is continuous with respect to $q_{1}$ using
the definition of the topology $\tau_{u_{C}}$ since using bounded
convergence, the function 
\[
f(c,c'):=\int_{K\times D}g(k,c',d)d\nu(k,d|c),
\]
is continuous with respect to the second variable. \p A similar argument
holds for $q_{2}$ and this concludes the proof. \p It remains to
prove the first equality. Let $g\in\CG_{c}$ such that $\val(u,g)>\val(v,g)+3\varepsilon$
for some $\varepsilon>0$. According to Proposition III.4.2 in \citet{mertens_sorin_zamir_2015}
(up to relabeling the signals), Player 1 has an $\varepsilon$-optimal
strategy in the game $\Gamma(u,g)$ taking values almost surely in
some finite set $A_{1}\subset C$. Similarly, player 2 has an $\varepsilon$-optimal
strategy in the game $\Gamma(v,g)$ taking values almost surely in
some finite set $A_{2}\subset D$. It results that the game with action
spaces $A_{1},A_{2}$ and payoff $\hat{g}=g|_{K\times A_{1}\times A_{2}}$
is such that 
\[
\val(u,\hat{g})\geq\val(v,\hat{g})+\varepsilon.
\]
The conclusion follows.

\subsection{Extensions of the results of section \ref{sec:Applications}}

We now explain how to extend the 5 Propositions of section \ref{sec:Applications}
to uncountable information structures. \p At first, note that Propositions
\ref{prop:diameter} and \ref{prop. Info complements} hold with the same proofs. Then, Proposition \ref{prop: SA=00003D00003D00003D00003D00003D00003D00003D00003D00003D00003DGames}
can be generalized by using the same ideas and a few technical adaptations
described below. \p Let us start with general properties of the single-agent
distance $\dist_{1}$. Let ${\CG}_{1,c}=\{g:K\times C\rightarrow[-1,1],|\,g\text{ continuous }\}$,
which can be identified with a subset of $\CG_{c}$ made of functions
that do not depend on $d$. Then, as in the second part of the above
proof of Theorem \ref{thm1}, we have 
\[
\forall u,v\in{\CU}_{c},\;\dist_{1}(u,v)=\sup_{g\in\CG_{1}}|\val(u,g)-\val(v,g)|=\sup_{g\in\CG_{1,c}}|\val(u,g)-\val(v,g)|.
\]
Recall that $C=D=[0,1]$ and define the set of single-agent information
structures as ${\CU}_{1,c}=\Delta(K\times C)$. Note that given $u\in{\CU}_{c}$,
$\marg_{K\times C}u\in{\CU}_{1,c}$. Define for $u',v'\in{\CU}_{1,c}$,
$\dist'_{1}(u',v')=\sup_{g\in{\CG}_{1,c}}|\val(v',g)-\val(u',g)|$.
For any $u,v\in{\CU}_{c}$, 
\begin{equation}
\dist_{1}(u,v)=\dist'_{1}(u',v')=\max\{\min_{q\in\CQ_{c}}2\|u'-q.v'\|_{TV},\min_{q\in\CQ_{c}}2\|q.u'-v'\|_{TV}\}\label{eeq:d_1_marg}
\end{equation}
where $u'=\marg_{K\times C}u$, $v'=\marg_{K\times C}v$, and $q.u'$
is the probability defined by 
\[
\int_{K\times C}fd(q.u')=\int_{K\times C}\int_{C}f(k,c')dq(c'|c)du'(k,c),
\]
for all bounded measurable function $f$. This result can be obtained
by mimicking (and simplifying) the arguments of the proof of Theorem
\ref{thm1}.

\subsubsection{Extension of Proposition \ref{prop: SA=00003D00003D00003D00003D00003D00003D00003D00003D00003D00003DGames}}
\begin{prop}
Suppose that $u,v\in\overline{\CU}$ are two information structures
with conditionally independent information such that $\marg_{K\times D}u=\marg_{K\times D}v$.
Then, $\dist\left(u,v\right)=\dist_{1}\left(u,v\right)$. 
\end{prop}
Let us first introduce some additional notation. Let $u_{S}=\marg_{S}(u)$
for $S=K,K\times C,K\times D$. Let $u_{1}:K\rightarrow\Delta(C)$
denote a version of the conditional law of $c$ given $k$ induced
by $u$, and $u_{2}:K\rightarrow\Delta(D)$ denote a version of the
conditional law of $d$ given $k$ induced by $u$. Let $v_{S}$,
$v_{1}:K\rightarrow\Delta(C)$ and $v_{2}:K\rightarrow\Delta(D)$
be defined as above, and note that by assumptions $v_{K}=u_{K}$ and
$u_{2}=v_{2}$, and that the conditional law of $(c,d)$ given $k$
induced by $u$ is given by $u_{1}(.|k)\otimes u_{2}(.|k)$ (and similarly
for $v$). \p Let us fix a pair of garblings $q_{1},q_{2}$. Consider
the measure $\pi_{2}=\frac{1}{2}(u_{K\times D}+(u.q_{2})_{K\times D})$
on $K\times D$ and the Radon-Nikodym densities 
\[
\phi_{2}=\frac{du_{K\times D}}{d\pi_{2}},\;\psi_{2}=\frac{d(u.q_{2})_{K\times D}}{d\pi_{2}}.
\]
Similarly, let $\pi_{1}=\frac{1}{2}(u_{K\times C}+(q_{1}.v)_{K\times C})$
and 
\[
\phi_{1}=\frac{du_{K\times C}}{d\pi_{1}},\;\psi_{1}=\frac{d(q_{1}.v)_{K\times C}}{d\pi_{1}}.
\]
Note that $\marg_{K}(\pi_{1})=\marg_{K}(\pi_{2})=u_{K}$. \p For
any and any bounded measurable function $f$ we have using the conditional
independence assumption: 
\begin{align*}
\int_{K\times C\times D}fd(u.q_{2}) & =\int_{K\times D}\int_{C}f(k,c,d)du_{1}(c|k)d(u.q_{2})_{K\times D}(k,d)\\
 & =\int_{K\times D}\int_{C}f(k,c,d)du_{1}(c|k)\psi_{2}(k,d)d\pi_{2}(k,d)\\
 & =\int_{K\times C}\int_{D}f(k,c,d)\psi_{2}(k,d)d\pi_{2}(d|k)du_{K\times C}(k,c)\\
 & =\int_{K\times C}\int_{D}f(k,c,d)\psi_{2}(k,d)d\pi_{2}(d|k)\phi_{1}(k,c)d\pi_{1}(k,c)\\
 & =\int_{K}\int_{D}\int_{C}f(k,c,d)\psi_{2}(k,d)\phi_{1}(k,c)d\pi_{1}(c|k)d\pi_{2}(d|k)du_{K}(k).
\end{align*}
Similarly, one has 
\begin{align*}
\int_{K\times C\times D}fd(q_{1}.v)=\int_{K}\int_{D}\int_{C}f(k,c,d)\phi_{2}(k,d)\psi_{1}(k,c)d\pi_{1}(c|k)d\pi_{2}(d|k)du_{K}(k).
\end{align*}

We obtain 
\begin{align*}
\int_{K\times C\times D} & fd(u.q_{2}-q_{1}.v)\\
 & =\int_{K}\int_{D}\int_{C}f(k,c,d)\left[\psi_{2}(k,d)\phi_{1}(k,c)-\phi_{2}(k,d)\psi_{1}(k,c)\right]d\pi_{2}(d|k)d\pi_{1}(c|k)du_{K}(k),
\end{align*}
so that the supremum over all measurable $f$ bounded by $1$ is equal
to 
\begin{align*}
2 & \|u.q_{2}-q_{1}.v\|_{TV}=\int_{K}\int_{D}\int_{C}\left|\psi_{2}(k,d)\phi_{1}(k,c)-\phi_{2}(k,d)\psi_{1}(k,c)\right|d\pi_{2}(d|k)d\pi_{1}(c|k)du_{K}(k)\\
 & =\int_{K}\int_{D}\int_{C}\left|(\psi_{2}(k,d)-\phi_{2}(k,d))\phi_{1}(k,c)+\phi_{2}(k,d)(\phi_{1}(k,c)-\psi_{1}(k,c))\right|d\pi_{2}(d|k)d\pi_{1}(c|k)du_{K}(k)
\end{align*}
Because $\left|x+y\right|\geq|x|+\text{sgn}(x)y$ for each $x,y\in\R$,
we have 
\begin{align*}
2 & \|u.q_{2}-q_{1}.v\|_{TV}\\
 & \geq\int_{K}\int_{D}\int_{C}\phi_{2}(k,d)\left|\phi_{1}(k,c)-\psi_{1}(k,c)\right|d\pi_{2}(d|k)d\pi_{1}(c|k)du_{K}(k)\\
 & +\int_{K}\int_{D}\int_{C}(\psi_{2}(k,d)-\phi_{2}(k,d))\phi_{1}(k,c)\text{sgn}(\phi_{1}(k,c)-\psi_{1}(k,c))d\pi_{2}(d|k)d\pi_{1}(c|k)du_{K}(k)\\
 & =\int_{K}\int_{D}\int_{C}\phi_{2}(k,d)\left|\phi_{1}(k,c)-\psi_{1}(k,c)\right|d\pi_{2}(d|k)d\pi_{1}(c|k)du_{K}(k)\\
 & =2\|u-q_{1}.v\|_{TV}
\end{align*}
where the last equality is obtained exactly as above with $2\|u.q_{2}-q_{1}.v\|_{TV}$.

\p We deduce that $\min_{q_{1},q_{2}}\left\Vert u.q_{2}-q_{1}.v\right\Vert _{TV}=\min_{q_{1}}\left\Vert u-q_{1}.v\right\Vert _{TV}.$
Inverting the roles of the players, we also have $\min_{q_{1},q_{2}}\left\Vert v.q_{2}-q_{1}.y\right\Vert _{TV}=\min_{q_{1}}\left\Vert v-q_{1}.u\right\Vert _{TV}.$
We conclude that 
\begin{align*}
\dist(u,v) & =\max\{\min_{q_{1},q_{2}}2\left\Vert u.q_{2}-q_{1}.v\right\Vert _{TV};\min_{q_{1},q_{2}}2\left\Vert v.q_{2}-q_{1}.y\right\Vert _{TV}\}\\
 & =\max\{\min_{q_{1}}2\left\Vert u-q_{1}.v\right\Vert _{TV};\min_{q_{1}}2\left\Vert v-q_{1}.u\right\Vert _{TV}\}=\dist_{1}(u,v),
\end{align*}
where the last equality follows from (\ref{eeq:d_1_marg}) together
with the fact that $\marg_{K\times D}u=\marg_{K\times D}v$.

\subsubsection{Extension of Proposition \ref{prop. Info substitutes}}

Proposition \ref{prop. Info substitutes} holds using the same ideas and a few technical
adaptations described below. \p Let $C=C_{1}=C_{2}=D=[0,1]$. 
\[
\begin{split} & u\in\Delta\left(K\times(C\times C_{1}\times C_{2})\times D\right)\text{ and }v=\marg_{K\times(C\times C_{1})\times D}u,\\
 & u^{\prime}=\marg_{K\times(C\times C_{2})\times D}u,\text{ and }v^{\prime}=\marg_{K\times C\times D}u.
\end{split}
\]

\begin{prop}
Suppose that, under $u$, $c_{1}$ is conditionally independent from
$\left(c,c_{2},d\right)$ given $k$. Then, $\dist\left(u^{\prime},v^{\prime}\right)\geq\dist\left(u,v\right)$. 
\end{prop}
Recall that Theorem \ref{thm1} holds for more general signal spaces. Because
$u\succeq v$, 
\[
\dist\left(u,v\right)=\min_{q_{2}}\min_{q_{1}}2\left\Vert u.q_{2}-q_{1}.v\right\Vert _{TV}\leq\min_{q_{2}}\min_{q_{1}:C\rightarrow\Delta\left(C\times C_{2}\right)}2\left\Vert u.q_{2}-\hat{q}_{1}.v\right\Vert _{TV},
\]
where $q_{1}$ ranges through the set of garblings from $C\times C_{1}$
to $\Delta(C\times C_{1}\times C_{2})$, $q_{2}$ ranges through the
set of garblings from $D$ to $\Delta(D)$ and where in the right-hand
side of the inequality, we use a restricted set of player 1's garblings.
Precisely, given $q_{1}:C\rightarrow\Delta\left(C\times C_{2}\right)$,
we define the garbling $\hat{q}_{1}$ by the relation 
\[
\int_{C'\times C'_{1}\times C'_{2}}fd\hat{q}_{1}(c,c_{1})=\int_{C'\times C'_{2}}f(c',c_{1},c'_{2})dq_{1}(c',c'_{2}|c)
\]
where $f$ is an arbitrary bounded measurable function. Further, for
any such $q_{1}$, an arbitrary garbling $q_{2}$ and an arbitrary
measurable function $f$ bounded by $1$, we have 
\[
\begin{split} & \int fd(u.q_{2}-\hat{q}_{1}.v)\\
 & =\int_{K\times C\times C_{1}\times C_{2}\times D}\left[\int_{D'}f(k,c,c_{1},c_{2},d')dq_{2}(d'|d)-\int_{C'\times C'_{2}}f(k,c',c_{1},c'_{2},d)dq_{1}(c',c'_{2}|c)\right]du(k,c,c_{1},c_{2},d)\\
 & =\!\int_{K}\!\int_{C\times C_{2}\times D}\!\int_{C_{1}}\!\!\left[\!\int_{D'}f(k,c,c_{1},c_{2},d')dq_{2}(d'|d)-\!\int_{C'\times C'_{2}}f(k,c',c_{1},c'_{2},d)dq_{1}(c',c'_{2}|c)\right]\!\!du(c_{1}|k)du(c,c_{2},d|k)du(k)\\
 & =\int_{K}\int_{C\times C_{2}\times D}\left[\int_{D'}h(k,c,c_{2},d')dq_{2}(d'|d)-\int_{C'\times C'_{2}}h(k,c',c'_{2},d)dq_{1}(c',c'_{2}|c)\right]du(c,c_{2},d|k)du(k)\\
 & =\int_{K\times C\times C_{2}\times D}hd(u^{\prime}.q_{2}-q_{1}.v{}^{\prime})
\end{split}
\]
where we used the conditional independence assumption, and where the
function $h$ is a measurable function bounded by $1$ and defined
by 
\[
h(k,c,c_{2},d)=\int_{C_{1}}f(k,c,c_{1},c_{2},d)du(c_{1}|k).
\]
We deduce that 
\[
\min_{q_{2}}\min_{q_{1}:C\rightarrow\Delta\left(C\times C_{2}\right)}2\left\Vert u.q_{2}-\hat{q}_{1}.v\right\Vert _{TV}\leq2\left\Vert u^{\prime}.q_{2}-q_{1}.v{}^{\prime}\right\Vert _{TV}.
\]
Hence 
\[
\dist\left(u,v\right)\leq\min_{q_{2}}\min_{q_{1}:C\rightarrow\Delta\left(C\times C_{2}\right)}2\left\Vert u^{\prime}.q_{2}-q_{1}.v{}^{\prime}\right\Vert _{TV}=\dist\left(u^{\prime},v{}^{\prime}\right),
\]
and this concludes the proof.

\subsubsection{Extension of Proposition \ref{prop: joint info}}

Proposition \ref{prop: joint info} holds by adapting the definitions. \p Precisely,
we define $\varepsilon$ conditional independence using the total
variation norm. Consider a distribution $\mu\in\Delta\left(X\times Y\times Z\right)$
over compact metric spaces. We say that the random variables $x$
and $y$ are $\varepsilon$-conditionally independent given $z$ under
$\mu$ if 
\[
\int_{Z}\int_{X\times Y}2\left\Vert \mu\left(x,y|z\right)-\mu\left(x|z\right)\otimes\mu\left(y|z\right)\right\Vert _{TV}d\mu(z)\leq\varepsilon.
\]
Let $C=C_{1}=D=D_{1}=[0,1]$. Let $u\in\Delta\left(K\times(C\times C_{1})\times(D\times D_{1})\right)$
and $v=\text{marg}_{K\times C\times D}u$. 
\begin{prop}
Suppose that $d_{1}$ is $\varepsilon$-conditionally independent
from $\left(k,c\right)$ given $d$, and $c_{1}$ is $\varepsilon$-conditionally
independent from $\left(k,d\right)$ given $c$. Then, $\dist\left(u,v\right)\leq\varepsilon.$ 
\end{prop}
It is enough to show that if $c_{1}$ is $\varepsilon$-conditionally
independent from $\left(k,d\right)$ given $c$, then 
\[
\sup_{g\in\CG}\val\left(u,g\right)-\val\left(v,g\right)\leq\varepsilon.
\]
For this, let $q_{2}:D\times D_{1}\rightarrow D$ be defined as $q_{2}\left(d,d_{1}\right)=\delta_{d}$
and let $q_{1}:C\rightarrow C\times C_{1}$ be defined as $q_{1}\left(c,c_{1}|c\right)=u\left(c_{1}|c\right)$.
Then, for any measurable function $f$ bounded by $1$: 
\begin{align*}
\int  & fd(u.q_{2}-q_{1}.v) \\
& =\int_{K\times C\times C_{1}\times D}\left[f(k,c,c_{1},d)-\int_{C'_{1}}f(k,c,c'_{1},d)du(c'_{1}|c)\right]du(k,c,c_{1},d)\\
 & =\int_{C}\left[\int_{K\times C_{1}\times D}f(k,c,c_{1},d)du(k,c_{1},d|c)-\int_{K\times D}\int_{C'_{1}}f(k,c,c'_{1},d)du(c'_{1}|c)du(k,d|c)\right]du(c)\\
 & \leq\int_{C}2\|u(k,c_{1},d|c)-u(c'_{1}|c)\otimes u(k,d|c)\|_{TV}du(c)\\
 & \leq\varepsilon.
\end{align*}
The claim follows from Theorem \ref{thm1}.

 \bibliographystyle{ecca}
\bibliography{zerosum03042020}

\end{document}